\documentclass[11pt]{article}
\usepackage[utf8]{inputenc}
\usepackage[margin=1in]{geometry}
\usepackage[T1]{fontenc} %
\usepackage{graphicx}
\usepackage{mathpazo}
\usepackage{hyperref}
\hypersetup{
  colorlinks = true,  
  urlcolor = {blueGrotto},
  linkcolor = {royalBlue},
  citecolor = {navyBlueP}
} 

\usepackage[
  backref=true,
  backend=biber,
  natbib=true,
  style=alphabetic,
  sorting=alphabeticlabel,
  sortcites=true,
  minbibnames=3,
  maxbibnames=999,
  mincitenames=1,
  maxcitenames=4,
  minalphanames=4,
  maxalphanames=10,
  doi=false, url=false, eprint=false, isbn=false, 
]{biblatex}

\DeclareSortingTemplate{alphabeticlabel}{
  \sort[final]{%
    \field{labelalpha} %
  }
  \sort{%
    \field{year}   %
  }
  \sort{%
    \field{title}  %
  }
}

\AtBeginRefsection{\GenRefcontextData{sorting=ynt}}
\AtEveryCite{\localrefcontext[sorting=ynt]}
\usepackage{setspace}
\usepackage[dvipsnames]{xcolor}
\definecolor{niceRed}{RGB}{190,38,38}
\definecolor{niceYellow}{HTML}{f5b400}
\definecolor{blueGrotto}{HTML}{059DC0}
\definecolor{royalBlue}{HTML}{057DCD}
\definecolor{navyBlue}{HTML}{0B579C}
\definecolor{limeGreen}{HTML}{81B622}
\definecolor{nicePurple}{HTML}{9c27b0}
\definecolor{lightRoyalBlue}{HTML}{def2ff}  
\definecolor{pastelGreen}{HTML}{4CB944}
\definecolor{gold}{HTML}{ffa300}
\definecolor{lightRoyalBlue}{HTML}{def2ff}

\usepackage{booktabs} 
\usepackage{multirow} 
\usepackage{multicol} 

\usepackage{hyperref}
\hypersetup{
  colorlinks = true, 
  urlcolor = {blueGrotto},
  linkcolor = {royalBlue},
  citecolor = {navyBlue}
}

\usepackage{booktabs} 
\usepackage{multicol} 
\usepackage{multirow} 
\usepackage{makecell} 
\usepackage{longtable}

\usepackage{subfigure}
\usepackage{graphicx}
\usepackage{float} 
\usepackage{tikz}
\usepackage{pgfplots}
\pgfplotsset{compat=1.17}
\usepgfplotslibrary{fillbetween}

\usepackage{physics} %
\usepackage{amsmath}
\usepackage{amssymb}
\usepackage{amsthm}
\usepackage{bbm}
\usepackage{blkarray}
\usepackage{mathtools}
\usepackage{nicefrac}
\usepackage{dsfont}
\usepackage{pifont} %
\usepackage{thm-restate} %
\usepackage[capitalise,noabbrev,nameinlink]{cleveref}

\usepackage{typed-checklist}
\usepackage{enumitem} 
\usepackage{tcolorbox}
\usepackage[framemethod=TikZ]{mdframed}
\mdfsetup{%
backgroundcolor=white, %
roundcorner=4pt,
linewidth=1pt}
\usepackage{changepage} %

\usepackage{algorithm2e}

\usepackage{mathrsfs} %
\usepackage{soul} %

\theoremstyle{plain} 
\newtheorem{theorem}{Theorem}[section]
\newtheorem{corollary}[theorem]{Corollary}
\newtheorem{proposition}[theorem]{Proposition}
\newtheorem{lemma}[theorem]{Lemma}

\newtheorem{claim}[theorem]{Claim}

\newtheorem{problem}{Problem}

\newtheorem{inftheorem}{Informal Theorem}

\newtheorem{definition}{Definition}
\newtheorem*{definition*}{Definition}

\theoremstyle{definition} 
\newtheorem{example}[theorem]{Example}
\newtheorem{remark}[theorem]{Remark}

\theoremstyle{remark}

\AfterEndEnvironment{definition}{\noindent\ignorespaces}
\AfterEndEnvironment{assumption}{\noindent\ignorespaces}
\AfterEndEnvironment{lemma}{\noindent\ignorespaces}
\AfterEndEnvironment{theorem}{\noindent\ignorespaces}
\AfterEndEnvironment{proposition}{\noindent\ignorespaces}
\AfterEndEnvironment{fact}{\noindent\ignorespaces}
\AfterEndEnvironment{question}{\noindent\ignorespaces}
\AfterEndEnvironment{corollary}{\noindent\ignorespaces}
\AfterEndEnvironment{model}{\noindent\ignorespaces}
\AfterEndEnvironment{remark}{\noindent\ignorespaces}
\AfterEndEnvironment{proof}{\noindent\ignorespaces}
\AfterEndEnvironment{fact}{\noindent\ignorespaces}
\AfterEndEnvironment{minftheorem}{\noindent\ignorespaces}
\AfterEndEnvironment{inftheorem}{\noindent\ignorespaces}
\AfterEndEnvironment{maintheorem}{\noindent\ignorespaces}
\AfterEndEnvironment{restatable}{\noindent\ignorespaces}
\AfterEndEnvironment{infassumption}{\noindent\ignorespaces}
\AfterEndEnvironment{conjecture}{\noindent\ignorespaces}
\AfterEndEnvironment{claim}{\noindent\ignorespaces}
\AfterEndEnvironment{example}{\noindent\ignorespaces}

\renewcommand{\cref}{\Cref}
\crefname{section}{Section}{Sections}
\crefname{theorem}{Theorem}{Theorems}
\crefname{lemma}{Lemma}{Lemmas}
\crefname{definition}{Definition}{Definitions}
\crefname{conjecture}{Conjecture}{Conjectures}
\crefname{corollary}{Corollary}{Corollaries}
\crefname{construction}{Construction}{Constructions}
\crefname{conjecture}{Conjecture}{Conjectures}
\crefname{claim}{Claim}{Claims}
\crefname{observation}{Observation}{Observations}
\crefname{proposition}{Proposition}{Propositions}
\crefname{fact}{Fact}{Facts}
\crefname{question}{Question}{Questions}
\crefname{problem}{Problem}{Problems}
\crefname{remark}{Remark}{Remarks}
\crefname{example}{Example}{Examples}
\crefname{equation}{Equation}{Equations}
\crefname{appendix}{Appendix}{Appendices}
\crefname{algorithm}{Algorithm}{Algorithms}
\crefname{model}{Model}{Models}
\crefname{figure}{Figure}{Figures}
\crefname{inftheorem}{Informal Theorem}{Informal Theorems}
\crefname{infthm}{Informal Theorem}{Informal Theorems}
\crefname{infassumption}{Informal Assumption}{Informal Assumptions}
\crefname{minftheorem}{Main Informal Theorem}{Main Informal Theorems}
\crefname{maintheorem}{Main Theorem}{Main Theorems}
\crefname{assumption}{Assumption}{Assumptions}
\crefname{case}{Case}{Cases}

\newlist{asmpenum}{enumerate}{1} %
\setlist[asmpenum]{label={\arabic*.},ref=\theassumption.{\arabic*}}
\crefname{asmpenumi}{Assumption}{Assumptions}

\crefname{condition}{Condition}{Conditions}
\crefname{infcondition}{Informal Condition}{Informal Conditions}

\newcommand{\yesnum}{\addtocounter{equation}{1}\tag{\theequation}} 
\makeatletter
\newcommand{\tagnum}[2]{%
    \refstepcounter{equation}%
    \tag{#1) \ (\theequation}%
    \protected@write \@auxout {}{%
        \string \newlabel {#2}{{\theequation}{\thepage}{}{equation.\theequation}{}}%
    }%
}
\makeatother

\newcommand{\Stackrel}[2]{\stackrel{\mathmakebox[\widthof{\ensuremath{#2}}]{#1}}{#2}}

\newcommand{\quadtext}[1]{\quad\text{#1}\quad}
\newcommand{\qquadtext}[1]{\qquad\text{#1}\qquad}
 
\newcommand{\quadand}{\quadtext{and}}
\newcommand{\qquadand}{\qquadtext{and}}

\def\abs#1{\left| #1 \right|}
\def\sabs#1{| #1 |}
\newcommand{\given}{\;\middle|\;}
\newcommand{\sinparen}[1]{(#1)}
\newcommand{\sinbrace}[1]{\{#1\}}
\newcommand{\sinsquare}[1]{[#1]}
\newcommand{\inbrace}[1]{\left\{#1\right\}}

\newcommand{\inparen}[1]{\left(#1\right)}
\newcommand{\insquare}[1]{\left[#1\right]}

\newcommand{\snorm}[1]{\ensuremath{\| #1 \|}}
\let\norm\relax
\newcommand{\norm}[1]{\ensuremath{\left\lVert #1 \right\rVert}}

\newcommand{\N}{\mathbb{N}}
\newcommand{\R}{\mathbb{R}}

\newcommand{\evE}{\ensuremath{\mathscr{E}}}

\newcommand{\E}{\operatornamewithlimits{\mathbb{E}}} 
\newcommand{\Ex}{\E}
\newcommand{\Exp}{\Ex}

\newcommand\ind{\mathds{1}}

\newcommand{\tv}[2]{\operatorname{d}_{\mathsf{TV}}\inparen{#1,#2}}

\newcommand{\zo}{\ensuremath{\inbrace{0, 1}}}

\newcommand{\sfrac}[2]{{#1/#2}} 
\newcommand{\nfrac}[2]{\nicefrac{#1}{#2}}

\newcommand{\vol}{\textrm{\rm vol}}
\newcommand{\supp}{\operatorname{supp}}

\newcommand{\iid}{{i.i.d.}}

\newcommand{\eps}{\varepsilon}
\renewcommand{\epsilon}{\varepsilon}
\makeatletter
\newcommand*{\tran}{{\mathpalette\@tran{}}}
\newcommand*{\@tran}[2]{\raisebox{\depth}{$\m@th#1\intercal$}}
\makeatother

\mathchardef\NABLA"272
\newcommand*{\Nabla}{\boldsymbol\NABLA}
\let\nabla\Nabla

\renewcommand{\hat}{\widehat}
\newcommand{\wh}[1]{\widehat{#1}}

\renewcommand{\bar}{\overline}

\newcommand{\wt}[1]{\widetilde{#1}}

\newcommand{\customcal}[1]{\euscr{#1}} %

\newcommand{\cC}{\customcal{C}}
\newcommand{\cD}{\customcal{D}}

\newcommand{\cF}{\customcal{F}}

\newcommand{\cN}{\customcal{N}}

\newcommand{\cP}{\customcal{P}} 
\newcommand{\cQ}{\customcal{Q}}

\usepackage{mdframed}

\newmdenv[
    backgroundcolor=lightgray!10, %
    roundcorner=5pt,            %
    linecolor=black,             %
    linewidth=1pt,               %
    innertopmargin=5pt,         %
    innerbottommargin=0pt,      %
    innerleftmargin=10pt,        %
    innerrightmargin=10pt,       %
    skipabove=5pt,              %
    skipbelow=0pt               %
]{curvybox}

\DeclareMathAlphabet{\mathdutchcal}{U}{dutchcal}{m}{n}
\SetMathAlphabet{\mathdutchcal}{bold}{U}{dutchcal}{b}{n}
\DeclareMathAlphabet{\mathdutchbcal}{U}{dutchcal}{b}{n}

\DeclareMathAlphabet\urwscr{U}{urwchancal}{b}{n}%
\DeclareMathAlphabet\rsfscr{U}{rsfso}{m}{n}
\DeclareMathAlphabet\euscr{U}{eus}{m}{n}
\DeclareFontEncoding{LS2}{}{}
\DeclareFontSubstitution{LS2}{stix}{m}{n}
\DeclareMathAlphabet\stixcal{LS2}{stixcal}{m} {n}

\newcommand{\eat}[1]{}

\newcommand{\hypo}[1]{\mathbbmss{#1}}

\newcommand{\hyD}{\hypo{D}}
\newcommand{\hyE}{\hypo{E}}

\newcommand{\hyH}{\hypo{H}}

\newcommand{\hyP}{\hypo{P}}

\renewcommand{\d}{{\rm d}}

\newcommand{\dataset}{\mathscr{C}}

\newcommand{\hyPou}{\hyP_{\rm OU}}
\newcommand{\hyPoverlap}{\hyP_{\rm O}} %
\newcommand{\hyPunconf}{\hyP_{\rm U}}

\newcommand{\hyDall}{\hyD_{\rm all}}

\newcommand{\hyDpoly}{\hyD_{\rm poly}}

\def \Be {\mathrm{Be}}
\def \fatShatDim {\mathrm{fat}_{\gamma}}

\let\colt\undefined %

\ifdefined\colt

\else

\fi

\renewcommand{\paragraph}[1]{\medskip \noindent\textbf{#1}~}
\newcommand{\paragraphit}[1]{\medskip \noindent\textit{#1}~}

\newcommand{\ie}{\textit{i.e.}}
\newcommand{\eg}{\textit{e.g.}} 

\usepackage{array}
\newcolumntype{L}[1]{>{\raggedright\let\newline\\\arraybackslash\hspace{0pt}}m{#1}}
\newcolumntype{C}[1]{>{\centering\let\newline\\\arraybackslash\hspace{0pt}}m{#1}}
\newcolumntype{R}[1]{>{\raggedleft\let\newline\\\arraybackslash\hspace{0pt}}m{#1}}

\title{ 
    What Makes Treatment Effects Identifiable?\\ 
    Characterizations and Estimators Beyond Unconfoundedness
}      

\author{
  \begin{tabular}{C{4.8cm}C{4.8cm}C{4.8cm}}
    {\bf Yang Cai} & {\bf Alkis Kalavasis} & {\bf Katerina Mamali}\\[1mm]
    Yale University & Yale University & Yale University\\[0mm]
    \mbox{\small\texttt{\href{mailto:yang.cai@yale.edu}{yang.cai@yale.edu}}} & 
    \mbox{\small\texttt{\href{mailto:alkis.kalavasis@yale.edu}{alkis.kalavasis@yale.edu}}} & 
    \mbox{\small\texttt{\href{mailto:katerina.mamali@yale.edu}{katerina.mamali@yale.edu}}}\\[8mm]
  \end{tabular}
    \\[6mm]
  \begin{tabular}{C{5.5cm}C{5.5cm}}
    {\bf Anay Mehrotra} & {\bf Manolis Zampetakis}\\[1mm]
    Yale University & Yale University\\[0mm]
    {\small\texttt{\href{mailto:anaymehrotra1@gmail.com}{anaymehrotra1@gmail.com}}} &
    \smash{\small\texttt{\href{mailto:manolis.zampetakis@yale.edu}{manolis.zampetakis@yale.edu}}}
  \end{tabular}
}
\date{}

\newcommand\blfootnote[1]{%
  \begingroup
  \renewcommand\thefootnote{}\footnote{#1}%
  \addtocounter{footnote}{-1}%
  \endgroup
}

\begin{document}

\maketitle

\begin{abstract}
    Most of the widely used estimators of the \emph{average treatment effect} (ATE) in causal inference rely on the assumptions of \emph{unconfoundedness} and \emph{overlap}. Unconfoundedness requires that the observed covariates account for all correlations between the outcome and treatment. Overlap requires the existence of randomness in treatment decisions for all individuals. Nevertheless, many types of studies frequently violate unconfoundedness or overlap, for instance, observational studies with deterministic treatment decisions – popularly known as Regression Discontinuity designs – violate overlap.
 
    In this paper, we initiate the study of general conditions that enable the \emph{identification} of the average treatment effect, extending beyond unconfoundedness and overlap. In particular, following the paradigm of statistical learning theory, we provide an interpretable condition that is sufficient and necessary for the identification of ATE. 
    Moreover, this condition also characterizes the identification of the \emph{average treatment effect on the treated} (ATT) and can be used to characterize other treatment effects as well. 
    To illustrate the utility of our condition, we present several well-studied scenarios where our condition is satisfied and, hence, we prove that ATE can be identified in regimes that prior works could not capture. For example, under mild assumptions on the data distributions, this holds for the models proposed by \citet{tan2006distributional} and \citet{rosenbaum2002observational}, and the Regression Discontinuity design model introduced by \citet{thistlethwaite1960regressionDiscontinuity}. For each of these scenarios, we also show that, under natural additional assumptions, ATE can be estimated from finite samples.

    We believe these findings open new avenues for bridging learning-theoretic insights and causal inference methodologies, particularly in observational studies with complex treatment mechanisms.

    \blfootnote{Accepted for presentation{, as an extended abstract,} at the 38th Conference on Learning Theory (COLT) 2025}
\end{abstract}

\newpage

\newpage

{
    \linespread{1}
    \tableofcontents
}

\newpage

\section{Introduction}\label{sec:intro}
    Understanding cause and effect is a central goal in science and decision-making. Across disciplines, we ask: What is the effect of a new drug on disease rates? How does a policy impact growth? Is technology driving economic growth? \emph{Causal inference} tackles such questions by disentangling correlation from causation. Unlike statistical learning, which predicts outcomes from data, causal inference estimates the effects of interventions that alter the data-generating process.
    
    A fundamental challenge in causal inference is that we can never observe both potential outcomes for the same individual. For example, if a patient takes a medication and recovers, we do not know whether the patient would have recovered without it. This \emph{fundamental problem of causal inference} implies that causal effects must be inferred under certain assumptions \cite{holland1986statistics}. 

    To formalize this challenge, we consider the widely-used \emph{potential outcomes model} introduced by \citet{neyman1990applications} (originally published in 1923) and later formalized by \citet{rubin1974estimating}; see also \citet*{hernan2023causal,rosenbaum2002observational,chernozhukov2024appliedcausalinferencepowered}. 
    Here, for a unit with \textit{covariates} $X \in \R^d$, $Y(1)$ and $Y(0)$ denote potential outcomes under treatment and control, respectively. 
    Since only the outcome $Y(T)$ corresponding to the assigned treatment $T$ is observed, certain assumptions are needed to estimate the \textit{average treatment effect} (ATE), defined as $\tau \coloneqq \Exp[Y(1) - Y(0)]$, where $(T, Y(1), Y(0))$ are random variables whose distribution may depend on $X$. This framework underpins many modern causal inference methods -- both practical and theoretical -- and can capture many treatment effects, apart from $\tau,$ such as the \emph{average treatment effect on the treated} (ATT), defined as $\gamma\coloneqq \Ex[Y(1)-Y(0)\mid T{=}1]$. 
    Two fundamental questions under this framework, studied since \citet{cochran1965observationalStudies, rubin1974estimating, rubin1978randomization, heckman1979SelectionBias}, are as follows:
    \begin{enumerate}[leftmargin=13pt,itemsep=0pt]
        \item[$\triangleright$] \emph{Identification}: Given infinite samples of the form $(X,T,Y(T))$, can we \emph{identify} {treatment effects}?
        \item[$\triangleright$] \emph{Estimation}: Given $n$ samples   $(X,T,Y(T))$, can we \emph{estimate} {treatment effects} {up to} error $\eps(n)$?
    \end{enumerate}
    \noindent Due to the missingness in data (explained above), even the identification problem is unsolvable without making structural assumptions on the distribution of  $(X, T, Y(T))$, which is a censored version of the (complete) data distribution $(X,T,Y(1),Y(0))$.
    The earliest and most widely used such assumptions are \textit{unconfoundedness} and \textit{overlap}.
    \begin{enumerate}[leftmargin=13pt,itemsep=0pt]
      \item[$\triangleright$] \textit{Unconfoundedness} presumes that after conditioning on the value of the covariate $X$, the treatment random variable $T$ is independent of the outcomes $Y(1)$ and $Y(0)$, \ie{}, $T \bot (Y(0), Y(1)) \mid X$.
      \item[$\triangleright$] \textit{Overlap} requires that the probability of being assigned treatment conditional on the covariate $X$, \ie{}, $\Pr[T{=}1\mid X{=}x]$, is a quantity strictly between 0 and 1 for each covariate $x$.
    \end{enumerate} 
    Unconfoundedness (a.k.a.,\ ignorability, conditional exogeneity, conditional independence, selection on observables) and overlap (a.k.a.,\ positivity and common support) are essential for unbiased estimation of the average treatment effect and are widely studied across {Statistics (\eg{}, \cite{rosenbaum2002observational,hernan2023causal,rubin1974estimating,rubin1977regressionDiscontinuity,rubin1978randomization,rosenbaum1983central}) and many other disciplines, including} 
    Medicine (\eg{}, \cite{rosenbaum1983central}), 
    Economics (\eg{}, \cite{athey2017CausalityReview,dehejia1998causal,dehejia2002propensity,abadie2006large,abadie2016matching}), 
    Political Science (\eg{}, \cite{brunell2004turnout,sekhon2004quality,ho2007matching}), 
    Sociology (\eg{}, \cite{morgan2006matching,lee2009estimation,oakes2017methods}), 
    and other fields (\eg{}, \cite{austin2008critical}).
    Despite their wide use across different disciplines, there are fundamental instances where unconfoundedness or overlap are easily violated.

Unconfoundedness is often violated in \textit{observational studies}, where treatments or exposures are not assigned by the researcher but observed in a natural setting. In a prospective cohort study, for example, individuals are followed over time to assess how exposures influence outcomes. A common violation arises when key confounders are unmeasured. For instance, in studying smoking’s impact on health, omitting socioeconomic status (SES), which affects both smoking habits and health, can bias results, as lower SES correlates with higher smoking rates and poorer health, independent of smoking.

    Overlap is violated when certain covariate values make treatment assignments nearly deterministic. In a marketing study estimating the effect of personalized advertisements on purchases, covariates like demographics, browsing history, and preferences define a high-dimensional feature space. As this space grows, many user profiles either always or never receive the ad, leading to \emph{lack of overlap} \cite{damour2021highDimensional}. Without comparable treated and untreated units, causal inference methods struggle to estimate counterfactual outcomes, yielding unreliable effect estimates.
    
    We refer the reader to \cref{sec:examples:violation} for an in-depth discussion of scenarios demonstrating the fragility of unconfoundedness and overlap.
    Further, while Randomized Controlled Trials (RCTs) can eliminate hidden factors that lead to violation of unconfoundedness or overlap, they are often very expensive and, even unethical, for treatments that can harm individuals.
    Moreover, even RCTs can violate unconfoundedness due to participant non-compliance; see \cref{sec:examples:violation:Unc}.

    \medskip 
    
    These examples lead us to the following question, which we answer. 
    \begin{mdframed}[leftmargin=2.5cm, rightmargin=2.5cm]
            \begin{center}
                \emph{Is identification and estimation of treatment effects possible\\ in any meaningful setting without unconfoundedness or overlap?}
            \end{center}
        \end{mdframed}
        This question is not new and can be traced back to at least the work of \citet{rubin1977regressionDiscontinuity}, who recognized that, without substantial overlap between treatment and control groups, identification of treatment effects necessarily requires additional prior assumptions. 
        To the best of our knowledge, the present work provides the first formal characterization of the precise assumptions required to identify treatment effects in scenarios lacking substantial overlap, unconfoundedness, or both.

    \subsection{Framework}
    \label{sec:framework}
        The main conceptual contribution of this work is a \emph{learning-theoretic} approach that enables a characterization of when identification and estimation of {treatment effects} are possible.
        Before presenting this approach, it is instructive to reconsider how {unconfoundedness and overlap} enable identification {of the simplest and most widely used treatment effect -- the average treatment effect}: 
        Given the observational study $\cD$, which is a distribution over $(X,T,Y(0),Y(1))$, unconfoundedness and overlap put a strong constraint on $\cD$:
        they require that $Y(t) \perp T~|~X{=}x$ for each $t \in \{0,1\}$ and $x \in \R^d$ and that the propensity scores $e(x) = \Pr[T {=} 1 | X{=}x]$ are bounded away from 0 and 1. 
        Under these assumptions, identification and estimation of ATE $\tau = \tau_{\cD}$ is possible given censored samples $(X,T,Y(T))$ due to the following decomposition of $\tau_\cD$ for a fixed $x \in \R^d$ (we integrate over the $x$-marginal to get $\tau_\cD$):
        \[
            \Ex_{Y(0), Y(1)}
            \insquare{Y(1)-Y(0) \mid X{=}x}
            = 
            \Ex_{Y(0), Y(1),T}\insquare{
                    \frac{Y(1)\cdot T}{e(X)}
                -   \frac{Y(0)\cdot (1-T)}{1-e(X)}
                    \given X{=}x}
            \,,        \yesnum\label{eq:decomposition:unconfoundedness}
        \]
        {where we use overlap to divide with $e(X), 1-e(X)$ and  unconfoundedness to obtain the equation $\E[Y(1)\cdot T\mid X]
            =\Ex[Y(1)\mid X]\cdot \Pr[T{=}1\mid X]
            $ and analogously for $Y(0)$.}
        Note that all the quantities appearing in the RHS are identifiable and estimable\footnote{We remark that the problem of estimating the propensity scores $e(x) = \Pr[T{=}1|X{=}x]$ is identical to the classical problem of learning probabilistic concepts \cite{kearns1994pconcept}. We refer the reader to \Cref{appendix:probconcepts} for details.} from the censored distribution $\cC_\cD$~\cite{rubin1978randomization}, which is defined over $(X,T,Y(T))$.
        
        {When} no {constraints are} put on $\cD$, identification of ATE is impossible in general \cite{imbens2015causal}.
        Without unconfoundedness, propensity scores $\Pr[T{=}1\mid X{=}x]$ are not sufficient to identify the distribution of $T$, which can also depend on the outcomes $Y(0)$ and $Y(1)$ (conditioned on $X{=}x$).
        Instead, we can decompose the expression of $\tau_\cD$ for a fixed $x \in \R^d$ as follows:
            \[
                \Ex_{\substack{Y(0),Y(1)}}\insquare{Y(1)-Y(0)\mid X{=}x}
                =
                \Ex_{\substack{Y(0), Y(1), T}}\insquare{
                        \frac{Y(1)\cdot T}{\Pr[T{=}1|X, Y(1)]}
                    -   \frac{Y(0)\cdot (1-T)}{\Pr[T{=}0|X, Y(0)]}
                        \given X{=}x
                }\,.
        \]
        If unconfoundedness holds, then we could recover \eqref{eq:decomposition:unconfoundedness} since then $T$ would not depend on $Y(1), Y(0)$ given $X.$
        {However, unlike} the previous decomposition of \Cref{eq:decomposition:unconfoundedness}, the above equation always holds and crucially utilizes the \emph{generalized propensity scores} $p_t(x,y) = \Pr[T{=}t\mid X{=}x, Y(t) {=}y]$ with $t \in \{0,1\}$.\footnote{Observe that we need some overlap condition to divide by $p_0(\cdot)$ and $p_1(\cdot)$ in the above equation. In our main results, however, we do \emph{not} follow this decomposition and will \emph{not} need such overlap conditions.}
        Unfortunately, these generalized propensity scores, in contrast to the standard propensity scores, are not always identifiable from data. %
        To understand when these are identifiable, we need to consider the \emph{joint distribution of covariates and outcomes} $\cD_{X,Y(t)}$ for each $t \in \{0,1\}$.
    
    To this end, we adopt an approach inspired by statistical learning theory \cite{valiant1984theory,vapnik1999overview,blumer1989learnability,hastie2013elements,AnthonyBartlett1999NNLearning,alon1997scale,lugosi2002pattern,massartNoise2006,vapnik2006estimation,bousquet2003introduction,bousquet2003new}.
    We introduce \emph{concept classes} for the two key quantities derived by the above discussion $p_t$ and $\cD_{X,Y(t)}$ (for each $t \in \{0,1\}$) that will  \emph{place some restrictions} on the observational study $\cD$ {towards understanding which conditions enable} identification and estimation. 
    In the remainder of the paper, we assume that all distributions are continuous and have a density. (All results also extend to discrete domains by replacing densities by probability mass functions.)
    
    We are interested in the structure of two concept classes: the class of generalized propensity scores
    $\hyP \subseteq \{p \colon \R^d \times \R \to [0,1]\}$ and the  class of {covariate-outcome} distributions $\hyD \subseteq \Delta(\R^d \times \R)$. 
    As in classical {statistical} learning theory, having fixed the concept classes, our next step is to restrict the underlying distribution $\cD$ to be \emph{realizable} with respect to the {pair of concept classes $(\hyP, \hyD)$}.
    An observational study is said to be realizable with respect to the concept class pair $(\hyP, \hyD)$ if the generalized propensity scores {$p_0(\cdot),p_1(\cdot)$} induced by $\cD$ belong to $\hyP$ and $\cD_{X,Y(t)} \in \hyD$ for {each} $t \in \{0,1\}$. This learning-theoretic framework is quite expressive.
    For instance, it can capture unconfoundedness  and overlap\footnote{ We refer to overlap as $c$-overlap: for some absolute constant $c \in (0,\nfrac{1}{2})$, $c < {p_0(x,y),p_1(x,y)} < 1-c$.}  by letting $\hyD$ be the set of all distributions over $\R^d\times \R$, {denoted by $\hyDall$}, and restricting $\hyP$ to be the following class  
            \[ 
                \phantom{.}\hyPou(c)\coloneqq \inbrace{p\colon\R^d\times \R\to [0,1]\given p(x,y)=p(x,z) \text{ and } c<p(x,y)<1-c \text{ for each $(x,y,z)$}}.
                \hspace{-2mm}
                \yesnum\label{eq:PinScenarioI}
            \]
    That is, $\cD$ satisfies unconfoundedness and $c$-overlap if and only if it is realizable with respect to the pair of classes $(\hyPou(c), \hyDall)$. 

    Before proceeding to our results, we introduce some further terminology. Given classes $(\hyP, \hyD)$, we are {particularly} interested in the generalized propensity scores in $\hyP$ and covariate-outcome distributions in $\hyD$ that induce a valid observational study. To this end, we say that {a} tuple $(p,\cP) \in \hyP \times \hyD$ is \emph{compatible} with classes $(\hyP,\hyD)$ (henceforth, just compatible) if there exists another tuple $(\wh p, \wh \cP)\in \hyP\times \hyD$ such that setting $(p_0,p_1,\cD_{X,Y(0)},\cD_{X,Y(1)})=(p,\wh{p},\cP,\wh{P})$ (or \textit{equivalently} the re-ordered {assignment} $(p_0,p_1,\cD_{X,Y(0)},\cD_{X,Y(1)})=(\wh{p},p,\wh{P},\cP)$) defines a valid observational study, \ie{}, a valid distribution over $(X,T,Y(0),Y(1))$.

    \subsection{Main Results on Identification}
    \label{sec:results}
        We say that a certain treatment effect ${\eta}_\cD$ is \emph{identifiable} from the censored distribution $\cC_\cD$ when $(\hyP, \hyD)$ satisfy {some} Condition C, if there is a mapping $f$ such that $f(\cC_\cD) = {\eta}_\cD$ for any observational study $\cD$ realizable with respect to $(\hyP, \hyD)$ that satisfy C; in other words, if $\eta_{\cD_1} \neq \eta_{\cD_2}$ then it should be $\cC_{\cD_1} \neq \cC_{\cD_2}$ (see also \cref{prob:main} for a formal definition).
        Having set the stage, we now ask our first main question: 
        \begin{mdframed}[leftmargin=1.5cm, rightmargin=1.5cm]
            \begin{center}
                \emph{Which conditions on $(\hyP, \hyD)$ characterize the identifiability of {treatment effects}?}
            \end{center}
        \end{mdframed} 
    As a first contribution, we identify a condition on the classes $(\hyP, \hyD)$ that will be crucial for the results on the identification of ATE and ATT that proceed.
     \begin{restatable}[{Identifiability} Condition]{condition}{conditionIden}
        \label{cond:iden:informal} 
        \label{cond:iden} 
        \label{infcond1} 
        The concept classes $\inparen{\hyP,\hyD}$ satisfy the Identifiability Condition if for any  tuples $(p,\cP), (q,\cQ) \in {\hyP} \times {\hyD}$ that are compatible with $({\hyP}, {\hyD})$, at least one of the following holds:
    
        \begin{enumerate}[itemsep=-1pt,leftmargin=17.5pt]
            \item \textbf{(Equivalence of Outcome Expectations)} $\E_{(x,y) \sim \cP}[y] = \E_{(x,y) \sim \cQ}[y]$
            \item \textbf{(Distinction of Covariate Marginals)} $\cP_X \neq \cQ_X$ 
            \item \textbf{(Distinction under Censoring)} $\exists (x,y)\in \supp(\cP_X)\times \R$, such that, $p(x,y) \cP(x, y) \neq q(x,y) \cQ(x,y)$.
        \end{enumerate}
    \end{restatable}
     Observe that in \cref{cond:iden} we only focus on tuples that are compatible. This is due to the fact that incompatible tuples of $\hyP \times \hyD$ cannot be part of a realization of any valid observational study and hence properties of incompatible tuples are not relevant to our characterizations.%
    
     To gain some intuition for \cref{cond:iden}, consider two observational studies $\cD_1$ and $\cD_2$ which correspond to the pairs $(p, \cP)$ and $(q,\cQ)$ respectively, where $\cP$ and $\cQ$ are distributions of $(X,Y(1)).$ 
        Assume that the true observational study $\cD$ is either $\cD_1$ or $\cD_2$.
        Given the censored distribution $\cC_\cD$, we want to identify $\E_{\cD}[Y(1)].$
       First, suppose that the tuples $(p, \cP), (q, \cQ)$ satisfy Requirement 1 in \cref{cond:iden}. Then we are done since we only care about the expected outcomes $\E_{\cD}[Y(1)] = \E_{(x,y) \sim \cP}[y] = \E_{(x,y) \sim \cQ}[y]$ which are the same under both distributions.
     Next, let us assume that Requirement 1 is violated and, hence, the expected treatment outcome is different between the null and the alternative hypothesis. In this case, if Requirement 2 is satisfied, then we can distinguish $\cP$ and $\cQ$ from $\cC_{\cD}$ (by comparing $\cP_X$ and $\cQ_X$ to the covariate marginal of $\cC_{\cD}$) and, hence, distinguish between $\cD_1$ and $\cD_2$.
    Finally, if both Requirements 1 and 2 fail but Requirement 3 holds, then $p(x,y) \cP(x,y)$ is proportional to the density of $(X,T,Y(1))$ in the \emph{censored} distribution on each point $(x,y)$.
    Using this, we can again distinguish between the null and the alternative hypothesis.
    (Notice that, {in both the second and third steps,} we can distinguish between distributions that differ on a measure-zero set {since we allow the identification algorithms to be a function of the whole density. If one does not allow this, then one needs to consider the ``almost everywhere'' analogue of \cref{cond:iden}.})

    Our first result states that \cref{cond:iden} fully characterizes the ATE {in} any observational study $\cD$ realizable with respect to $(\hyP, \hyD)$.
    
    \begin{theorem}
    [Identification of ATE]
    \label{infthm:Suff}    
    The average treatment effect $\tau_\cD$ is identifiable from the censored distribution $\cC_\cD$ for any observational study $\cD$ realizable with respect to $(\hyP, \hyD)$ if and only if 
    $(\hyP, \hyD)$ satisfy \cref{cond:iden}. 
    \end{theorem}
    Interestingly, we show that \cref{cond:iden} also characterizes the identifiability of the average treatment effect on the treated (ATT), \ie{}, $\gamma_\cD\coloneqq \Ex[Y(1)-Y(0)|T{=}1]$. ({When talking about ATT, to simplify the results and exposition, we assume that $\Pr[T{=}1] > 0$}.)
    \begin{theorem}
    [Identification of ATT]
    \label{infthm1}
        The average treatment effect on the treated $\gamma_{\cD}$ is identifiable from the censored distribution $\cC_{\cD}$ for any observational study $\cD$ realizable with respect to $(\hyP, \hyD)$ (which  has $\Pr_\cD[T{=}1]>0]$) if and only if $(\hyP, \hyD)$ satisfy  \Cref{infcond1}. 
    \end{theorem} 
   Thus, the above two results, together, imply that {identifiability of} ATT and ATE is characterized by the same \Cref{cond:iden}; up to the mild assumption that $\Pr[T{=}1]>0$ which holds in any practical scenario where ATT identification is meaningful.
    
     \paragraph{Discussion.} 
        The above collection of results adds to classical {identifiability} conditions in Statistics (\eg{}, \cite{everitt2013finite,teicher1963identifiability}), Statistical Learning Theory (\eg{}, \cite{angluin1980inductive,angluin1988identifying}\footnote{The characterizing condition in language identification concerns \emph{pairs} of languages \cite{angluin1980inductive}. This is also the case in our setting (see \cref{cond:iden:informal}). Intuitively, this is expected since identification in both problems requires being able to distinguish between pairs of task instances that have distinct ''identities.''}), and Econometrics 
     (\eg{}, \cite{manski1990nonparametric,athey2002identification}).
     To the best of our knowledge, these are the first tight characterizations of when ATE and ATT identification is possible in observational studies. For an overview of the proofs, see the technical overview in \cref{sec:overview}. While we focus on the average treatment effect and the average treatment effect on the treated, the proposed concept class-based framework is flexible and allows us to characterize when other types of treatment effects are identifiable; see \Cref{sec:iden:extension} for an application to the heterogeneous treatment effect.

    \subsection{Applications and Estimation of ATE}
        \label{sec:applicationsAndEstimation}
         For \Cref{infcond1} to be useful {given the} other existing conditions (such as unconfoundedness and overlap), it needs to capture interesting examples \emph{not captured by existing conditions}.
        In what follows, we revisit several well-studied scenarios in causal inference or their generalizations and, for each scenario, provide identification results based on \Cref{infthm:Suff} -- in the process -- obtaining several novel identification results.
        Finally, we also give finite sample complexity guarantees for each of these scenarios.

        \paragraph{Scenario~I: Unconfoundedness and Overlap.} 
            At the end of \Cref{sec:framework}, we mentioned that our framework can capture unconfoundedness and overlap.
            Identification in this scenario is standard and can also be deduced using \cref{infthm:Suff}; see \Cref{section:UO}.
            Estimation in this setting is also standard \cite{imbens2015causal} and we discuss how our framework captures it in \Cref{sec:est:unconfoundedness-overlap}.

        \paragraph{Scenario~II: Overlap without Unconfoundedness.} 
            Next, we consider observational studies $\cD$ which satisfy $c$-overlap for some $c\in (0,\nfrac{1}{2})$ but may not satisfy unconfoundedness. We are going to use our framework to characterize the subset of these studies $\cD$ for which ATE is identifiable.
            Since overlap holds with some parameter $c \in (0,\nfrac{1}{2})$, it restricts the concept class $\hyP$ to be $\hyPoverlap(c)$ where $c < p(x,y) < 1-c$ for any $(x,y)$ and $p \in \hyPoverlap(c)$. %
            This case generalizes several models studied in the causal inference literature \cite{tan2006distributional,rosenbaum2002observational,rosenbaum1987sensitivity,kallus2021minimax}; see the discussion after \cref{infthm:2}.  
            Under this scenario, we can ask: which conditions should the covariate-outcome distributions $\hyD$ satisfy for $\tau$ to be identifiable, \ie{}, for which observational studies realizable by $(\hyPoverlap(c), \hyD)$ is the ATE identifiable? 
            Our result is the following. 

        \begin{inftheorem}
        [Informal, see \Cref{lem:iden:overlap}]
        \label{infthm:2}
        \label{infthm:overlap}
 Assume that for any pair $\cP, \cQ \in \hyD$, either Item 1 or 2 of \Cref{infcond1} hold or there exist $x \in \supp(\cP_X)$, $y \in \R$ such that $\cP(x,y) \notin (\frac{c}{1-c}, \frac{1-c}{c}) \cdot \cQ(x,y)$. Then
        $\tau_{\cD}$ is identifiable from the censored distribution  $\cC_{\cD}$ for any observational study $\cD$ realizable with respect to $(\hyPoverlap{(c)}, \hyD)$. Moreover, this condition is necessary.
        \end{inftheorem}
        The above condition for identification is quite similar to \Cref{infcond1} and is satisfied
        by setting the outcomes marginal of $\cP \in \hyD$ to be, \eg{}, Gaussian, Pareto, or Laplace, and letting the $x$-marginal $\cP_X$ be unrestricted. 
        This captures important practical models where the outcomes are modeled as a generalized linear model with Gaussian noise \cite{rosenbaum2002observational,chernozhukov2024appliedcausalinferencepowered}.

        \begin{figure}[!ht]
            \centering
            \subfigure[]{
                {\includegraphics[width=0.45\linewidth,clip,trim={2.5cm 0.15cm 3cm 0cm}]{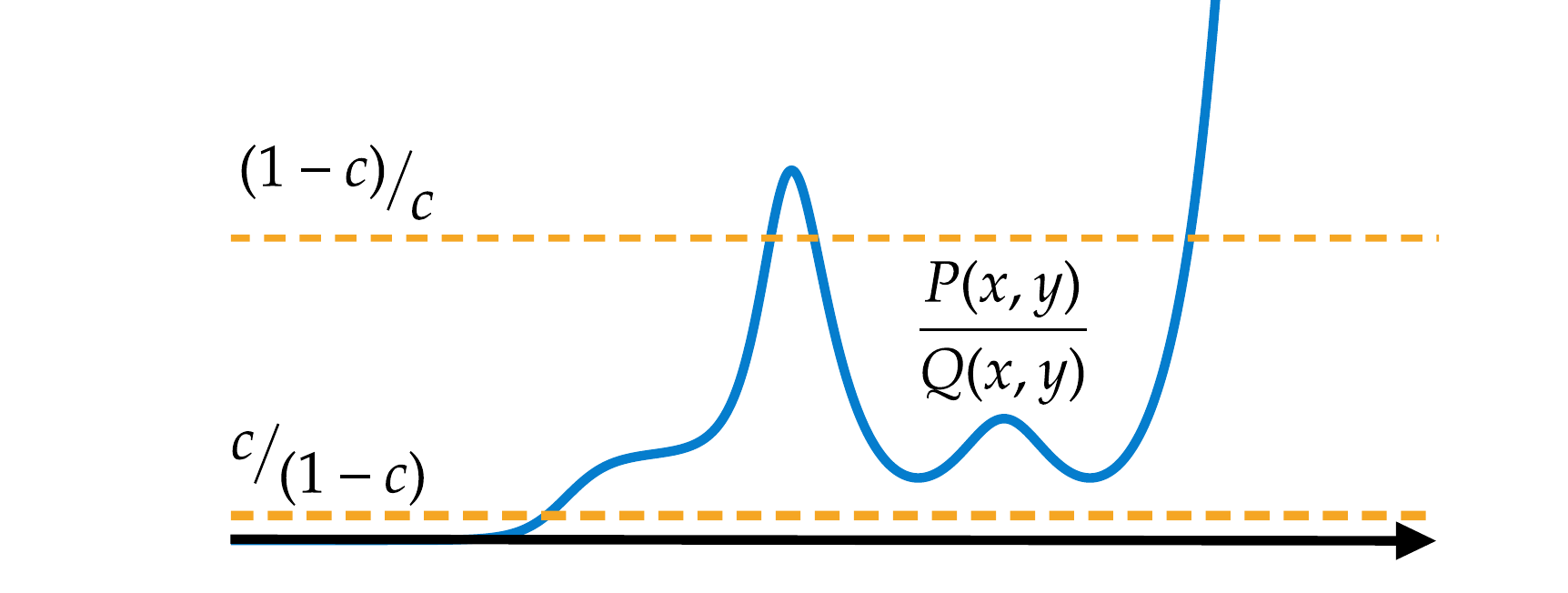}}
            }
            \subfigure[]{
                {\includegraphics[width=0.475\linewidth,clip,trim={2.5cm 0cm 3.25cm 0cm}]{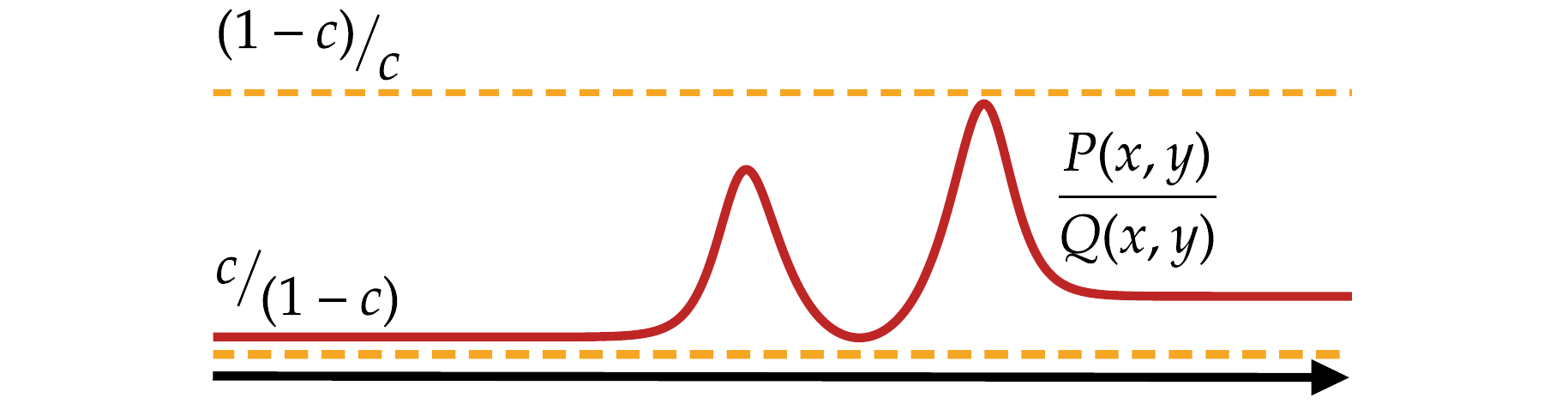}}
            }
            \caption{Illustration of identifiable and non-identifiable instances in Scenario II: 
                The left plot corresponds to an instance which is identifiable since {there are pairs $(x,y)$ where the density ratio $\cP(x,y)/\cQ(x,y)$ lies outside the interval $(\frac{c}{1-c}, \frac{1-c}{c})$; recall that, in Scenario II, the ratio of any two generalized propensity scores always lies in the interval $(\frac{c}{1-c}, \frac{1-c}{c})$.} 
                {The right plot illustrates a non-identifiable instance; to be precise, one also needs to check that neither Item 1 nor Item 2 of \cref{cond:iden} holds in this case.}
                }
            \label{fig:ScenarioII}
        \end{figure}

        \paragraphit{Connections to Prior Work.} 
            Since we do not require unconfoundedness in any form, the requirements {on the generalized propensity score class $\hyPoverlap{}(c)$,} in this scenario, are very mild and are already satisfied by most existing frameworks that relax unconfoundedness while retaining overlap.
           {The restriction on the propensity score class $\hyPoverlap{}(c)$} relaxes \citet{tan2006distributional}'s model and \citet{rosenbaum2002observational}'s odds-ratio model, which are widely used in the literature on sensitivity analysis; see \citet{kallus2021minimax,rosenbaum2002observational} and the references therein.
            Both of these models roughly speaking restrict the range of generalized propensity scores $p_0(x,y),p_1(x,y)$ for the same covariate $x$, while already assuming overlap; see \cref{sec:iden:overlap} for a detailed discussion.
            The range of the propensity scores in Tan's and Rosenbaum's models is parameterized by certain constants $\Lambda,\Gamma\geq 1$ respectively, where $\Lambda=\Gamma=1$ corresponds to unconfoundedness, and the extent of violation of unconfoundedness increases with $\Lambda$ and $\Gamma$.
            The parameter $c$ relates to $\Lambda$ and $\Gamma$ as $\Lambda,\Gamma=O\inparen{\nfrac{(1-c)^2}{c^2}}>1$. 
            As \citet{tan2006distributional,rosenbaum2002observational} note, when $\Lambda,\Gamma>1$, without distributional assumptions, $\tau$ can only be identified up to $O(\Lambda)$ and $O(\Gamma)$ factors respectively. 
            Hence, from earlier results, it is not clear which distribution classes $\hyD$ enable the identification of $\tau$; this is answered by \cref{infthm:2}.

        \paragraphit{Finite-Sample Complexity.} 
            Given the above characterization of when the identification of ATE is possible when only overlap holds, one can ask for finite sample estimation.
            We complement the above result with the following sample complexity guarantee.

        \begin{inftheorem}
        [Informal, see \Cref{lem:est:overlap}]
        \label{infthm:SampleComplexity1}
        Under a robust version of the condition in \Cref{infthm:2} with mass function $M(\cdot)$ and $c$-overlap (see \cref{cond:Over:estimation}) and mild smoothness conditions on $\hyD$, there is an algorithm that, 
                given $n$ \iid{} samples from the censored distribution $\cC_\cD$ for any $\cD$ realizable by $\inparen{\hyPoverlap{(c)},\hyD}$,
               and  $\eps,\delta \in (0,1)$,
            outputs an estimate $\hat{\tau}$ such that $\abs{\hat{\tau}-\tau_{\cD}}\leq \eps$ with probability $1-\delta$.
            The number of samples is ${\wt{O}\inparen{\nfrac{1}{M(\eps)}^2}} \cdot \log(\nfrac{1}{\delta}) \cdot  \mathrm{fat}_{O(\eps)}(\hyPoverlap(c)) \cdot \log N_\eps(\hyD)$.
        \end{inftheorem}
        The sample complexity depends on the fat-shattering dimension \cite*{alon1997scale,talagrand2003vc} of the class $\hyP = \hyPoverlap{(c)}$ and the covering number $\log N_\eps$ of the class of distributions $\hyD$. 
        Moreover, the mass function $M(\cdot)$ appearing in the sample complexity depends on the class of distributions studied (for illustrations, we refer to \cref{lem:est:overlap}).
        To the best of our knowledge, this result is the first sample complexity result for such a general setting. 
         For further details, we refer to \Cref{sec:est:overlap}.

        \begin{remark}
            For our estimation results, we use a ''robust'' version of our {identifiability} condition. 
            This is necessary, to some extent, as estimation is a harder problem than identification.\footnote{Here, we disregard computational considerations, exploring the relation between estimation and identification with computational constraints is an interesting direction.}
            To see this, consider an estimator $E(\cdot)$ of some quantity $\phi_\cD$ (associated with an observational study $\cD$).
            Let the estimator have rate $\eps(\cdot)$, \ie{}, $\abs{\Ex_{s_1,\dots,s_n\sim \cC_\cD}[E(s_1,\dots,s_n)] - \phi_\cD}\leq \eps(n)$; where $\lim_{n\to \infty}\eps(n)=0.$
            Now, one can define an identifier $I(\cdot)$ for $\phi_\cD$ as follows: $I(\cC_\cD)=\lim_{n\to\infty} \Ex_{s_1,\dots,s_n\sim \cC_\cD}[E(s_1,\dots,s_n)].$
        \end{remark}

        \paragraph{Scenario~III: Unconfoundedness without Overlap.} 
            We now consider the setting where overlap may fail but unconfoundedness holds. 
            Without additional assumptions, this allows for degenerate cases in which everyone (or no one) receives the treatment, making identification of the ATE impossible. 
            To rule out such extremes, one can assume that \emph{some} nontrivial subset of covariates satisfies overlap. 
            Concretely, there is a set $S\subseteq\R^d$ with Lebesgue measure  
            $\vol(S)\ge c$ such that for each $(x,y)\in S\times\R$, we have $c < p_0(x,y),\,p_1(x,y) < 1 - c$.\footnote{In general, we do not require the lower bound on $\vol(S)$ to match the lower bound on the propensity scores. We assume them to be the same in the exposition for notational convenience. Our approach still applies when the two lower bounds differ, and \Cref{lem:iden:unconfoundedness,lem:est:unconfoundedness} continue to hold with appropriate modifications.} This is already significantly weaker than the usual $c$-overlap assumption, which demands the previous inequalities \emph{pointwise} for every $(x,y)\in\R^d\times\R$. %
            We relax it further into the notion of \emph{$c$-weak-overlap} by removing the upper bound on $p_0(x, y)$ and $p_1(x, y)$ (defined formally in \cref{sec:scenario:unconfoundedness}). 
           Then, we define the class $\hyP = \hyPunconf(c)$ which captures both unconfoundedness and $c$-weak-overlap; see \cref{sec:scenario:unconfoundedness}.

            Scenarios with unconfoundedness but without full overlap frequently arise in practice.
            Classic examples include regression discontinuity designs \cite{imbens2008regressionDiscontinuity,lee2010regressionDiscontinuity,angrist2009mostlyHarmless} {(see also \citet{cook2008waitingforLife})} and observational studies with extreme propensity scores \cite{crump2009dealing,li2018overlapWeights,khan2024trimming,kalavasis2024cipw}; see further discussion after \cref{infthm:3}. 
            As before, we ask which conditions on $\hyD$ enable identification of ATE, \ie{},
            for which observational studies realizable with respect to $(\hyPunconf{(c)}, \hyD)$, can one identify the ATE?

         \begin{inftheorem}
         [Informal, see \Cref{lem:iden:unconfoundedness}]
        \label{infthm:3}
Assume that for any pair $\cP, \cQ \in \hyD$, either Item 1 or 2 of \Cref{infcond1} hold or there is {no set $S\subseteq\R^d$ with $\vol(S)\geq c$} such that $\cP(x,y) = \cQ(x,y)$ for $(x,y) \in S \times \R$.
        Then
        $\tau_{\cD}$ is identifiable from the censored distribution  $\cC_{\cD}$ for any $\cD$ realizable with respect to $(\hyPunconf{(c)}, \hyD)$. Moreover, this condition is necessary.
        \end{inftheorem} 

         \begin{figure}[!ht]
            \centering
            \subfigure[]{
                {\includegraphics[width=0.45\linewidth,clip,trim={2cm 0.1cm 2.5cm 0.45cm}]{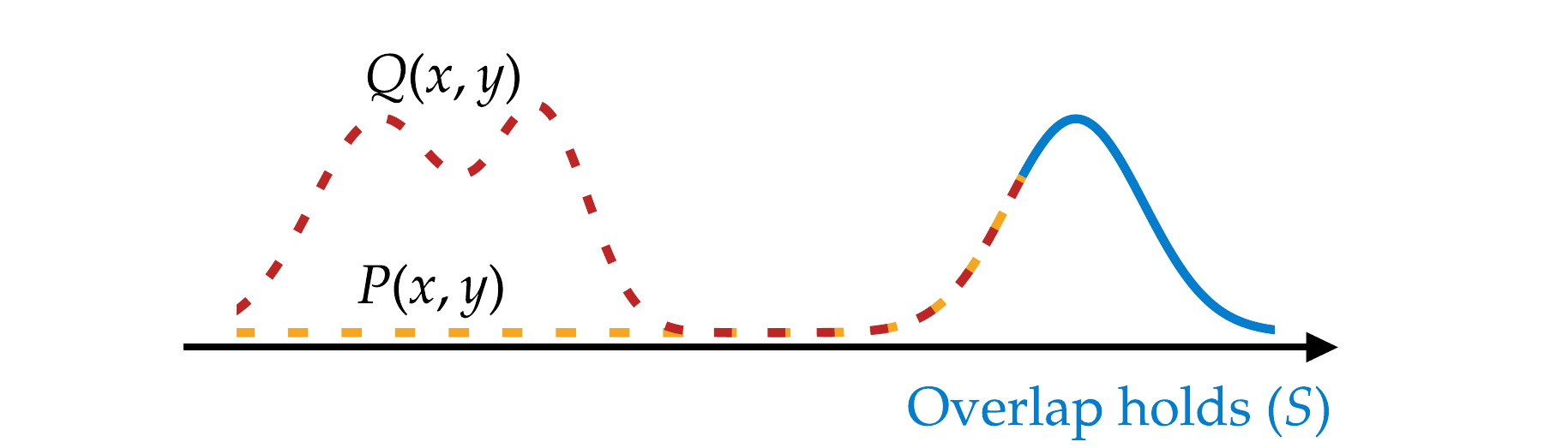}}
            }
            \subfigure[]{
                {\includegraphics[width=0.45\linewidth,clip,trim={2cm 0.1cm 2.5cm 0.45cm}]{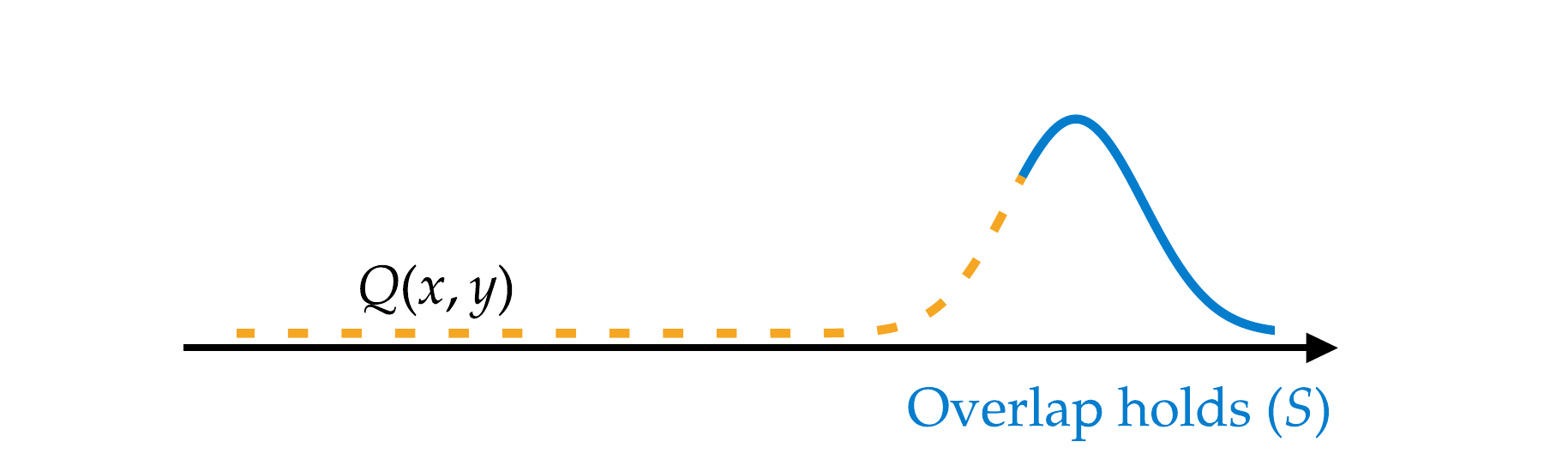}}
            }
            \caption{Illustration of Identifiable and Non-Identifiable Instances in Scenario III: 
                Left Plot: If there are two distributions $\cP$ and $\cQ$ in the class $\hyD$ that have (1) identical densities in the region $S$ where overlap holds  (highlighted in \textcolor{navyBlue}{blue}) but different densities outside this region and (2) $\Ex_\cP[y]\neq \Ex_\cQ[y]$, then ATE is non-identifiable.
                Right Plot: If the density in the overlap region uniquely determines the density $\cQ\in \hyD$ on the whole domain, then ATE is identifiable.
                (By unique, we mean that there is no $\cP\in \hyD$ with $\cP\neq \cQ$ and $\Ex_\cP[y]\neq \Ex_\cQ[y]$ with the same conditional density on $S$.) 
                In both plots, we assume that $\cP_X = \cQ_X$.
            }
            \label{fig:ScenarioIII}
        \end{figure}
        \noindent We refer the reader to \Cref{sec:scenario:unconfoundedness} for a formal discussion on this condition and result. We would like to stress that the above characterization has a novel conceptual connection with an important field of statistics, called \emph{truncated statistics} \cite{Galton1897,cohen1991truncated,woodroofe1985truncated,cohen1950truncated,laiYing1991truncation}. The main task in truncated statistics concerns \emph{extrapolation:} given a true density $\cD$ over some domain $X$ and a set $S \subseteq X$, the question is whether the structure of $\cD$ can be identified from \emph{truncated} samples, \ie{}, samples from the conditional density of $\cD$ on $S$. 
        The condition of the above result requires the pairs $\cP,\cQ$ to be distinguishable on any set of the form $S\times \R$ (where $S$ has sufficient volume).
        In other words, any $\cP$ and $\cQ$ (with $\cP_X=\cQ_X$) whose \textit{truncations} to the set $S\times \R$ are identical must also have the same \textit{untruncated} means.
        Roughly speaking, this condition holds for any family $\hyD$ whose elements $\cP$ can be extrapolated given samples from their truncations to full-dimensional sets, a problem which is well-studied and provides us with multiple applications \cite{Kontonis2019EfficientTS,daskalakis2021statistical,lee2024efficient} (see \cref{lem:extrapolation,rem:extensionOfExtrpolation}).
        We refer to \Cref{sec:iden:unconfoundedness,sec:est:unconfoundedness} for a more extensive discussion.

        \paragraphit{Connections to Prior Work.} 
            This scenario captures two important and practical settings.
            First, as mentioned before, it captures regression discontinuity (RD) designs where propensity scores violate the overlap assumption for a large fraction of individuals but unconfoundedness holds.
            These designs were introduced by \citet{thistlethwaite1960regressionDiscontinuity}, {were independently discovered in many fields \cite{cook2008waitingforLife},} and have found applications in various contexts from Education \cite{thistlethwaite1960regressionDiscontinuity,angrist1999classSizeRD,klaauw2002regressionDiscontinuityEnrollment,black1999regressionDiscontinuity}, to Public Health \cite{moscoe2015rdPublicHealth}, to Labor Economics \cite{lee2010regressionDiscontinuity}.
            Formally, in an RD design, the treatment is a known deterministic function of the covariates: there is some known set $S$ and $T=1$ if and only if $x\in S$.
            \begin{restatable}[Regression Discontinuity Design]{definition}{defRD}
                \label{def:rdDesign}
                Given $c \in (0,\nfrac{1}{2})$, an observational study  $\cD$ is said to have a $c$-RD-design if there exists $S\subseteq\R^d$ such that $\vol(S)>c$, $\vol(\R^d\setminus S)>c$, and 
                \[
                    \forall_{x\in \R^d}\,,~~\forall_{y\in \R}\,,~~\quad 
                    p_0(x,y) =\mathds{1}\inbrace{x\not\in S} 
                    \quadand
                    p_1(x,y) = \mathds{1}\inbrace{x\in S}\,.
                \]
            \end{restatable}
            To the best of our knowledge in RD designs, ATE is only known to be identifiable under strong linearity assumptions on the expected outcomes \cite{hahn2001regressionDiscontinuity}. 
            Due to that, recent work focuses on identifying certain local treatment effects, which, roughly speaking, measure the effect of the treatment for individuals close to the ``decision boundary'' \cite{imbens2008regressionDiscontinuity}.
            In contrast, \cref{infthm:3} enables us to achieve identification under much weaker restrictions, \eg{}, it allows the expected outcomes to be any polynomial functions of the covariates (see \cref{lem:unconfoundedness:identifiableFamilies}).

            \begin{figure}[!ht]
            \centering
                \subfigure[]{
                    \includegraphics[width=0.45\linewidth,clip,trim={2.1cm 1.2cm 2.25cm 0.25cm}]{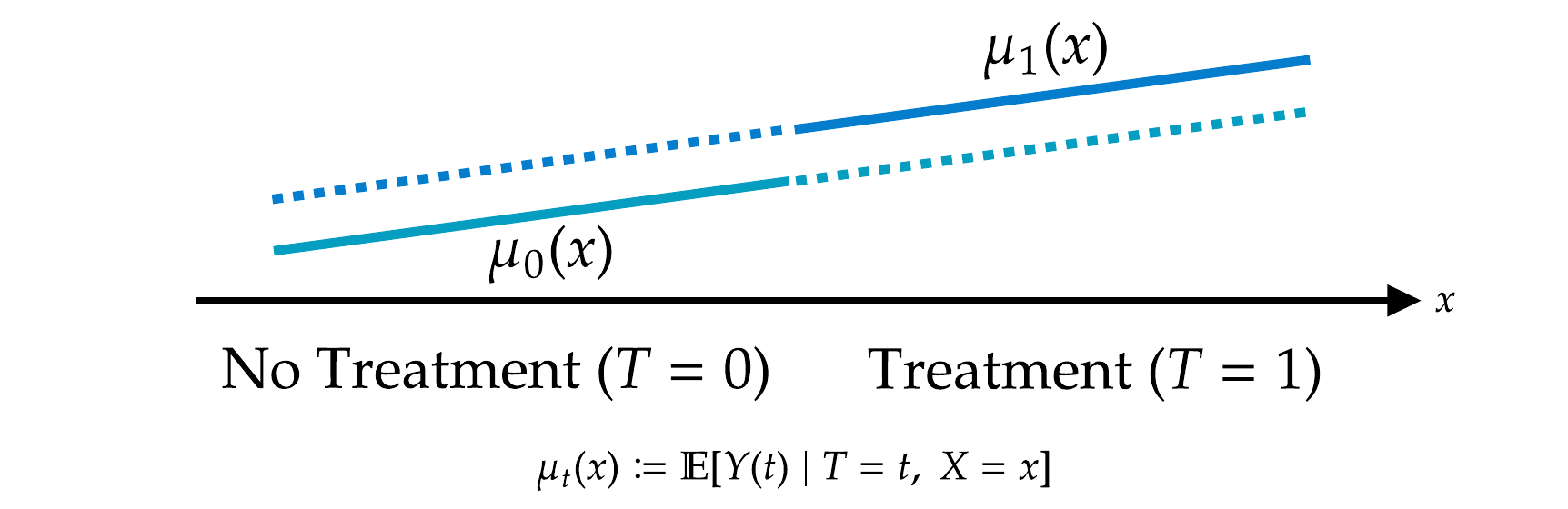}
                }
                \subfigure[]{
                    \includegraphics[width=0.45\linewidth,clip,trim={2.1cm 1.2cm 2.25cm 0.25cm}]{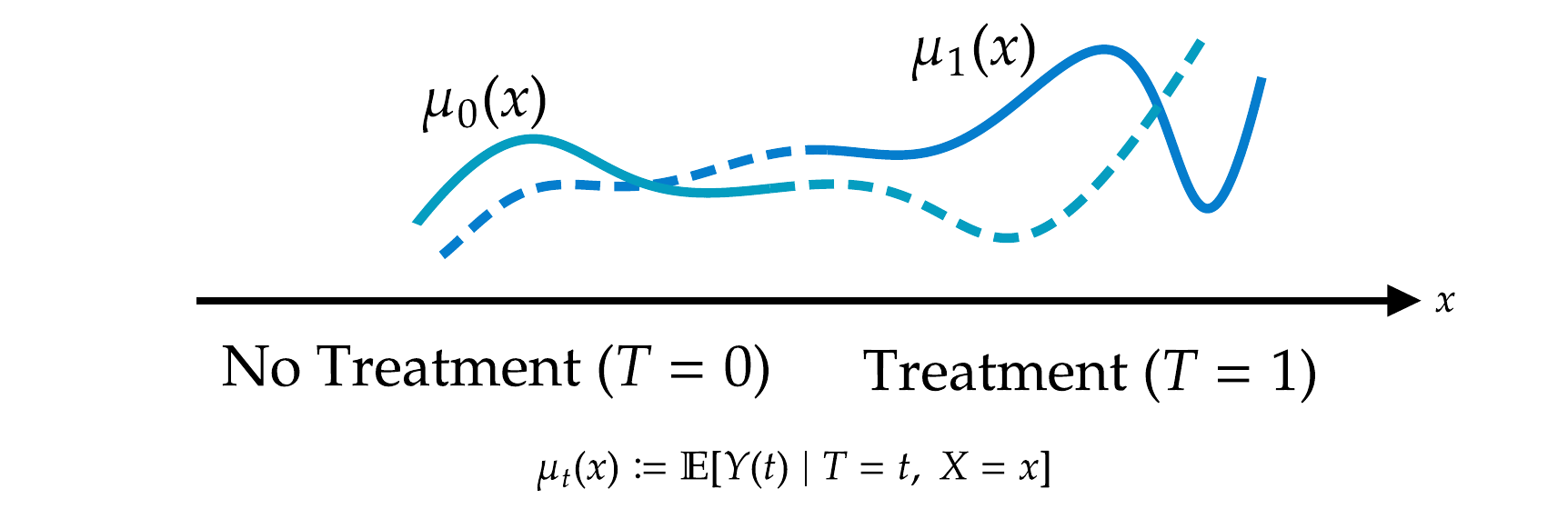}
                }
                \caption{
                {This figure illustrates two regression discontinuity designs. 
                In the first design (left figure), the expected outcomes $\mu_t\coloneqq \Ex[Y(t)\mid T{=}t, X{=}x]$ (for $t\in \zo$) are linear functions of the covariate $x\in \R$ and both $\mu_0(\cdot)$ and $\mu_1(\cdot)$ have the same slope.
                Hence, the treatment effect at the critical point $c$ (where the treatment assignment changes from 0 to 1) 
                is equal to ATE  --  \ie{}, $\tau_\cD=\lim_{x\to c^{+}}\mu_1(x)-\lim_{x\to c^{-}}\mu_0(x)$  --  this enables existing methods to identify ATE.
                In the second design (right figure), the expected outcomes are non-linear functions of the covariate and, hence, standard methods do not identify ATE.
                Here, provided $\mu_0(\cdot)$ and $\mu_1(\cdot)$ are polynomials, one can use the algorithms from \cref{infthm:3,infthm:SampleComplexity2} to identify and estimate ATE respectively.}
                }
                \label{fig:rdDesign}
            \end{figure}
        
            \noindent Apart from RD designs, the above scenario also captures observational studies where certain individuals have extreme propensity scores -- close to 0 or 1.
            This is a challenging case for the de facto inverse propensity weighted (IPW) estimators of $\tau$, whose error scales with $\sup_x \nfrac{1}{\inparen{e(x)(1-e(x))}}$ \cite{li2018overlapWeights,crump2009dealing,imbens2015causal}, and, hence, can be arbitrarily large even when overlap is violated for a single covariate $x$ \cite{kalavasis2024cipw}.
            In contrast to such estimators,  \cref{infthm:3} enables us to identify ATE even when propensity scores are violated for a large fraction of the covariates.

           \begin{remark}[Regression-Based Estimators]
               {Outcome-regression-based estimators for ATE estimate the regression functions $\mu_0(x)\coloneqq \Ex[Y|X{=}x, T{=}0]$ and $\mu_1(x)\coloneqq \Ex[Y|X{=}x, T{=}1]$.
               If overlap holds, this estimator can be computed from available censored data, providing an alternative proof of identification in Scenario~I. 
               Without overlap, the estimator may not be identifiable, and assumptions on $\mu_t(\cdot)$ are needed to enable identification. 
               A common assumption is that $\mu_t(\cdot)$ is a polynomial in $x$, this fits into the polynomial expectations model (\Cref{lem:unconfoundedness:identifiableFamilies}), {and can be used in Scenario~III as well}.}
               {Here, an interesting open problem is to use this approach to design some version of the popular doubly-robust estimators (\eg{}, \cite{Chernozhukov2018Double,Chernozhukov2018Double2018double,semenova2022estimationinferenceheterogeneoustreatment,Chernozhukov2022locally,robins2005doublyRobust,,foster2023orthognalSL,syrgkanis2022sampleSplitting,syrgkanis2022riesznet,syrgkanis2021long,syrgkanis2021dynamic}) in the general setting of Scenario~III.}
           \end{remark}

        \paragraphit{Finite-Sample Complexity.}
        As before, we complement the identification result with a finite sample complexity guarantee under a robust version of the above {identifiability} condition.

         \begin{inftheorem}
        [Informal, see \Cref{lem:est:unconfoundedness}] 
        \label{infthm:SampleComplexity2}
            Under a quantitative version of the condition in \Cref{infthm:3} parameterized with $c$ (see \cref{cond:Uncon:estimation} for details) and mild smoothness conditions on $\hyD$,
             there is an algorithm that, 
                given $n$ \iid{} samples from the censored distribution $\cC_\cD$ for any $\cD$ realizable by $\inparen{\hyPunconf(c),\hyD}$,
               and $\eps,\delta \in (0,1)$,
            outputs an estimate $\hat{\tau}$, such that $\abs{\hat{\tau}-\tau_{\cD}}\leq \eps$ with probability $1-\delta$.
            The number of samples is $\wt{O}(\nfrac{1}{\eps^4}) \cdot \log(\nfrac{1}{\delta}) \cdot  \mathrm{fat}_{O(\eps)}({\hyPunconf(c)}) \cdot \log N_\eps(\hyD)$.
        \end{inftheorem}
        As in the previous estimation result, the sample complexity depends on the fat-shattering dimension of $\hyP = \hyPunconf{(c)}$ and the covering number of $\hyD$. An interesting technical observation is that the estimation of (generalized) propensity scores corresponds to a well-known problem in learning theory, that of probabilistic-concept learning of \citet{kearns1994pconcept}. This connection allows us to get estimation algorithms for classes of bounded fat-shattering dimension.
        
        \paragraph{Scenario~IV: Neither Unconfoundedness nor Overlap.}
        {A natural extension of Scenarios~II and~III arises when both unconfoundedness and overlap fail simultaneously.
        In this setting, neither the overlap‐based arguments from Scenario~II nor the unconfoundedness‐based arguments from Scenario~III apply, making identification particularly challenging.
        Nevertheless, there are some special cases under this scenario where \cref{cond:iden} holds and, hence, ATE is identifiable.
        We illustrate one such example below, but we do not explore this scenario further because, to our knowledge, the resulting identifiable instances do not directly connect with existing causal inference literature. 
        
        \begin{example}
            This example is parameterized by a convex set $B$ with $\vol(\mathbb{R}^d{\setminus} B)>0$. 
            Let $\hyD_{\rm G}$ be the Gaussian family over $\R^d$ and $\hyP_B$ be the family of generalized propensities that \textit{(i)} may arbitrarily violate unconfoundedness and \textit{(ii)} satisfies $c$-overlap outside of $B$, \ie{}, for each $p\in \hyP_B$ and $x\not\in B$, $p(x,\cdot)\in (c,1-c)$ when $x\not\in B$.
            Here, ATE  is identifiable under any observational study realizable with respect to $\inparen{\hyD_{\rm G}, \hyP_B}$.
            (One way to see this is that restricting attention to $\R^d\setminus B$ recovers the overlap assumption in Scenario~II with $\hyD$ being truncations of Gaussians to non-convex sets -- which satisfies the corresponding identifiability condition, see \cref{infthm:2}.)
        \end{example} 
        }

\subsection{Related Work}     
    Our work is related to and connects several lines of work in causal inference and learning theory.
    We believe that an important contribution of our work is bridging these previously disconnected areas, possibly opening up new paths for applying learning-theoretic insights to causal inference problems.
    {We discuss the relevant lines of work below.}

    \subsubsection{Related Work in Causal Inference Literature}
        We begin with related work from the Causal Inference literature.
        Here, our work is related to the literature on sensitivity analysis -- which explores the sensitivity of results on deviations from unconfoundedness and is related to results in Scenario~II (\eg{}, \cite{cochran1965observationalStudies,rosenbaum1991sensitivity,tan2006distributional}), 
        the works on RD designs (\eg{}, \cite{hahn2001regressionDiscontinuity,imbens2008regressionDiscontinuity,cook2008waitingforLife}) -- which are a special case of Scenario~III -- and to works on handling extreme propensity scores (close to 0 or 1) which arise when overlap is violated and is considered in Scenario~III (\eg{}, \cite{crump2009dealing,li2018overlapWeights,khan2024trimming,kalavasis2024cipw}).  
        
    \vspace{-0.5mm}

        \paragraph{Extreme Propensity Scores.}
            Extreme propensity scores (those close to 0 or 1) are a common problem in observational studies.
            They pose an important challenge since the variance of most standard estimators of, \eg{}, the average treatment effect, rapidly increases as the propensity scores approach 0 or 1 -- leading to poor estimates. %
           {A} large body of work {designs estimators with lower variance} \cite{crump2009dealing,li2018overlapWeights,khan2024trimming,kalavasis2024cipw}.
            While these estimators are widely used, they introduce bias in the estimation of ATE, hence, they do not lead to point identification or consistent estimation, which is the focus of our work.
            We refer the reader to \citet{petersen2012diagnosing} for an extensive overview of violations of unconfoundedness and to \citet*{leger2022causal,li2018overlapWeights} for an empirical evaluation of the robustness of existing estimators in the absence of overlap.

        \vspace{-0.5mm}

        \paragraph{Sensitivity Analysis.}
        Sensitivity analysis methods in causal inference assess how unmeasured confounding can bias estimated treatment effects. 
        {The idea dates back to \citet*{cornfield1958smoking}, who studied the causal effect of smoking on developing lung cancer and showed that an unmeasured confounder needed to be nine times more prevalent in smokers than non-smokers to nullify the causal link between smoking and lung cancer  --  since this was unlikely, it strengthened the belief that smoking had harmful effects on health.
        \citet{rosenbaum1983sensitivity}, subsequently, proposed a sensitivity model for categorical variables.
        Since then, many works have extended the analysis of Rosenbaum's sensitivity model and introduced alternative parameterizations of the extent of confounding (\eg{}, \citet{rosenbaum2002observational,tan2006distributional,carnegie2016assessing,oster2019unobservable}).
        A notable line of work refines these models to obtain tight intervals in which the ATE lies with the desired confidence level \cite{zhao2019sensitivity,dorn2024doublyvalidSharpAnalysis,jin2022sensitivityanalysisfsensitivitymodels,dorn2023sharpSensitivityAnalysis,chernozhukov2023personalizedITE}.
        }
        {While these works construct valid uncertainty intervals that are valid without distributional assumptions, they do not achieve point identification.
        Finding the distributional assumptions necessary for point identification is the focus of our work.}

        \vspace{-0.5mm}
        
        \paragraph{Adversarial Errors in Propensity Scores.}
            Even with unconfoundedness, propensity scores have to be learned from data {(\eg{}, \cite{mccaffrey2004propensity,athey2019generalized,WESTREICH2010826})}, and errors in the estimation of propensity scores propagate to the estimate of ATE. %
            While under overlap, works from sensitivity analysis (discussed above) provide intervals containing the ATE, these intervals become vacuous even if overlap is violated for a single covariate.
            {\citet{kalavasis2024cipw} estimate ATE despite of adversarial errors and outliers, under specific assumptions, by merging outliers with nearby inliers to form ``coarse'' covariates.}
            {Our work is orthogonal to theirs in terms of both assumptions and objectives. They obtain interval estimates of treatment effects that are robust to adversarial errors, provided unconfoundedness holds. In contrast, we characterize settings where treatment effects can be point identified without adversarial errors, even when unconfoundedness or overlap fail.}

        \paragraph{Regression Discontinuity Designs.}
            {Regression discontinuity designs were introduced by \citet{thistlethwaite1960regressionDiscontinuity} in 1960, and have since been independently re-discovered\footnote{{Though there is some debate around this; see \citet{cook2008waitingforLife}.}} and studied in several disciplines, including Statistics (\eg{}, \cite{rubin1977regressionDiscontinuity,sacks1978regressionDiscontinuity}) and Economics (\eg{}, \cite{goldberger1972selection}). 
            See \citet{cook2008waitingforLife} for a detailed overview.}
            Today, there are two main types of Regression discontinuity (RD) designs: sharp RD designs, where treatment is deterministically assigned based on whether an observed covariate crosses a fixed cutoff,\footnote{We note that typically RD designs consider one-dimensional covariates and where the set $S$ (from \cref{def:rdDesign}) is an interval of the form $(\alpha,\infty)$ for some constant $\alpha$. In this work, we allow for high-dimensional covariates and any measurable set $S$ satisfying some mild assumptions on its volume.} and fuzzy RD designs, in which treatment assignment is probabilistic near the cutoff (\eg{}, \cite{lee2010regressionDiscontinuity, imbens2008regressionDiscontinuity,hahn2001regressionDiscontinuity}).
            In this work, we considered sharp RD designs, although our framework can also be applied to some fuzzy RD settings and exploring this further is a promising direction for future research. 
            {Recent works in regression discontinuity designs use} local linear regression to estimate the treatment effect at the cutoff (\eg{}, \cite{fan1996local,porter2003estimation,calonico2014robust}). %
            These approaches yield only a local average treatment effect and often require linearity or other strong parametric assumptions to ``extrapolate'' to a global average treatment effect (ATE); see \citet{hahn2001regressionDiscontinuity,cattaneo2019practical,chernozhukov2024appliedcausalinferencepowered}.  
            In contrast, our work facilitates point identification of the ATE in more general settings, by utilizing recent developments in truncated statistics (see \cref{rem:extensionOfExtrpolation,lem:iden:unconfoundedness:examples}). Finding interesting classes (apart from the ones mentioned in this work) that can be extrapolated is an interesting open question in truncated statistics, and any progress on it will also enable applications of our framework to these classes. 
    
    \subsubsection{Related Work in Learning Theory Literature}
        Next, we discuss relevant work in Learning Theory.
        Here, we draw on foundational results on probabilistic-concept learning \cite{kearns1994pconcept,alon1997scale} to get sample complexity bounds.
        Moreover, to satisfy the extrapolation condition in Scenario~III (\cref{infthm:3}), we leverage recent advances in truncated statistics \cite{daskalakis2021statistical}.
 
        \paragraph{Probabilistic Concepts.}
        {Most} prior works {in causal inference} assume {access to} an oracle that {estimates the propensity scores $e(x) = \Pr[T{=}1 | X{=}x]$}. 
       The propensity scores $e(\cdot)$ are $[0,1]$-valued, but the feedback provided to the learning algorithm is binary; it is the result of a coin toss where for each $x$, the probability of observing 1 is $e(x)$. Inference in this setting is well-studied in learning theory and corresponds to the problem of learning probabilistic concepts (or $p$-concepts), introduced by  \citet{kearns1994pconcept}. Learnability of a concept class of $p$-concepts is characterized by the finiteness of the fat-shattering dimension of the class (see \citet*{alon1997scale}). To the best of our knowledge, this connection was not reported in the area of causal inference prior to our work.
        
        \paragraph{Truncated Statistics.}
        Our work and in particular applications which violate overlap are closely related to the area of truncated statistics \cite{maddala1986limited,Galton1897,cohen1991truncated,woodroofe1985truncated,cohen1950truncated,laiYing1991truncation}. 
        Recently, there has been extensive work on truncated statistics regarding the design of efficient algorithms  \cite{daskalakis2018efficient,plevrakis2021learning,fotakis2020efficient,lee2025learningpositiveimperfectunlabeled}.
        However, all these works focus on computationally efficient learning of parametric families, while we focus on identification and estimation of treatment effects {and do not impose constraints on computational efficiency}.

\section{Preliminaries}
\label{sec:preliminaries}

     An observational study involves units (\eg{}, patients) with covariates $X\in\R^d$ (\eg{}, medical history). Each unit receives a binary treatment $T\in\{0,1\}$ (\eg{}, medication) with a fixed but unknown probability, independent across units, and we observe a treatment-dependent outcome $Y(T)\in\R$ (\eg{}, symptom severity). The tuple $(X,Y(0),Y(1),T)$ follows an unknown joint distribution $\cD$, which defines the study. For each $t\in\{0,1\}$, $\cD_{X,Y(t)}$ denotes the marginal over $X$ and $Y(t)$ and $\cD_X$ the marginal over $X$. To simplify the exposition, we assume that $\cD_{X,Y(0)}$ and $\cD_{X,Y(1)}$ are continuous distributions with densities throughout.

    \paragraph{Treatment Effects.}
    An important goal in causal inference is to identify treatment effects.
    The Average Treatment Effect $\tau_\cD$ (ATE) and the Average Treatment Effect on the Treated $\gamma_\cD$ (ATT) \cite{imbens2015causal, hernan2023causal, rosenbaum2002observational} are defined as
    \[
        \tau_\cD \coloneqq \Ex\nolimits_\cD\insquare{Y(1)-Y(0)}
        \qquadand
        \gamma_\cD \coloneqq \Ex\nolimits_\cD\insquare{Y(1)-Y(0)|T{=}1}
        \,.
    \]
    Since instead of observing full samples $(X,Y(0),Y(1),T)$, we only see the censored version $(X,Y(T),T)$, $\tau_\cD$ and $\gamma_\cD$ are non-identifiable without further assumptions \cite{chernozhukov2024appliedcausalinferencepowered,rosenbaum2002observational}.\footnote{In particular, $\Ex_\cD\insquare{Y(1)}$ is unobserved and may differ from $\Ex_\cD\insquare{Y(1)\mid T{=}1}$ by an arbitrary amount.}
    This brings us to our main tasks (presented in terms of  ATE but also relevant for any treatment effect):

    \begin{problem}[Identifying and Estimating ATE]\label{prob:main}
    An observational study is specified by the distribution $\cD$ of $(X,T,Y(0),Y(1))$ over $\R^d\times\{0,1\}\times\R\times\R$. Instead of $\cD$, the statistician has sample access to the censored distribution $\cC_\cD$ of $(X,T,Y(T))$. The statistician's goal is to address:
    \begin{enumerate}[leftmargin=15pt]
        \item \textbf{(Identification)}: What are  the minimal assumptions on an observational study $\cD$ so that there is a {deterministic} mapping $f$ satisfying $f(\cC_\cD)=\tau_\cD$ for any $\cD$ satisfying the assumptions?
        \item \textbf{(Estimation)}: What are the minimal assumptions on an observational study $\cD$
                    so that there is an algorithm $(\wh \tau_n)_{n \in \N}$
                    that given $n\geq 1$ \iid{} samples from $\cC_\cD$,  outputs an estimate $\wh{\tau}_n$ such that, with high probability, $\abs{\wh{\tau}_n - \tau_\cD}\leq \eps(n)$ for some $\eps(\cdot)$ satisfying $\lim_{n\to \infty}\eps(n)=0$?
    \end{enumerate}
    \end{problem}
    \noindent When the distribution $\cD$ is clear from context, we write $\tau$ and $\cC$ for $\tau_\cD$ and $\cC_\cD$, respectively. In general, $\tau_\cD$ cannot be identified from censored samples. This is because there exist $\cD^{(1)}$ and $\cD^{(2)}$ with $\abs{\tau_{\cD^{(1)}}-\tau_{\cD^{(2)}}}\gg1$ but $\cC_{\cD^{(1)}}=\cC_{\cD^{(2)}}$. 
    Hence, one needs some assumptions on $\cD$ to have any hope of solving \cref{prob:main}. The above can be naturally adapted to ATT.
    
       \paragraph{Unconfoundedness and Overlap.} Unconfoundedness and overlap are common sufficient assumptions that enable the identification and estimation of ATE, and have been utilized in a number of important studies; see \cite{imbens2015causal,hernan2023causal,rosenbaum2002observational} and \cref{sec:intro}.
      The {observational study} $\cD$ is said to satisfy unconfoundedness if, for each $x\in \R^d$, it holds: $Y(0)~\bot~ T~\mid~ X{=}x$ and $Y(1)~\bot~ T~\mid~ X{=}x$.
        In other words, the potential outcomes are independent of the treatment $T$ given $X{=}x$. 
        Next, we move to overlap, which ensures that treatment probabilities are bounded away from 0 and 1.
                The observational study $\cD$ is said to satisfy overlap if, for each $x\in \R^d$, $0<\Pr_\cD[T{=}1\mid X{=}x]<1$.
            Given a constant $c\in(0,\nfrac{1}{2})$, if $\cD$ satisfies $c<\Pr_\cD[T{=}1\mid X{=}x]<1-c$ (for each $x\in \R^d$) then $\cD$ is said to satisfy the $c$-overlap condition.
        Although unconfoundedness and overlap suffice to estimate $\tau$ with enough samples, they are not necessary. 
        Unconfoundedness and overlap are often violated (see \cref{sec:examples:violation} for a discussion and examples). %
        To derive necessary and sufficient conditions for identifying $\tau$, we now introduce certain conditional probabilities. 
        \begin{definition}[Generalized Propensity Score]\label{def:newScoresCov}
            Fix distribution $\cD$.
            For each $y\in \R$ and $t \in \zo$, the generalized propensity score induced by $\cD$ is  
            $
                p_t(x, y) \coloneqq \Pr\nolimits_\cD[T{=}t\mid X{=}x, Y(t){=}y].
        $
        \end{definition}
        For the reader familiar with causal inference terminology, note that the generalized propensity scores differ from the ``usual'' propensity score $e(x)\coloneqq\Pr_\cD[T{=}1\mid X{=}x]$: while $e(\cdot)$ is always identifiable from the data, $p_0(\cdot)$ and $p_1(\cdot)$ in general are not.\footnote{There exist $\cD^{(1)}$ and $\cD^{(2)}$ with very different generalized propensity scores but identical censored distributions.} To succinctly state assumptions on generalized propensity scores and $\cD$, we adopt a statistical learning theory notion of realizability.

        \begin{definition}
        [Concepts]
            We say that $\hyP$ is a concept class of generalized propensity scores if $\hyP \subseteq [0,1]^{\R^d \times \R}$ and $\hyD$ is a concept class of conditional-outcome distributions if $\hyD \subseteq \Delta(\R^d \times \R)$.
        \end{definition}
       Realizability couples the observational study $\cD$ with the pair of concept classes $\inparen{\hyP,\hyD}$.
        \begin{definition}
            [Realizability] 
            \label{def:realizability}
            Consider a pair of concept classes $(\hyP, \hyD)$.
            An observational study $\cD$ is said to be realizable with respect to $(\hyP, \hyD)$, if $p_0(\cdot),p_1(\cdot)\in \hyP$ and $\cD_{X,Y(0)},\cD_{X,Y(1)}\in \hyD$.
        \end{definition}
        If $\cD$ only satisfies $p_0(\cdot),p_1(\cdot)\in \hyP$ (respectively $\cD_{X,Y(0)},\cD_{X,Y(1)}\in \hyD$), then $\cD$ is said to be realizable with respect to $\hyP$ (respectively $\hyD$).
        We are interested in conditions on the pair $(\hyP, \hyD).$
        Finally, we use the following terminology for {certain} elements of {interest in} $\hyP \times \hyD.$

        \begin{definition}
        [Compatibility] 
        \label{def:compatible}
        
        {A tuple} $(p, \cP) \in \hyP \times \hyD$ is {said to be} compatible
        with concept classes $(\hyP, \hyD)$ if there {is} another tuple $(\wh p, \wh \cP)\in \hyP\times \hyD$ such that setting $(p_0,p_1,\cD_{X,Y(0)},\cD_{X,Y(1)})=(p,\wh{p},\cP,\wh{P})$ (or \textit{equivalently} the re-ordered {assignment} $(p_0,p_1,\cD_{X,Y(0)},\cD_{X,Y(1)})=(\wh{p},p,\wh{P},\cP)$) defines a valid observational study, \ie{}, a valid distribution over $(X,T,Y(0),Y(1))$.
        \end{definition}
        {Note that the tuples $(p,\cP)$ and $(\wh p, \wh \cP)$ in the above definition are not necessarily distinct.}

\section{Proofs of Characterizations and Overview of Estimation Algorithms}
\label{sec:overview} 
    In this section, we prove {characterizations of identifiability of} ATE and ATT (\cref{infthm:Suff,infthm1}) and provide an overview of our estimation algorithms.
    The characterizations of \cref{infthm:Suff,infthm1} consist of two parts.
    In \cref{sec:proofof:thm:identification}, we show that \cref{cond:iden} is sufficient for {identifying} ATE {and ATT}.
    Then, in \cref{sec:proofof:infthm:nece}, we prove that {it is also necessary}.
        Finally, we provide an overview of our algorithms for estimating ATE in Scenarios I, II, and III in \cref{sec:overview:estimation}.

    \subsection{\cref{cond:iden} is Sufficient for Identification of ATE and ATT} %
        \label{sec:overview:ATESufficient}
        \label{sec:proofof:thm:identification}

     \cref{cond:iden} is our main tool to obtain our identification characterizations for ATE and ATT (\cref{infthm:Suff,infthm1}). 
        In this section, we explain our technique for identifying these treatment effects from the censored distribution $\cC_\cD$ under \cref{cond:iden}.  
        We proceed with the following claim.

    \begin{claim}[Sufficiency of \Cref{cond:iden}]\label{claim:sufficiency_condition1}
Assume that $(\hyP, \hyD)$ satisfy \Cref{cond:iden}.
\begin{enumerate}[itemsep=0pt]
    \item The average treatment effect $\tau_\cD$ is identifiable from the censored distribution $\cC_\cD$ for any observational study $\cD$ realizable with respect to $(\hyP, \hyD)$.
    
    \item If $\Pr[T = 1] > 0$, then the average treatment effect on the treated $\gamma_\cD$ is identifiable from the censored distribution $\cC_\cD$ for any observational study $\cD$ realizable with respect to $(\hyP, \hyD)$.
\end{enumerate}
\end{claim}

       \begin{proof}
            First, we show that \cref{cond:iden} is sufficient to identify $\Ex_\cD[Y(0)]$ (which provides the sufficiency part of \cref{infthm1}\footnote{\label{fn:ATT_footnote}Recall that ATT is $\Ex[Y(1)\mid T{=}1]-\Ex[Y(0)\mid T{=}1]$.
            To identify ATT it is sufficient to identify $\Ex[Y(0)]$ since the first term is always identifiable and the second term is related to $\Ex[Y(0)]$ as $\Ex[Y(0)\mid T{=}1]=\inparen{\nfrac{1}{\Pr[T{=}1]}}\cdot \inparen{\Ex[Y(0)] - \Ex[Y(0)|T{=}0]\cdot \Pr[T{=}0]}$ (where all quantities except $\Ex[Y(0)]$ are always identifiable from $\cC_\cD$).
            Note $\Pr[T{=}1]>0$ as otherwise ATT may not be well-defined. Note that this connection between $\E[Y(0)]$ and ATT implies that in terms of identification, the two quantities are equivalent.}).
            Then, an analogous proof shows that the same condition is sufficient to identify $\Ex_\cD[Y(1)]$.
            The combination of the two is sufficient to identify ATE $\tau_\cD=\Ex_\cD[Y(1)]-\Ex_\cD[Y(0)]$ and proves the sufficiency part of \cref{infthm:Suff}.
            
            Fix some $\cD$ realizable with respect to $\inparen{\hyP,\hyD}$.
            We claim that, given as input the censored distribution $\cC_\cD$, there is a deterministic procedure $\Phi$ that constructs a function $\Phi(\cC_\cD)\colon \R^d \times \R \to \R$  such that
            $\Phi(\cC_\cD)(x,y)
                    =
                    p_0(x,y) \cdot \cD_{X,Y(0)}(x,y)$ for all $(x,y).$
            In other words, this means that, given $\cC_\cD$, there is a deterministic method that identifies the product $p_0(x,y) \cdot \cD_{X,Y(0)}(x,y)$ {at} any $x,y.$
            
            \paragraphit{Existence of $\Phi$.}
            Given $\cC_\cD$ as input, we let $\Phi(\cC_\cD)$ be a function that maps $(x,y) \mapsto p_0(x,y) \cdot \cD_{X,Y(0)}(x,y)$.
            This mapping can be identified from $\cC_\cD$ because a censored sample has the density of $X, Y(0) \mid T{=}0$ and we are interested in the density of $X, Y(0), T{=}0$.
            By the Bayes rule, we can obtain the latter from the former by multiplying with $\Pr[T{=}0]$, which itself is identifiable from the sample as there is no censoring over $T$.

            \paragraphit{Identification of ~$\E_\cD[Y(0)]$~ via  $\Phi$.}
            Given $\cC_\cD$, the procedure $\Phi$ allows us to eliminate some candidates in $(\hyP, \hyD).$
            For each $\cD$ realizable with respect to $(\hyP, \hyD)$, let $S_{\Phi,\cC_\cD}\subseteq\inparen{\hyP,\hyD}$ be the set of $(p,\cP)\in\sinparen{\hyP,\hyD}$ that are compatible (recall  \Cref{def:compatible}) and consistent with $\Phi(\cC_\cD)$ (the subset is non-empty because $\cD$ is realizable). 
            {That is,} each $(p,\cP) \in S_{\Phi, \cC_\cD}$ is compatible and satisfies
            \begin{align*}
                \forall_{x\in \R^d}\,,~~
                \forall_{y\in \R}\,,~~\quad 
                    p(x,y) \cdot \cP(x,y)
                    &= 
                    \Phi(\cC_\cD)(x,y)
                \,,
                \yesnum\label{eq:def_Sg}\\
                \forall_{x\in \R^d}\,,\qquad\qquad\quad
                    \cP_X(x) &= \cD_X(x)\,.
                \yesnum\label{eq:def_Sg:2}
            \end{align*}
            Here, $\cD_X$ is indeed specified by $\cC_\cD$ since there is no censoring on the covariates and, hence, $\inparen{\cC_\cD}_X=\cD_X$.
            Hence, due to \cref{eq:def_Sg}, for any $(p,\cP), (q,\cQ)\in S_{\Phi, \cC_\cD}$, it holds 
            that $ p(x,y) \cdot \cP(x,y)
                    =
                    q(x,y) \cdot \cQ(x,y)$
            for all $x,y$ and so, combining the above with {\cref{cond:iden}}, we get that 
            $
                \Ex_{(x,y)\sim \cP}[y]
                = \Ex_{(x,y)\sim \cQ}[y]\,.
              $
            Since $\cD$ is realizable with respect to $\inparen{\hyP,\hyD}$, $p_0\in \hyP$ and $\cD_{X,Y(0)}\in \hyD$, and, further, since $(p_0,\cD_{X,Y(0)})$ satisfies the requirements in \cref{eq:def_Sg,eq:def_Sg:2}, $(p_0,\cD_{X,Y(0)})\in S_{\Phi,\cC_\cD}$.
            Therefore, for any $(p,\cP)\in S_{\Phi, \cC_\cD}$, 
            $
                \Ex_{(x,y)\sim \cP}[y]
                =
                \Ex_\cD[Y(0)].
            $
            Now, we have shown $\Ex_\cD[Y(0)]$ is a deterministic function of $\Phi$ and $\cC_\cD$: 
                $\Ex_\cD[Y(0)]$ is equal to $\Ex_{(x,y)\sim\cP}[y]$ for any $\cP$ which is consistent with $\Phi$ and $\cC_\cD$. %
            Since $\Phi$ itself is a deterministic function of $\cC_\cD$ (due to our claim at the start), there is a mapping $m_0$ satisfying $m_0(\cC_\cD)=\Ex_\cD[Y(0)]$ for any $\cD$ consistent with $\inparen{\hyP, \hyD}$.
           \end{proof}

    \subsection{\cref{cond:iden} is Necessary for Identification of ATE and ATT}
    \label{sec:proofof:infthm:nece}

    In this section, we complete the proofs of \Cref{infthm:Suff,infthm1} by further showing that \Cref{cond:iden} is necessary for identification. To do that, we show that if the condition does not hold, then one can find two observational studies with different treatment effects but identical censored distributions. We prove the following claim, which deals with the case of ATT.

    \begin{claim}[Necessity of \Cref{cond:iden}]
If the tuple $(\hyP, \hyD)$ does not satisfy \Cref{cond:iden}, then:
\begin{enumerate}[itemsep=0pt]
    \item For every deterministic mapping $f$, there exists an observational study $\cD$ realizable with respect to $(\hyP, \hyD)$ such that the average treatment effect $\tau_\cD\neq f(\cC_\cD)$.

    \item For every deterministic mapping $f$, there exists an observational study $\cD$ realizable with respect to $(\hyP, \hyD)$ {satisfying} $\Pr_\cD[T {=} 1] > 0$ and the average treatment effect on the treated $\gamma_\cD \ne f(\cC_\cD)$.
\end{enumerate}
In both cases, $\cC_\cD$ denotes the censored distribution of $\cD$.
\label{claim1}
\end{claim}

       \begin{proof}[Proof of \Cref{claim1}]
            Fix any classes $\hyP,\hyD$ that do not satisfy {\cref{cond:iden}}.
            Since \cref{cond:iden} does not hold, there exist distinct compatible tuples $(p,\cP),(q,\cQ)\in (\hyP,\hyD)$  satisfying: 
            \begin{align*}
                \Ex_{(x,y)\sim \cP}[y] &\neq \Ex_{(x,y)\sim \cQ}[y],,\yesnum\label{eq:ATTproof:1}\\
                \cP_X &= \cQ_X\,,\yesnum\label{eq:ATTproof:2}\\
                \forall_{x\in   \supp(\cP_X)},~
                    \forall_{y\in\R},~\quad
                        p(x,y) \cP(x,y) &=  q(x,y) \cQ(x,y)\,.\yesnum\label{eq:ATTproof:3}
            \end{align*}
            We construct distributions $\cD^{(1)}$ and $\cD^{(2)}$ such that:  
{\textbf{(i)~} $\E_{\cD^{(1)}}[Y(0)] \neq \E_{\cD^{(2)}}[Y(0)]$;  
\textbf{(ii)~} $\E_{\cD^{(1)}}[Y(1)] = \E_{\cD^{(2)}}[Y(1)]$; and  
\textbf{(iii)~}} the censored distributions coincide, \ie{}, $\cC_{\cD^{(1)}} = \cC_{\cD^{(2)}}$. {Clearly, the ATEs differ: $\tau_{\cD^{(1)}} \ne \tau_{\cD^{(2)}}$. Since $\cC_{\cD^{(1)}} = \cC_{\cD^{(2)}}$, no deterministic mapping $f$ can satisfy $f(\cC_{\cD^{(i)}}) = \tau_{\cD^{(i)}}$ for both $i\in\sinbrace{1, 2}$. This establishes the first part of the statement: the ATE is not identifiable in the absence of \Cref{cond:iden}.}

            The construction is as follows and uses the compatibility of the tuples $(p,\cP)$ and $(q,\cQ)$. 
            Let $(\wh p,\wh \cP)\in\hyP\times \hyD$ be the tuple witnessing that $(p,\cP)$ is compatible.
            \begin{enumerate}
                \item {First, we construct $\cD^{(1)}$ as the  distribution consistent with the following
                \begin{align*}
                    \cD^{(1)}_{X,Y(0)} =  \cP\,,\quad &
                    \Pr\nolimits_{\cD^{(1)}}[T{=} 0\mid X{=}x, Y(0){=}y] = p(x,y)\,,\quad\yesnum\label{eq:constructionD1:a:ATT}\\
                    \cD^{(1)}_{X,Y(1)} = \wh \cP\,,\quad &
                    \Pr\nolimits_{\cD^{(1)}}[T{=} 1\mid X{=}x, Y(1){=}y] = \wh p(x,y)\,.\yesnum\label{eq:constructionD1:b:ATT}
                \end{align*}
                The distribution $\cD^{(1)}$ is a valid observational study due to compatibility of $(p,\cP)$.}
                \item {Next, we construct $\cD^{(2)}$ as the  distribution consistent with the following %
                \begin{align*}
                    \cD^{(2)}_{X,Y(0)} =  \cQ\,,\quad &
                    \Pr\nolimits_{\cD^{(2)}}[T{=} 0\mid X{=}x, Y(0){=}y] = q(x,y)\,,\quad\yesnum\label{eq:constructionD2:a:ATT}\\
                    \cD^{(2)}_{X,Y(1)} = \wh \cP\,,\quad &
                    \Pr\nolimits_{\cD^{(2)}}[T{=} 1\mid X{=}x, Y(1){=}y] = \wh p(x,y)\,.\yesnum\label{eq:constructionD2:b:ATT}
                \end{align*}
                Under
                these choices, it is not immediate that $\cD^{(2)}$ is a valid distribution on $(X,T,Y(0),Y(1)).$ In particular, the above pair of equations couples the distributions of $X$ and $T$ and hence we have to ensure that they are consistent between each other.
              To this end, to verify that the distribution $\cD^{(2)}$ constitutes a valid observational study, we need to ensure that the marginal distribution of $(X, T)$ induced by \cref{eq:constructionD2:a:ATT} is consistent with that induced by \cref{eq:constructionD2:b:ATT}.
                First, observe that the marginal over $X$ in the two equations is $\cQ_X$ and $\wh {\cP}_X$ respectively, which are identical because $\cQ_X=\cP_X$ (\cref{eq:ATTproof:2}) and $\wh {\cP}_X = \cP_X$ (because $(p,\wh p, \cP, \wh{\cP})=(p_0,p_1,\cD_{X,Y(0)},\cD_{X,Y(1)})$ forms a valid observational study).
                It remains to check that $e(x)=\Pr[T{=}1\mid X={x}]$ is identical across \cref{eq:constructionD2:a:ATT,eq:constructionD2:b:ATT}.
                To see this, fix any $x\in \supp(\cP_X)$.
                Observe that from \cref{eq:constructionD2:a:ATT}
                \[
                    \Pr[T{=}1\mid X={x}]
                    \,=\, 1 - \frac{\int_y q(x,y)\cQ(x,y)  \d y}{\int_y \cQ(x,y)  \d y}
                    ~~~\Stackrel{\eqref{eq:ATTproof:2},~\eqref{eq:ATTproof:3}}{=}~~~ 1 - \frac{\int_y p(x,y)\cP(x,y)  \d y}{\int_y \cP(x,y)  \d y}\,,
                \]
                and from \cref{eq:constructionD2:b:ATT},
                \[
                    \Pr[T{=}1\mid X={x}]
                    \,=\, \frac{\int_y \wh {p}(x,y)\wh \cP(x,y)  \d y}{\int_y \wh \cP(x,y)  \d y}
                    \,\Stackrel{}{=}\, 1 - \frac{\int_y p(x,y)\cP(x,y)  \d y}{\int_y \cP(x,y)  \d y}\,.
                \]
                To see the second equality, consider the observational study $\cD^{(1)}$ defined above.
                Since this study is valid, we have that for each $x\in \supp(\cP_X)$
                \[
                    \frac{\int_y \wh {p}(x,y)\wh \cP(x,y)  \d y}{\int_y \wh \cP(x,y)  \d y}
                    = \Pr\nolimits_{\cD^{(1)}}\sinsquare{T {=} 1\mid X {=}x}
                    = 1- \frac{\int_y \Pr[T{=}0 \mid X{=}x, Y(0) {=} y]\cdot \cD^{(1)}_{X,Y(0)}(x,y) \d y}{\int_y \cD^{(1)}_{X,Y(0)}(x,y) \d y}
                \]
                The result follows by noting that $\Pr[T{=}0 \mid X{=}x, Y(0) {=} y] = p$ and $\cD^{(1)}_{X,Y(0)} = \cP.$}
            \end{enumerate}
            {We have constructed valid observational studies $\cD^{(1)}$ and $\cD^{(2)}$, which by construction}
            are realizable by $\inparen{\hyP,\hyD}$ and satisfy 
              $  \E_{\cD^{(1)}}[Y(0)] - \E_{\cD^{(2)}}[Y(0)]
                    = 
                    \Ex_{(x,y)\sim \cP}[y] - \Ex_{(x,y)\sim \cQ}[y]
                    ~~\Stackrel{}{\neq}~~ 0\,.
                $ {It is also clear from our construction that $ \E_{\cD^{(1)}}[Y(1)] - \E_{\cD^{(2)}}[Y(1)]
                    =  \Ex_{(x,y)\sim \wh \cP}[y] - \Ex_{(x,y)\sim \wh \cP}[y]
                    ~~\Stackrel{}{=}~~ 0\,.$}
                
            Finally, we claim that $\cC_{\cD^{(1)}}=\cC_{\cD^{(2)}}$, which leads to a contradiction since it implies $f(\cC_{\cD^{(1)}})=f(\cC_{\cD^{(2)}})$ and, hence, for at least one $i\in \inbrace{1,2}$, $f(\cC_{\cD^{(i)}})\neq \E_{\cD^{(i)}}[Y(0)].$

                It remains to prove that $\cC_{\cD^{(1)}}=\cC_{\cD^{(2)}}$.
                Consider any $i\in \inbrace{1,2}$.
                Let $(X, Y_{\rm obs}, T)$ denote the random variables observed in the censored data, where if $T{=}1$, then $Y_{\rm obs}=Y(1)$ (\ie{}, $Y(0)$ is censored) and, otherwise, $Y_{\rm obs}=Y(0)$ (\ie{}, $Y(1)$ is censored). 
                Consider the observation $(X{=}x, Y_{\rm obs}=y, T{=}t).$
                $\cC_{\cD^{(i)}}$ is the distribution which assigns the following density to it: 
                \begin{align*}
                    \cC_{\cD^{(i)}}(x,y,t)
                    =%
                    \begin{cases}
                        \cD_{X,Y(1)}^{(i)}(x,y) \cdot \Pr_{\cD^{(i)}}[T{=}1\mid X{=}x, Y(1){=}y]
                            & \text{if }t=1\,,\\
                        \cD_{X,Y(0)}^{(i)}(x,y)\cdot {\Pr_{\cD^{(i)}}[T{=}0\mid X{=}x, Y(0){=}y]}
                        &\text{if }t=0\,.
                    \end{cases}
                \end{align*}
                Note that for any $(x,y)\in\supp{(\cP_X)}\times \mathbb{R}$, \[\cC_{\cD^{(i)}}(x,y,0)=\int_{y'\in\mathbb{R}} \cD^{(i)}_{X,Y(1),Y(0),T}(x,y',y,0) \d y'\] and \[\cC_{\cD^{(i)}}(x,y,1)=\int_{y'\in\mathbb{R}} \cD^{(i)}_{X,Y(1),Y(0),T}(x,y,y',1) \d y',\] so $\cD^{(i)}$ is a valid density for the censored distribution.
                We claim that the above does not depend on the choice of $i\in \inbrace{1,2}$.
                To see this, first consider $t=0$, then $\cC_{\cD^{(1)}}(x,y,t) = %
                \cP(x,y)p(x,y)$ and $\cC_{\cD^{(2)}}(x,y,t) = %
                \cQ(x,y)q(x,y)$.
                However, due to \cref{eq:ATTproof:3}, $\cP(x,y)p(x,y)$ and $\cQ(x,y)q(x,y)$ are identical for each $(x,y)\in \supp(\cP_X)\times \R$.
                Moreover, the two are also identical for any $x\not\in \supp(\cP_X)$ and $y\in \R$, as $\supp(\cP_X)=\supp(\cQ_X)$, hence, for this $(x,y)$, $\cP(x,y)=\cQ(x,y)=0$.
                Finally, consider $t=1$ and observe that in this case the constructions of $\cD^{(1)}$ and $\cD^{(2)}$ are identical and, hence, immediately $\cC_{\cD^{(1)}}(x,y,t)=\cC_{\cD^{(2)}}(x,y,t)$.

               In the following two paragraphs, we argue for the unidentifiability of ATT in the absence of \Cref{cond:iden}. Assume $\Pr_{\cC_{\cD^{(1)}}}[T = 1] > 0$. Since $\Pr_{\cD^{(i)}}[T = 1] = \Pr_{\cC_{\cD^{(i)}}}[T = 1]$ for both $i \in \{1, 2\}$, it follows that $\Pr_{\cD^{(i)}}[T = 1] > 0$ for both $i \in \{1, 2\}$. Moreover, since $\E_{\cD^{(1)}}[Y(0)] \ne \E_{\cD^{(2)}}[Y(0)]$ and $\cC_{\cD^{(1)}} = \cC_{\cD^{(2)}}$, it follows from the argument in~\cref{fn:ATT_footnote} that ATTs differ: $\gamma_{\cD^{(1)}} \ne \gamma_{\cD^{(2)}}$. Therefore, no deterministic function $f$ can satisfy $f(\cC_{\cD^{(i)}}) = \gamma_{\cD^{(i)}}$ for both $i \in \{1, 2\}$.
                
                If $\Pr_{\cC_{\cD^{(1)}}}[T = 1] = 0$, consider two new distributions $\cD^{(3)}$ and $\cD^{(4)}$ constructed from $\cD^{(1)}$ and $\cD^{(2)}$, respectively, by relabeling the treatment and control groups -- that is, switching the roles of treated and control units. Then $\Pr_{\cD^{(i)}}[T = 1] = 1$ for both $i \in \{3, 4\}$. In other words, the ATT for $\cD^{(3)}$ and $\cD^{(4)}$ equals the ATE for $\cD^{(3)}$ and $\cD^{(4)}$, respectively. Since $\cC_{\cD^{(3)}} = \cC_{\cD^{(4)}}$ and $\tau_{\cD^{(3)}} \ne \tau_{\cD^{(4)}}$, it follows that no deterministic mapping $f$ can satisfy $f(\cC_{\cD^{(i)}}) = \tau_{\cD^{(i)}}$, which equals $\gamma_{\cD^{(i)}}$, for both $i \in \{3, 4\}$.
                
 \end{proof}

\noindent By combining \Cref{claim:sufficiency_condition1,claim1}, we complete the proofs of \Cref{infthm:Suff,infthm1}.

    \subsection{Overview of Estimation Algorithms}\label{sec:overview:estimation}
        In this section, we overview our algorithms for estimating ATE in Scenarios I-III. We refer the reader to \cref{sec:estimation} for formal statements of results and algorithms.

        \paragraph{Standard Approach to Estimate ATE.}
        We begin with the standard scenario (Scenario I) where unconfoundedness and $c$-overlap hold (and where methods to estimate ATE are already known).
        Recall that in this scenario, $\tau_\cD$ can be decomposed as in \cref{eq:decomposition:unconfoundedness}, which leads to the following finite sample version: given estimates $\wh{e}(\cdot)$ of the propensity scores $e(\cdot)$, 
        \[ 
            \wh{\tau} = \sum_i \frac{y_i t_i}{\wh{e}(x_i)} - \sum_i \frac{y_i(1-t_i)}{1-\wh{e}(x_i)}\,.
            \yesnum\label{eq:overview:decomposition}
        \] 
        \noindent This decomposition has several useful properties.
        First, when the outcomes are bounded -- a standard setting (see, \eg{}, \citet{kallus2021minimax}) -- each term in the decomposition (\ie{}, $\nfrac{y_i t_i}{\wh{e}(x)}$ and $\nfrac{y_i (1-t_i)}{(1-\wh{e}(x))}$) is a bounded random variable.
        Roughly speaking, under the assumption that $\nfrac{1}{\wh{e}(\cdot)}\approx \nfrac{1}{e(\cdot)}$, this enables one to use the Central Limit Theorem to deduce that, given $n$ samples, $\abs{\tau-\wh{\tau}}\leq O\inparen{\nfrac{1}{\sqrt{n}}}$ with high probability.
        Second, because we assume $c$-overlap, one can show that if $\wh{e}(\cdot)$ is close to $e(\cdot)$ (\eg{}, $\int\abs{e(x)-\wh{e}(x)}\cD_X(x)\d x\approx 0$), then their inverses $\nfrac{1}{\wh{e}(\cdot)}$ and $\nfrac{1}{e(\cdot)}$ -- which show up in the above decomposition -- are also close to each other.
        The sample complexity of learning $e(\cdot)$ can be bounded by observing that the problem is equivalent to estimating probabilistic concepts, introduced by \citet{kearns1994pconcept} and imposing the family of propensity scores to have a finite fat-shattering dimension.
        While the equivalence to probabilistic concept learning is straightforward, we have not been able to find a reference for it in the learning theory or causal inference literature (which usually assume an estimation oracle {with a small, \eg{}, $L_2$-error,} as a black-box{; see} \cite{kennedy2024agnostic,foster2023orthognalSL,jin2024structureagnosticoptimalitydoublyrobust}). For completeness, we present details on obtaining the sample complexity in \cref{sec:nuisanceParameters}.

        \paragraph{Hurdles in Using \cref{eq:overview:decomposition} in General Scenarios.}
        None of these ideas work in the more general Scenarios II and III.
        \begin{itemize}
            \item[$\triangleright$] 
                In Scenario II, since unconfoundedness does not hold, the above decomposition is not even true and, while one could write a similar decomposition with generalized propensity scores $p_0(\cdot)$ and $p_1(\cdot)$, they cannot generally be estimated from censored data.
            \item[$\triangleright$] 
                In Scenario III, overlap does not hold and, hence, the terms in the above decomposition are no longer bounded.
            In fact, for regression discontinuity designs \textit{all} terms are unbounded, since for each $x,$ $e(\cdot)\in \zo$.
            In fact, even when overlap holds for ``most'' covariates, extreme propensity scores (close to 0 or 1) are known to be problematic -- a number of heuristics have been proposed in the literature (\eg{}, \cite{crump2009dealing,li2018overlapWeights,khan2024trimming}). Finally, the recent work of \citet{kalavasis2024cipw} presents a (rigorous) variant of IPW estimators that handles outliers and errors in the propensities but requires additional assumptions on data. %
        \end{itemize}
        Thus, different ideas are needed to estimate $\tau$ in general scenarios.

        \paragraph{Our Approach.}
            We take a completely different approach to estimation, based on \cref{cond:iden}.
            Since \cref{cond:iden} is sufficient for identifying ATE in all scenarios, our approach is quite general -- we present two algorithms -- one for Scenario II and one for Scenario III -- that work for all of the interesting and widely studied special cases of these scenarios discussed in \cref{sec:applicationsAndEstimation}.
            Having general estimators can be useful since, like unconfoundedness, distributional assumptions and, hence, \cref{cond:iden}, are \textit{not} testable from $\cC_\cD$.\footnote{{In particular, given censored samples from  $\cC_\cD$ and concept classes $(\hyP,\hyD)$, it is impossible to verify whether $\cD$ is realizable with respect to $(\hyP,\hyD)$; then it could be the case that is either realizable with respect to $(\hyP,\hyD)$ or with respect to  alternative classes $(\hyP',\hyD')$ by balancing the products accordingly.}} Thus one cannot pick the estimator based on whether specific assumptions hold. 

            \paragraphit{Estimator for Scenario II.}
                Our estimator is simple: it first uses the censored samples to find $(p,\cP),(q,\cQ)\in \hyP\times \hyD$ 
                such that $p\,\cP$ approximates $p_0\,\cD_{X,Y(0)}$ and $q\,\cQ$ approximates $p_1\,\cD_{X,Y(1)}$ in the following sense:
                for a sufficiently small $\eps>0$,
                \begin{align*}
                    \begin{split}
                        \snorm{p\cP - p_0\cD_{X,Y(0)}}_1
                    &\coloneqq 
                    \iint \abs{p(x,y)\cP(x,y)-p_0(x,y)\cD_{X,Y(0)}(x,y)}\d x\d y
                     \leq \eps \,,\\
                    \snorm{q\cQ - p_1\cD_{X,Y(1)}}_1
                    &\coloneqq 
                    \iint \abs{p(x,y)\cP(x,y)-p_1(x,y)\cD_{X,Y(1)}(x,y)}\d x\d y
                     \leq \eps \,.
                    \end{split}
                    \yesnum\label{eq:techOverview:est:2}
                \end{align*}
                Then, it outputs $\widehat{\tau} = \Ex_{\cP}[y] - \Ex_{\cQ}[y]$.
                Here, $\Ex_{\cP}[y]$ estimates $\Ex_\cD[Y(1)] $ and $\Ex_{\cQ}[y]$ estimates $\Ex_\cD[Y(0)] $.
                
                {The correctness of the estimator follows because under \cref{cond:Over:estimation} (a robust version of the condition in \cref{infthm:overlap}), we show that  
                \[
                    \sabs{ \Ex\nolimits_{\cP}[y] - \Ex\nolimits_\cD[Y(1)]}\leq f(\eps)
                    \quadand
                    \sabs{ \Ex\nolimits_{\cQ}[y] - \Ex\nolimits_\cD[Y(0)]}\leq f(\eps)\,,
                \]
                where $f(\cdot)$ is a function determined by \cref{cond:Over:estimation} with the property that $\lim_{z\to 0^+} f(z) = 0$.}
               {To see that this procedure can be implemented,} note that each product $p_t\,\cD_{X,Y(t)}$ (for $t\in\{0,1\}$) is identified from the censored data. 
                {To obtain} finite-sample guarantees, we use the following standard assumptions: 
                \begin{enumerate}[itemsep=0pt]
                    \item $\hyP$ has finite fat-shattering dimension at scale $O(\eps)$,
                    \item each distribution in $\hyD$ is $O(1)$-smooth with respect to a reference measure $\mu$,
                    \item $\hyD$ admits an $O(\eps)$-TV cover.
                \end{enumerate}
               {Under these assumptions, we can construct finite covers $C_{P}$ of~$\hyP$ and $C_{D}$ of~$\hyD$, so that $C_{P}\times C_{D}$ is an $O(\eps)$‐cover of~$\hyP\times\hyD$.}
                {This, in particular, ensures that to find the pairs $(p,\cP)$ and $(q,\cQ)$ in \cref{eq:techOverview:est:2}, it suffices to select the elements of the cover $C_P\times C_D$ that are closest to $(p_0,\cD_{X,Y(0)})$ and $(p_1,\cD_{X,Y(1)})$ respectively -- as estimated from a suitably large set of samples  (see \cref{sec:nuisanceParameters}).}
                {Hence, the} estimation of $\wh{\tau}$ {reduces} to finding $(\wh{p},\wh{\cP})$ of $C_P\times C_D$ that is closest to the {empirical distribution induced by the censored samples}. 
               {We note that the size of the cover $C_P\times C_D$ is} exponential in $O(\log(\nfrac{1}{\eps})) \cdot  \mathrm{fat}_{O(\eps)}(\hyP) \cdot \log N_{O(\eps)}(\hyD)$ {and this is why we obtain the} sample complexity {claimed in} \cref{infthm:SampleComplexity1}.

                The pseudo-code of the algorithm appears in Algorithm~\ref{alg:est:scenario2}.\begin{remark}
                    {There are also other approaches to estimation in Scenario II. For instance, one could follow the template laid out in \cref{cond:iden} by (1) first, learning an estimate of $\cD_X$ and using it to eliminate any members of the cover $C_D$ that are not close to (the learned estimate of) $\cD_X$ thus resulting in a class $C_D'$, (2) then, picking the element of $C_D' \times C_P$ which is closest to the empirical estimate of $p_t(x,y) D_{X,Y(t)}(x,y)$ and, (3) outputting the resulting estimate of $\Ex[Y(t)]$. %
                    }
                    Since this approach does not improve our sample complexity, we present the more direct approach which slightly deviates from the outline in \cref{cond:iden}.
                \end{remark}

            \paragraphit{Estimator for Scenario~III.}
                In this scenario, unconfoundedness holds, but overlap is very weak: there are sets $S_0,S_1\subseteq\R^d$ with $\vol(S_0),\vol(S_1)\geq c$ such that for each $(x,y)\in S_t\times \R$, $p_t(x,y)\geq c$ (for each $t\in \zo$).
                If one has membership access to sets $S_0$ and $S_1$ and query access to the propensity scores $e(\cdot)$, then a slight modification of the Scenario~II estimator would suffice:  
                    one can find $\inparen{p,\cP}$ such that the product $p\cP$ approximates the product $p_1\cD_{X,Y(1)}$ over $S_1$, and output $\Ex_{\cP}[y]$ as an estimate for $\Ex_\cD[Y(1)]$.
                (With an analogous algorithm to estimate $\Ex_\cD[Y(0)]$.)
                The correctness of this algorithm {follows} from a robust version of the condition in \cref{infthm:3} (see \cref{cond:Uncon:estimation}) -- which guarantees that if $\cP_{S_1}$ (the truncation of $\cP$ to ${S_1}$) is close in TV distance to $\sinparen{\cD_{X,Y(1)}}_{S_1}$ (the truncation of $\cD_{X,Y(1)}$ to ${S_1}$), then their means are also close.
                However, because we neither have access to $S_0,S_1$ nor to $e(\cdot)$, we must estimate both of them from samples and carefully handle the estimation errors.
                
                Next, we describe our estimator for $\Ex_\cD[Y(1)]$, the estimator for
                $\Ex_\cD[Y(0)]$ is symmetric, {and} the estimator for ATE follows by subtracting the two. %
                Our estimator proceeds in three phases:
                \begin{enumerate}[itemsep=0pt]
                    \item First, it uses censored samples to find a propensity score $\hat{e}(\cdot)$ that approximates $e(\cdot)$. Let $\wh{S}=\inbrace{x\given \wh{e}(x)\geq c-\eps}$ which satisfies $\vol(\wh{S})\geq c-\eps$ and $\min_{x\in \wh{S}} e(x)\geq c-2\eps$.  
                    \item Then, it finds $\inparen{p,\cP}\in \hyP\times \hyD$ approximating $p_1\cdot \cD_{X,Y(1)}$ (over censored samples).
                    \item Third, it finds $\cP'$ approximating $\frac{p(x,y)\cP(x,y)}{\wh{e}(x,y)}$ over $\wh{S}$, such that,
                    \[
                        \iint_{(x,y)\in \wh{S}\times \R} \abs{\cP'(x,y)-\frac{p(x,y)\cP(x,y)}{\wh{e}(x,y)}} \d x \d y\leq O(\eps)\,.
                    \] 
                    It finally outputs $\Ex_{(x,y)\sim\cP'}[y]$ as the estimate for $\Ex_\cD[Y(1)]$.
                \end{enumerate}
                Here, as in the algorithm in Scenario II, one might be tempted to use $\Ex_{(x,y)\sim\cP}[y]$ (instead of  $\Ex_{(x,y)\sim\cP'}[y]$) as an estimate for $\Ex_\cD[Y(1)]$.
                However, this fails because $\cP$ does not approximate $\cD_{X,Y(1)}$ well in regions outside of $S_1$ -- where overlap is violated.
                This is also why Step 2 above is necessary: intuitively, in Step 2, we find a distribution $\cP'$ which approximates $\nfrac{p(x,y)\cP(x,y)}{\wh{e}(x,y)}$ over the set $\wh{S}$ -- restricting the optimization to $\wh{S}$ is important because over $\wh{S}$, it holds that $\nfrac{p(x,y)\cP(x,y)}{\wh{e}(x,y)}\approx \cP(x,y)\approx \cD_{X,Y(1)}$.
                Now, the correctness follows due to a robust version of the condition in \cref{infthm:3} which, at a high level, ensures that $\hyP'$ extrapolates and is a good estimate of $\cD_{X,Y(1)}$ over the whole domain and not just $\wh{S}.$
               {We provide the pseudo-code of our algorithm in}  Algorithm~\ref{alg:est:scenario3}.
                
                As for the previous algorithm, all the quantities estimated by this algorithm are also identifiable from the censored samples. 
                For finite sample guarantees, we use the same standard assumptions as for the previous scenario.
                As mentioned above, proving the correctness of this estimator is much more challenging than for the previous estimator and requires a careful analysis; see \cref{sec:proofof:lem:est:unconfoundedness}.

\section{Identification of ATE in Scenarios I-III}
    \label{sec:scenarios}
    In this section, we present several scenarios, including many novel ones, that satisfy \cref{cond:iden} and, hence, enable the identification of average treatment effect $\tau$.
    Later, in the upcoming \cref{sec:estimation}, we show that, under natural assumptions, $\tau$ can also be estimated from finite samples in all of these scenarios.

    \subsection{Identification under Scenario~I (Unconfoundedness and Overlap)}
    \label{section:UO}
        To gain some intuition about \cref{cond:iden}, we begin with the classical scenario where unconfoundedness and overlap both hold. 
        We verify that this scenario satisfies \cref{cond:iden}.
        Before proceeding, we note that in this scenario $\tau$ is already known to be identifiable and, under mild additional assumptions, one also has finite sample estimators for it \cite{imbens2015causal,chernozhukov2024appliedcausalinferencepowered}.
        To verify that \cref{cond:iden} is satisfied, we first need to put this scenario in the context of \cref{cond:iden} by identifying the structure of the concept classes $\hyP$ and $\hyD$.
        As mentioned in \cref{sec:framework},  an observational study $\cD$ satisfies unconfoundedness and overlap if and only if it is realizable with respect to $\hyPou(c)$ (see \cref{eq:PinScenarioI}).\footnote{To see that if $\cD$ satisfies unconfoundedness and $c$-overlap it belongs to $\hyPou(c)$ consider that $p_t(x,y_1) = \Pr[T{=}t \mid X{=}x, Y(t){=}y_1] = \Pr[T{=}t \mid X{=}x] = \Pr[T{=}t \mid X{=}x, Y(t){=}y_2] = p_t(x,y_2)$ whenever $T\bot Y(t) \mid X$ for $t\in\{0,1\}$. %
        {Next, to see that if} $\cD$ belongs to $\hyPou(c)$, {then} it satisfies unconfoundedness and $c$-overlap consider that for $t\in\{0,1\}$ $p_t(x,y) = \Pr[T{=}t\mid X{=}x]$ by the first property and, so $c$-overlap holds and 
        $\Pr\insquare{T{=}t, Y(t)\in S\mid X{=}x}
                    =\Pr[T{=}t\mid X{=}x]\cdot \int_{S}\cD_{Y(t)\mid X{=}x}(y) \d y
                    =\Pr[T{=}t\mid X{=}x]\cdot \cD_{Y(t)\mid X{=}x}(S)$, 
        \ie{}, $T\bot Y(1), Y(0) \mid X$. 
        }  
        Since unconfoundedness and overlap place no restrictions on the concept class $\hyD$, we let $\hyD$ be the set of all distributions over $\R^d\times \R$, which we denote by $\hyDall$.
        Now, we are ready to verify that unconfoundedness and overlap satisfy \cref{cond:iden}.
        \begin{restatable}[Identification in Scenario~I]{theorem}{thmScenarioOne}
            \label{lem:iden:unconfoundedness-overlap}
            For any $c \in (0,\nfrac{1}{2})$, $\inparen{\hyPou(c), \hyDall}$ satisfies \cref{cond:iden}.
        \end{restatable}
        Hence, if an observational study $\cD$ is realizable with respect to $\inparen{\hyPou, \hyDall}$, then $\tau_\cD$ can be identified.
        The proof of \cref{lem:iden:unconfoundedness-overlap} appears in \cref{sec:proofof:lem:iden:unconfoundedness-overlap}. 
    
    \subsection{Identification under Scenario~II (Overlap without Unconfoundedness)}
        \label{sec:scenario:overlap}
        \label{sec:iden:overlap}
        Next, we consider the scenario where overlap holds but unconfoundedness may not.
        Concretely, in this scenario, the generalized propensity scores belong to the following concept class. %
        \begin{lemma}
        [Structure of Class $\hyP$; Immediate from Definition]
        \label{def:pcoverlap}
            For any $c \in (0,\nfrac{1}{2}),$ an observational study $\cD$ satisfies $c$-overlap if and only if $\cD$ is realizable with respect to $\hyP = \hyPoverlap(c)$, where
             \[ 
                \hyPoverlap (c)
                \coloneqq \inbrace{p\colon \R^d\times \R \to [0,1]\given ~\forall_{x\in \R^d},~~\forall_{y\in \R}\,,~~
                    c < p(x,y) < 1-c 
                }\,.
            \]
        \end{lemma}
        This is a very weak requirement on the generalized propensity scores. Since it makes no assumptions related to unconfoundedness, it already captures the many existing models for relaxing unconfoundedness in the literature \cite{tan2006distributional,rosenbaum2002observational,rosenbaum1987sensitivity,kallus2021minimax}.
        \begin{itemize}
            \item For instance, it captures \citet{tan2006distributional}'s model which requires that, {apart from $c$-overlap}, the propensity scores satisfy the following bound for some $\Lambda\geq 1$:
            \[
                \forall_{x\in \R^d}\,~~ 
                    \forall_{y\in \R}\,,~~ 
                        \forall_{t\in \zo}\,,\qquad 
                    \frac{1}{\Lambda} \leq {\frac{\inparen{1-e(x)}p_0(x,y)}{e(x) \inparen{1-p_0(x,y)}}, \frac{e(x)\inparen{1-p_1(x,y)}}{\inparen{1-e(x)}p_1(x,y)}}
                \leq \Lambda\,.
                \yesnum\label{eq:ratioOfTan}
            \]
            (Where $e(x)$ is the usual propensity score defined as $e(x)=\Pr[T{=}1|X{=}x]$.)
            The scenario we consider is strictly weaker since we only require $c$-overlap and not the above condition.
            We note that both \citet{tan2006distributional,rosenbaum2002observational} also assume overlap for an unspecified constant $c > 0$ as their focus is \textit{not} on getting sample complexity bounds, \ie{}, they assume $e(x)\in (c,1-c)$ for each $x$. 
            {(At first, this may seem weaker than overlap for generalized propensity scores. However, \cref{eq:ratioOfTan} along with $c$-overlap for $e(\cdot)$, implies $\Omega(\nfrac{c}{\Lambda})$-overlap for generalized propensity scores.)}
            To get sample complexity bounds for any standard estimator one either requires $c$-overlap (for, \eg{}, inverse propensity score weighted estimators) or additional assumptions (for, \eg{}, outcome-regression-based estimators).
            \item As another example, $\hyPoverlap(c)$ also captures the seminal odds ratio model that was formalized and extensively studied by Rosenbaum \cite{rosenbaum1987sensitivity,rosenbaum1991sensitivity,rosenbaum1988sensitivityMultiple}, and has since been utilized in a number of studies for conducting sensitivity analysis; see \citet{lin1998assessing,rosenbaum2002observational} and the references therein. 
            This model also aims to relax unconfoundedness.
            In addition to $c$-overlap, the odds ratio model places the following constraint for some $\Gamma \geq 1$:\footnote{To be precise, Rosenbaum assumes propensity score $e(x, i)$ can differ for different individuals $i$ with the same covariates $x$, but does not specify the reason for the differences \citet{rosenbaum2002observational}. Here, we study differences due to differences in the outcomes of individuals $i$ as also studied by \citet{kallus2021minimax}.}
            \[
                \forall_{x\in \R^d}\,~~ 
                    \forall_{y_1,y_2\in \R}\,,~~ 
                        \forall_{t\in \zo}\,,\qquad 
                \frac{1}{\Gamma}
                \leq 
                \frac{p_t(x,y_1)\inparen{1-p_t(x,y_2)}}{p_t(x,y_2)\inparen{1-p_t(x,y_1)}}
                \leq \Gamma\,.
            \]
            Again, we capture this model since we only require $c$-overlap without the above condition.
            {Like \citet{tan2006distributional}, \citet{rosenbaum2002observational} assumes overlap for an unspecified constant $c>0$, as their focus is not bounding the sample complexity; to get sample complexity bounds, we need to use either $c$-overlap or other assumptions.}
        \end{itemize}
        Under the scenario we consider we can ensure that $\Gamma=\Lambda=O(\nfrac{(1-c)^2}{c^2})>1$.
        However, as noted by \citet{rosenbaum2002observational,tan2006distributional}, if $\Gamma, \Lambda > 1$, then without distributional assumptions, $\tau$ cannot be identified up to factors better than $O(\Gamma)$ and $O(\Lambda)$ respectively.
        Hence, based on earlier results, it is not clear when $\tau$ can be identified. 
        Our main result in this section is a characterization of the concept class $\hyD$ that enables identifiability in the above scenario -- where overlap holds but unconfoundedness may not.
        Its proof appears in \cref{sec:proofof:lem:iden:overlap}.
        \begin{restatable}[Characterization of Identification in Scenario~II]{theorem}{thmScenarioTwo}
            \label{lem:iden:overlap}
            The following hold:
        \begin{enumerate}[itemsep=-1pt]
            \item \textbf{(Sufficiency)}\quad  If~ $\hyD$ satisfies \cref{cond:Over}, then there is a mapping $f\colon  \Delta(\R^d \times \{0,1\} \times \R) \to \R$ with $f(\cC_\cD)=\tau_\cD$ for each distribution $\cD$ realizable with respect to $\inparen{\hyPoverlap(c),\hyD}$.
            \item \textbf{(Necessity)}\quad Otherwise, for any map $f\colon \Delta(\R^d \times \{0,1\} \times \R) \to \R$, there exists a  distribution $\cD$ realizable with respect to $\inparen{\hyPoverlap(c),\hyD}$ such that $f(\cC_\cD)\neq \tau_\cD$.
        \end{enumerate}    
        \end{restatable}
        \vspace{-5mm}
        \begin{restatable}[Structure of Class $\hyD$]{condition}{conditionScenarioTwo}
            \label{cond:Over}
            Given a constant $c>0$, the class of distributions $\hyD$ over $(X,Y)$ is said to satisfy \cref{cond:Over} with constant $c$ if for each $\cP,\cQ\in \hyD$ with $\Ex_{(x,y)\sim\cP}[y] \neq \Ex_{(x,y)\sim\cQ}[y]$,
            either
            \begin{enumerate}[itemsep=-1pt]
                \item the marginals of $\cP$ and $\cQ$ over $X$ are different, \ie{}, $\cP_X \neq \cQ_X$ or
                \item there exist some $x\in \supp(\cP_X)$ and $y\in \R$, such that $
            \nfrac{\cP(x,y)}{\cQ(x,y)}
            \notin 
            \inparen{
                    \nfrac{c}{(1-c)}, \nfrac{(1-c)}{c}}.
               $
            \end{enumerate}            
        \end{restatable}
This condition is similar to \cref{cond:iden}.
            Each tuple $(p,\cP)$ corresponds to some propensity score $p_t(\cdot)$ and distribution $\cD_{X,Y(t)}$ for some $t\in \zo$. 
            The above condition ensures that any two tuples that lead to different guesses for $\tau$, are distinguishable from the available samples.
                This is because of two reasons.
                First, as before, the marginal $\cD_X$ can be identified from data and, hence, all distributions $\cP$ with $\cP_X\neq \cD_X$ can be eliminated.
                Now, all remaining distributions have the same marginal over $X$.
                Since any two propensity scores $p,q$, their ratio $\nfrac{p(x,y)}{q(x,y)}\in \inparen{\nfrac{c}{(1-c)}, \nfrac{(1-c)}{c}}$ (due to $c$-overlap).
                The above condition ensures that $p(x,y)\cdot \cP(x,y)\neq q(x,y)\cdot \cQ(x,y)$ for some $x,y$ {enabling} us to distinguish $\inparen{p,\cP}$ and $\inparen{q,\cQ}$ as in \cref{cond:iden}.

             The above result is valuable because a number of common distribution families, including the Gaussian distributions, Pareto distributions, and Laplace distributions, can be shown to satisfy \Cref{cond:Over} (for any $c>0$).
            Hence, the above characterization shows that overlap alone already enables identifiability for many distribution families. 
            A specific, interesting, and practically relevant example captured by this condition is generalized linear models (GLMs):
                in this setting, for each $t\in \zo$, $Y(t)=\mu_t(x)+\xi_t$ for some function $\mu_t(\cdot)$ and noise $\xi_t\sim \cN(0,1)$.

    \subsection{Identification under Scenario~III (Unconfoundedness without Overlap)}    
        \label{sec:scenario:unconfoundedness}
        \label{sec:iden:unconfoundedness}
        {Next, we consider the scenario where unconfoundedness holds but overlap may not. 
        Without further assumptions, this includes the extreme cases where either no one receives the treatment or everyone receives the treatment, \ie{}, for any $t\in \zo$,
        \[
            \forall_{x\in \R^d}\,,~~
                \forall_{y\in \R}\,,\quad
                    p_t(x,y) = 0 
            \qquadtext{or}
            \forall_{x\in \R^d}\,,~~
                \forall_{y\in \R}\,,\quad
                    p_t(x,y) = 1\,.
        \]
        Clearly, in these cases, identifying ATE is impossible. 
        To avoid these extreme cases, we assume that at least some non-trivial set of covariates satisfies overlap. 
        A natural way to satisfy this is to require that there is some set $S$ of covariates with $\vol(S)\geq \Omega(1)$ such that for each $(x,y)\in S\times \R$ overlap holds, \ie{}, $c < p_0(x,y), p_1(x,y) < 1-c$.
        This requirement is already significantly weaker than $c$-overlap which requires $c<p_0(x,y),p_1(x,y)<1-c$ to hold \textit{pointwise} for each $(x,y)\in \R^d\times \R$.
        We make an even weaker requirement, which we call \textit{$c$-weak-overlap} that removes the lower bound on $p_0(\cdot)$ and $p_1(\cdot)$:
        \begin{definition}[$c$-weak-overlap]\label{def:weakoverlap}
            Given $c\in (0,\nfrac{1}{2})$, 
            the observational study $\cD$ is said to satisfy $c$-weak-overlap if, for each $t\in \zo$, there exists a set $S_t\subseteq \R^d$ with $\vol(S_t)\geq c$ such that
            \[
                \forall_{(x,y)\in S_t\times \R}\,,~~\quad
                    p_t(x,y) > c\,.
            \] 
        \end{definition}
        The following class encodes the resulting scenario.
        \begin{lemma}
        [Structure of Class $\hyP$]
        \label{def:classPu}
        For any $c \in (0,\nfrac{1}{2})$, an observational study $\cD$ satisfies unconfoundedness with $c$-weak overlap if and only if $\cD$ is realizable with respect to $\hyP = \hyPunconf(c)$, where 
            \[
                \hyPunconf(c) \coloneqq 
                \inbrace{
                    p\colon \R^d\times \R\to [0,1]\given  
                    \begin{array}{c}
                         ~\forall_{x\in \R^d},~~\forall_{y_1,y_2\in \R}\,,\quad 
                                p(x,y_1)=p(x,y_2)\,,\\
                                \exists S~~\text{with}~~\vol(S)\geq c ~~\text{such that,}~~ \forall_{(x,y)\in S\times \R}\,,~~ p(x,y) > c 
                    \end{array}
                }\,.
            \]
        \end{lemma} 
        Two remarks are in order.
        First, to simplify the notation, we use the same constant $c$ to denote the lower bound on $\vol(S)$ and the values of $p(\cdot)$.
        One can extend our results to use different constants $c_S,c_p>0$.
        Second, for the above guarantee to be meaningful, the set $S$ must be a subset of $\supp(\cD_X)$; otherwise, the propensity scores could be 0 for each $x\in \supp(\cD_X)$ or 1 for each $x\in \supp(\cD_X)$, returning us to the extreme cases described above where ATE is clearly not identifiable.
        To ensure that this is always the case, in this section, we  make the simplifying assumption $\supp(\cD_X)=\R^d$ and, hence, also assume for each $\hyP\in \hyD$, $\supp(\cP_X)=\R^d$ (otherwise, we can remove $\cP$ from $\hyD$).}

        The identification and estimation methods we develop in this scenario are relevant to many well-studied topics in causal inference.
        \begin{itemize}
            \item First, this scenario captures the regression discontinuity designs -- where propensity scores violate overlap for a large fraction of individuals but unconfoundedness holds -- which have found wide applicability \cite{hahn2001regressionDiscontinuity,thistlethwaite1960regressionDiscontinuity,imbens2008regressionDiscontinuity,angrist2009mostlyHarmless,lee2010regressionDiscontinuity}.
            (Also see the more extensive discussion on RD designs at the end of this section). 
            To the best of our knowledge, in RD designs, ATE is only known to be identifiable under strong linearity assumptions whereas we are able to achieve identification under much weaker restrictions.
            \item Second, most standard estimators of ATE are based on inverse propensity score weighting (IPW). 
            IPW estimators require overlap and unconfoundedness to identify $\tau$. 
            These estimators, however, are fragile: their error scales with $\sup_x\nfrac{1}{\inparen{e(x)\inparen{1-e(x)}}}$ \cite{li2018overlapWeights,crump2009dealing,imbens2015causal,kalavasis2024cipw,khan2024trimming}.
            In particular, this quantity can be arbitrarily large even when the (usual) propensity score $e(\cdot)$ violates overlap for a single covariate $x$ \cite{kalavasis2024cipw}, as is bound to arise in high-dimensional data \cite{damour2021highDimensional}.
            In contrast to such estimators, the estimators we design can identify and estimate ATE even when propensity scores are violated for a large fraction of the covariates. 
            Moreover, while our estimators do rely on certain distributional assumptions, these distributional assumptions are satisfied for standard models, \eg{}, when the conditional outcome distributions follow a linear regression or polynomial regression model.
        \end{itemize}
        Next, we present the class of conditional outcome distributions that, together with the propensity scores in \cref{def:classPu}, characterize the identifiability of $\tau$.
        \begin{restatable}[Structure of Class $\hyD$]{condition}{conditionScenarioThree}
            \label{cond:Uncon}
            Given $c\in (0,\nfrac{1}{2})$, a class $\hyD$ is said to satisfy \cref{cond:Uncon} if for each $\cP,\cQ\in \hyD$ with $\Ex_{(x,y)\sim \cP}[y]\neq \Ex_{(x,y)\sim \cQ}[y]$ either
            \begin{enumerate}[itemsep=-2pt]
                \item the marginals of $\cP$ and $\cQ$ over $X$ are different, \ie{}, $\cP_X \neq \cQ_X$, or
                \item there is no $S\subseteq \R^d$ with $\vol(S)\geq c$ such that $\cP(x,y)=\cQ(x,y)$ for each $(x,y)\in S\times \R$.
            \end{enumerate}
        \end{restatable} 
        As for the other conditions we discussed so far, this condition allows us to distinguish any pair of tuples $(p,\cP)$ and $(q,\cQ)$ that lead to a different prediction for $\tau$.
        The requirement for the marginal of $\cP$ and $\cQ$ over $X$ to match is the same as in \cref{cond:iden,cond:Over}, so let us consider the second requirement.
        It requires the pairs $\cP,\cQ$ to be distinguishable on any set of the form $S\times \R$ where $S$ is a full-dimensional set.
        In other words, any $\cP$ and $\cQ$ (with $\cP_X=\cQ_X$) whose \textit{truncations} to the set $S\times \R$ are identical must also have the same \textit{untruncated} means.
        Roughly speaking, this condition holds for any family $\hyD$ whose elements $\cP$ can be extrapolated given samples from their truncations to full-dimensional sets.
        While this might seem like a strong requirement at first, it is satisfied by many families of parametric densities:
        For instance, using Taylor's theorem, one can show that it holds for distributions of the form $\propto e^{f(x,y)}$ for any polynomial $f(x,y)$ (see \cref{lem:unconfoundedness:identifiableFamilies}).
        This already includes several exponential families, including the Gaussian family.

        Now, we are ready to state the main result of this section: a characterization of when $\tau$ is identifiable under unconfoundedness when overlap may not hold.
        The proof of this result appears in \cref{sec:proofof:lem:iden:unconfoundedness}.
        \begin{restatable}[Characterization of Identification in Scenario~III]{theorem}{thmScenarioThree}
            \label{lem:iden:unconfoundedness}
            {Fix any $\hyD$ such that each $\cP\in \hyD$ satisfies $\supp(\cP_X)=\R^d$.}
            The following hold:
            \begin{enumerate}[itemsep=-1pt]
                \item \textbf{(Sufficiency)}\quad  If ~$\hyD$ satisfies \cref{cond:Uncon}, then there is a mapping $f\colon  \Delta(\R^d \times \{0,1\} \times \R) \to \R$ with $f(\cC_\cD)=\tau_\cD$ for each distribution $\cD$ realizable with respect to $\inparen{\hyPunconf(c),\hyD}$.
                \item \textbf{(Necessity)}\quad Otherwise, for any map $f\colon \Delta(\R^d \times \{0,1\} \times \R) \to \R$, there exists a  distribution $\cD$ realizable with respect to $\inparen{\hyPunconf(c),\hyD}$ such that $f(\cC_\cD)\neq \tau_\cD$. 
            \end{enumerate}    %
        \end{restatable}
        {The requirement that $\supp(\cP_X)=\R^d$ for each $\cP\in \hyD$, in particular, ensures that $\supp(\cD_X)=\R^d$, which is necessary to ensure that the definition of $c$-weak-overlap is meaningful.
        Recall that if it does not hold and one can select a set $S$ with $\vol(S)>c$ disjoint from $\supp(\cD_X)$, then one can satisfy $c$-weak-overlap even in cases where no one receives the treatment or everyone receives the treatment, where ATE is clearly non-identifiable.
        That said, we note that the above result can be generalized to require $\supp(\cP_X)=K$ for any full-dimensional set $K$.}

        Our next result presents several examples of families of distributions that satisfy \cref{cond:Uncon}. 
        \begin{restatable}[]{lemma}{lenUnconfoundednessIdentifiableFamilies}
            \label{lem:unconfoundedness:identifiableFamilies}       
            \label{lem:iden:unconfoundedness:examples}
            The following concept classes $\hyD \subseteq \Delta(\R^d \times \R)$ satisfy \cref{cond:Uncon}: %
            \begin{enumerate}[itemsep=-2pt,leftmargin=17.5pt]
                \item \textbf{(Polynomial Log-Densities)}\quad  
                    Each element $\cP$ of this family can have an arbitrary marginal over $X$ and, for each $x$, the conditional distribution $\cP(y\mid x)$ is parameterized by a polynomial $f=f_\cP$ as %
                    \[
                        \cP(y|x) \propto e^{f(x,y)}\,.
                    \]
                \item \textbf{(Polynomial Expectations)}\quad 
                    Each element $\cP$ of this family can have an arbitrary marginal over $X$ and, for each $x$, the conditional distribution $\cP(y\mid x)$ satisfies the following for some polynomial $f=f_\cP$ %
                    \[
                            \Ex\nolimits_{(x,y)\sim \cP}[y|X{=}x] = f(x)\,.
                    \]
            \end{enumerate} 
        \end{restatable}
        These distribution families capture a broad range of parametric assumptions commonly used in causal inference. 
        The polynomial log-density framework includes widely applied exponential families, such as Gaussian outcome models with arbitrary distributions over covariates $X$. 
        The second family allows for polynomial conditional expectations, covering popular linear and polynomial regressions \cite{chernozhukov2024appliedcausalinferencepowered}.  
        Both families leave the marginal distribution of $X$ unrestricted, allowing for rich covariate distributions while ensuring identifiability under the present scenario.
        The proof of \cref{lem:unconfoundedness:identifiableFamilies} appears in \cref{sec:proofof:lem:unconfoundedness:identifiableFamilies}.

    \paragraph{Regression Discontinuity Design.} As a concrete application of Scenario~III,
        we study regression discontinuity (RD) designs \cite{hahn2001regressionDiscontinuity,thistlethwaite1960regressionDiscontinuity,imbens2008regressionDiscontinuity,lee2010regressionDiscontinuity,rubin1977regressionDiscontinuity,sacks1978regressionDiscontinuity} which were introduced by and studied in several disciplines \cite{rubin1977regressionDiscontinuity,sacks1978regressionDiscontinuity,goldberger1972selection} (see \cite{cook2008waitingforLife} for an overview), and have found applications in various contexts from Education \cite{thistlethwaite1960regressionDiscontinuity,angrist1999classSizeRD,klaauw2002regressionDiscontinuityEnrollment,black1999regressionDiscontinuity}, to Public Health \cite{moscoe2015rdPublicHealth}, to Labor Economics \cite{lee2010regressionDiscontinuity}.
        In an RD design, the treatment assignment is a known deterministic function of the covariates.
        \defRD*
        \noindent Since the treatment assignment is only a function of the covariates and not the outcomes, unconfoundedness is immediately satisfied.
        However, overlap may fail since any covariate $x$ outside of the treatment set $S$ does not receive treatment, while individuals within $S$ always receive treatment. 
        To avoid degenerate cases in which the entire population is treated (or untreated), we require the treatment set $S$ and its complement to have a positive volume. Under these conditions, RD designs become a special case of Scenario~III, where the generalized propensity scores lie in $\hyPunconf(c)$.
        The following corollary of \cref{lem:iden:unconfoundedness} shows that ATE can be identified in any RD design.

        \begin{restatable}[Identification is Possible]{corollary}{thmIdenRD}
            \label{lem:iden:rdDesign} 
            Fix any $c\in (0,\nfrac{1}{2)}$, set $S\subseteq \R^d$, and class $\hyD$ satisfying \cref{cond:Uncon}. 
            Then, there exists a mapping $f$ with $f(\cC_\cD)=\tau_\cD$
            for each $c$-RD-design $\cD$ that is realizable with respect to $\hyD$.
        \end{restatable} 
        To the best of our knowledge, all results for identifying ATE in RD assume linear outcome regressions, \ie{}, that $\Ex[Y(t)\mid X{=}x]$ is a linear function of $x$ (for each $t\in \zo$). 
        \cref{lem:iden:rdDesign} substantially broadens these assumptions and is applicable in very general and practical models where $\Ex[Y(t)\mid X{=}x]$ are polynomial functions of $x$ and the distribution of covariates is arbitrary; see \cref{lem:iden:unconfoundedness:examples} for a proof.

\section{Estimation of ATE in Scenarios I-III}\label{sec:estimation}
    In this section, we study the estimation of the average treatment effect $\tau$ from finite samples in the scenarios presented in \cref{sec:scenarios}. We show that, under mild additional assumptions, the estimation of the ATE is possible in all of them.
     \subsection{Estimation under Scenario~I (Unconfoundedness and Overlap)}
    \label{sec:est:unconfoundedness-overlap}
        We begin with estimating ATE under the classical assumptions of unconfoundedness and $c$-overlap.
        As mentioned before, given access to propensity scores, estimators for ATE are already known in this scenario \cite{imbens2015causal,chernozhukov2024appliedcausalinferencepowered}.
        For completeness, we prove ATE's end-to-end estimability (the proof appears in \cref{sec:proofof:lem:est:unconfoundedness-overlap}).
        \begin{restatable}[Estimation of ATE under Scenario~I]{theorem}{thmScenarioOneEstimation}
            \label{lem:est:unconfoundedness-overlap}
            Fix constants $c\in (0,\nfrac{1}{2})$, $B>0$, $\eps,\delta\in (0,1)$.
            Let concept classes $\hyP\subseteq\hyPou(c)$ and $\hyD$ satisfy:
            \begin{enumerate}[itemsep=-1pt]
                \item  $\hyP$ has a finite fat-shattering dimension (\cref{def:fat-shattering-dimension}) $\mathrm{fat}_{\gamma}(\hyP)<\infty$ at scale $\gamma=\Theta(\nfrac{c^2\eps}{B})$; 
                \item Each $\cP\in \hyD$ has support $\supp(\cP)\subseteq [-B,B]$.
            \end{enumerate}
            There is an algorithm that, given $n$ \iid{} samples from the censored distribution $\cC_\cD$ for any $\cD$ realizable with respect to $\inparen{\hyP,\hyD}$ and $\eps,\delta \in (0,1)$,
            outputs an estimate $\hat{\tau}$, such that, with probability $1-\delta$,
            \[
                \abs{\hat{\tau}-\tau_{\cD}}\leq \eps\,.
            \]
            There is a universal constant $\eta\geq \nfrac{1}{256}$, such that, the number of samples $n$ is 
            \[
                n = 
                O\inparen{
                    \frac{B^2}{c^4\eps^2} 
                    \cdot
                    \inparen{
                        \mathrm{fat}_{\eta c^2\eps/B}(\hyP)\cdot\log(\frac{B}{c^2\eps})
                        +\log(\nfrac{1}{\delta})}}
                \,.
            \] 
        \end{restatable} 
        The assumption on the range of the outcomes being bounded is standard in the causal inference literature when one aims to get sample complexity (\eg{}, \citet{kallus2021minimax}), and the bound on the fat-shattering dimension is expected because of the reduction to probabilistic concepts from statistical learning theory \cite{alon1997scale}.
    
    \subsection{Estimation under Scenario~II (Overlap without Unconfoundedness)}
        \label{sec:est:overlap}
        Next, we estimate ATE in Scenario~II where $c$-overlap holds but unconfoundedness does not.
        In \cref{lem:iden:overlap}, we characterized the identifiability of ATE under this scenario: ATE was identifiable if and only if the class $\hyD$ satisfied \cref{cond:Over}.
        To estimate ATE, we need the following quantitative version of \cref{cond:Over}.
        \begin{restatable}[Estimation Condition for Scenario II]{condition}{conditionScenarioTwoEstimation}
            \label{cond:Over:estimation}
            Let $\eps > 0$ be an accuracy parameter.
            The class $\hyD$ satisfies \cref{cond:Over:estimation} with mass function $M\colon (0, \infty) \to [0, 1]$ if, for any $\cP, \cQ \in \hyD$ with  
            \[
            \abs{\Ex\nolimits_{(x,y) \sim \cP}[y] - \Ex\nolimits_{(x,y) \sim \cQ}[y]} > \eps\,,  
            \]  
            there exists a set $S \subseteq \R^d \times \R$ with $\cP(S), \cQ(S) \geq \nfrac{M(\eps)}{c}$ such that  
            \[
                \forall_{(x, y) \in S}\,,\qquad \frac{\cP(x, y)}{\cQ(x, y)} \notin \inparen{\frac{c}{2(1-c)}, \frac{2(1-c)}{c}}\,.
                \yesnum\label{cond:Over:estimation:req}
            \]  
        \end{restatable}
        \cref{cond:Over:estimation} and \cref{cond:Over} differ in two key aspects:
        First, \cref{cond:Over:estimation} scales the bounds on the ratio between any pair of distributions $\cP, \cQ \in \hyD$ by a factor of $2$.  
        This factor is arbitrary and can be replaced by any constant strictly greater than 1.  
        The crucial aspect of \cref{cond:Over:estimation} is that the bound on the ratio of densities holds not just at a single point but on a set $S$ with non-trivial probability mass.  
        Intuitively, this ensures that differences between distributions can be detected using finite samples, allowing us to correctly identify the underlying distribution.  
        In the next result, we formalize this intuition, showing that the sample complexity naturally depends on the mass function $M(\cdot)$.
  
        \begin{restatable}[Estimation of ATE under Scenario~II]{theorem}{thmScenarioTwoEstimation}
            \label{lem:est:overlap}
            Fix constants $c\in (0,\nfrac{1}{2})$, $\sigma,\eps,\delta\in (0,1)$, and a distribution $\mu$ over $\R^d\times \R$.
            Let concept classes $\hyP\subseteq\hyPoverlap(c)$ and $\hyD$ satisfy:
            \begin{enumerate}[itemsep=-1pt] 
                \item $\hyD$ satisfies \cref{cond:Over:estimation} with mass function $M(\cdot)$.
                \item Each $\cP\in \hyD$ is $\sigma$-smooth with respect to $\mu$.\footnote{Distribution $\cP$ is said to be $\sigma$-smooth with respect to $\mu$ if its probability density function $p$ satisfies $p(\cdot)\leq \inparen{\nfrac{1}{\sigma}}\cdot \mu(\cdot)$.}
                \item  $\hyP$ has a finite fat-shattering dimension (\cref{def:fat-shattering-dimension}) $\mathrm{fat}_{\gamma}(\hyP)<\infty$ at scale $\gamma=\Theta({c}\sigma M(\nfrac{\eps}{2}))$;  
                \item $\hyD$ has a finite covering number with respect to total variation distance $N_{O({c}M(\nfrac{\eps}{2}))}(\hyD)<\infty$.
            \end{enumerate}
            There is an algorithm that, given $n$ \iid{} samples from the censored distribution $\cC_\cD$ for any $\cD$ realizable with respect to $\inparen{\hyP,\hyD}$ and $\eps,\delta \in (0,1)$,
            outputs an estimate $\hat{\tau}$, such that, with probability $1-\delta$,
            \[
                \abs{\hat{\tau}-\tau_{\cD}}\leq \eps\,.
            \]
            There is a universal constant $\eta\geq \nfrac{1}{256}$, such that, the number of samples $n$ is 
            \[
                n = 
                O\inparen{
                    \frac{1}{\eta M(\nfrac{\eps}{2})^2}
                    \cdot 
                    \inparen{
                        \mathrm{fat}_{\eta c\sigma M(\nfrac{\eps}{2})}(\hyP)\cdot 
                            \log{\frac{1}{\eta {c}\sigma M(\nfrac{\eps}{2})}}
                        +\log{\frac{N_{\eta {c}M(\nfrac{\eps}{2})}(\hyD)}{\delta}}
                    }
                }\,.
            \]   
        \end{restatable}
        The proof of \cref{lem:est:overlap} appears in \cref{sec:proofof:lem:est:overlap}.

   \begin{algorithm}[tbh!]
             \begin{curvybox}
                    \DontPrintSemicolon
                    \small
                
                    \KwIn{Classes $(\hyP,\hyD)$, 
                        $\eps,\delta \in (0,1)$, 
                        access to $M(\cdot)$, and 
                        i.i.d.\ censored samples $\dataset=\inbrace{c_1,c_2,\dots}$}  
                
                    \SetKwFunction{FEstimateATE}{\textsf{Estimate~ATE~in~Scenario~II}}
                    \SetKwProg{Fn}{Function}{:}{}
                
                    \Fn{\FEstimateATE{$(\hyP,\hyD), \,\dataset, \, M(\cdot), \,\eps,  \,\delta$}}{

                        \vspace{2mm}
                         
                        Use the censored samples $\dataset$, to find $(p,\cP),(q,\cQ)\in \hyP\times\hyD$, such that, with probability at least $1-\delta$,
                        \[
                            \norm{p\cP - p_1\cD_{X,Y(1)}}_1 \;\le\; M(O(\eps))
                            \quadand
                            \norm{q\cQ - p_0\cD_{X,Y(0)}}_1 \;\le\; M(O(\eps))\,.
                        \]
                        \noindent \hspace{-3mm} $\#$ \textit{Where we define the $L_1$-norm between $\alpha(\cdot)$ and $\beta(\cdot)$ as $\|\alpha - \beta\|_1 \coloneqq \iint \bigl|\alpha(x,y) - \beta(x,y)\bigr|\;\d x\d y$\;}

                        \vspace{2mm}
                
                        Define the estimator $\wh{\tau} \;\gets\; \Ex_{\cP}[y] \;-\; \Ex_{\cQ}[y]$
                        
                        \vspace{2mm}
                        
                        \Return $\wh{\tau}$  ~ $\#$ \textit{which is an estimate of $\tau$} \;
                    }
                    \vspace{4mm}  
                    \caption{Algorithm to estimate ATE in Scenario II. \label{alg:est:scenario2}}
                    \end{curvybox}
                \end{algorithm}

        \noindent\textbf{Proof Sketch of \cref{lem:est:overlap}.} The argument proceeds in two steps. 
        
        \noindent\textit{Construction of estimator $\wh{\tau}$.}
            At a high level, the assumptions on $\hyP$ and $\hyD$ enable one to create a cover of $\hyP\times \hyD$ {with respect to the $L_1$-norm}. ({Where, we define the $L_1$-norm between $\alpha(x,y)$ and $\beta(x,y)$ as $\|\alpha - \beta\|_1 \coloneqq \iint \bigl|\alpha(x,y) - \beta(x,y)\bigr|\;\d x\d y.$)} %
            This, in turn, is sufficient to get $(p,\cP)$ and $(q,\cQ)$ such that the products $p\cP$ and $q\cQ$ are good approximations for the products $p_1\cD_{X,Y(1)}$ and $p_0\cD_{X,Y(0)}$ respectively.
            Concretely, they satisfy the following guarantee
            \[
                \norm{p_1\cD_{X,Y(1)} - p\cP}_1 < M(O(\eps))
                \qquad\text{ and }\qquad
                \norm{p_0\cD_{X,Y(0)} - q\cQ}_1 < M(O(\eps))
                \,,
                \yesnum\label{sec:est:overlap:guarantee}
            \]
            where we define the $L_1$-norm between $\alpha(x,y)$ and $\beta(x,y)$ as $\|\alpha - \beta\|_1 \coloneqq \iint \bigl|\alpha(x,y) - \beta(x,y)\bigr|\;\d x\d y.$
            We present the details of constructing the cover and finding the tuples $(p,\cP)$ and $(q,\cQ)$ using finite samples in \cref{sec:nuisanceParameters}.
            We then define our estimator as
            \[
                \hat{\tau}
                \;\;=\;\;
                \abs{\Ex\nolimits_{(x,y)\sim\cP}[y] - \Ex\nolimits_{(x,y)\sim\cQ}[y]}\,.
            \]

            \noindent\textit{Proof of accuracy of $\wh{\tau}$.}
                Due to \cref{cond:Over:estimation} and the fact that all elements of $\hyP$ satisfy overlap, if $\Ex_{\cD_{X,Y(1)}}[y]$ is $\eps$-far from $\Ex_{\cP}[y]$, then $\cD_{X,Y(1)}(x,y)/\cP(x,y)$ must be very large or very small (concretely, outside $\inparen{\nfrac{c}{2(1-c)}, \nfrac{2(1-c)}{c}}$) for each $(x,y)\in S$ where $S$ is a set with measure at least $M(\eps)$ under $\cP$ and $\cQ$.  
                Because $p_1,p\in \hyPoverlap$, their ratios are bounded and always lie in $\inparen{\nfrac{c}{(1-c)}, \nfrac{(1-c)}{c}}$.

                Our proof relies on the following observation: intuitively,  \cref{cond:Over:estimation} forces any two distributions, say $\cP$ and $\cD_{X,Y(1)}$ in $\hyD$, with a large difference in mean-outcomes to have a large (multiplicative) difference in their densities over a set of measure at least $M(\eps)$. 
                Concretely, if $\sabs{\Ex_{\cD_{X,Y(1)}}[y]-\Ex_{\cP}[y]}\geq O(\eps)$, then $\cD_{X,Y(1)}(x,y)/\cP(x,y)\not\in \inparen{\nfrac{c}{2(1-c)}, \nfrac{2(1-c)}{c}}$ on at least a set $S$ of mass $M(\eps)$ under both $\cP$ and $\cD_{X,Y(1)}$.
                Further, the ratios of propensity scores $p(\cdot)$ and $p_1(\cdot)$ are bounded between $\inparen{\nfrac{c}{(1-c)}, \nfrac{(1-c)}{c}}$.
                The combination of these facts allows one to show that if $\sabs{\Ex_{\cD_{X, Y(1)}}[y]-\Ex_{\cP}[y]}\geq O(\eps)$, then 
                \[
                    \norm{p_1\cD_{X,Y(1)} - p\cP} > M(O(\eps))\,,
                \]
                which contradicts the guarantee in \cref{sec:est:overlap:guarantee}.
                Thus, due to the contradiction, one can conclude that $\sabs{\Ex_{\cD_{X,Y(1)}}[y]-\Ex_{\cP}[y]}\leq O(\eps)$.
                An analogous proof shows $\sabs{\Ex_{\cD_{X,Y(0)}}[y]-\Ex_{\cQ}[y]}\leq O(\eps)$.
                Together, these are sufficient to conclude the proof.

    \subsection{Estimation under Scenario~III (Unconfoundedness without Overlap)} 
    \label{sec:est:unconfoundedness}
        Next, we study estimation under Scenario~III, where unconfoundedness is guaranteed but overlap is not.
        Recall that this scenario is  captured by the following class of propensity scores.
        \[
            \hyPunconf(c) \coloneqq 
            \inbrace{
                p\colon \R^d\times \R\to [0,1]\given  
                \begin{array}{c}
                     ~\forall_{x\in \R^d},~~\forall_{y_1,y_2\in \R}\,,\quad 
                            p(x,y_1)=p(x,y_2)\,,\\
                            \exists S~~\text{with}~~\vol(S)\geq c ~~\text{such that,}~~ \forall_{(x,y)\in S\times \R}\,,~~ p(x,y) > c 
                \end{array}
            }\,.
        \]
        In \cref{lem:iden:unconfoundedness}, we showed that, in this case, the identifiability of $\tau$ is characterized by \cref{cond:Uncon}.
        In this section, we show that $\tau$ can be estimated from finite samples under the following quantitative version of \cref{cond:Uncon}.
        \begin{restatable}[]{condition}{conditionScenarioThreeEstimation}
            \label{cond:Uncon:estimation}
            Given $c,C>0$, a class $\hyD$ is said to satisfy \cref{cond:Uncon:estimation}  with constants $c,C$ if for each $\cP,\cQ\in \hyD$ and set $S\subseteq\R^d$ with $\vol(S)>c$, the following holds: for each $\eps>0$
            \[
                \qquadtext{if} \tv{\cP\sinparen{S\times \R}}{\cQ\sinparen{S\times \R}} \leq \eps\,,
                \qquadtext{then}
                \abs{\Ex\nolimits_{\cP}[y] - \Ex\nolimits_{\cQ}[y]}\leq \eps\cdot C\,.
            \]
            Where distributions $\cP\sinparen{S\times \R}$ and $\cQ\sinparen{S\times \R}$ are the truncations of $\cP$ and $\cQ$ to $S\times \R$ defined as follows:
            for each $(x,y)$, $\cP\sinparen{S\times \R; x, y}\propto \mathds{1}\inbrace{x\in S}\cdot \nfrac{\cP(x,y)}{\cP(S\times \R)}$ and analogously for $\cQ.$     
        \end{restatable}
        To gain some intuition, fix a set $S$.
        Now, the above condition holds if whenever the truncated distributions $\cP\sinparen{S\times \R}$ and $\cQ\sinparen{S\times \R}$ are close, then the means of the \textit{untruncated} distributions $\cP,\cQ$ are also close.
        \cref{cond:Uncon:estimation} requires this for any set $S$ of sufficient volume.
        At a high level, this holds whenever the truncated distribution can be  ``approximately extended'' to the whole domain to recover the original distribution -- \ie{}, whenever \textit{extrapolation} is possible.
        At the end of this section, in \cref{lem:extrapolation}, we show that -- under some mild assumptions -- a rich class of distributions can be extrapolated and, hence, satisfy \cref{cond:Uncon:estimation}. 
        \mbox{Now, we are ready to state our estimation result.}
        \begin{restatable}[Estimation of ATE under Scenario~III]{theorem}{thmScenarioThreeEstimation}
            \label{lem:est:unconfoundedness} 
            Fix constants $c\in (0,\nfrac{1}{2})$, $C>0$, $\sigma,\eps,\delta\in (0,1)$, and a distribution $\mu$ over $\R^d\times \R$.
            Let concept classes $\hyP\subseteq\hyPunconf(c)$ and $\hyD$ satisfy:
            \begin{enumerate}[itemsep=-1pt]   
                \item $\hyD$ satisfies \cref{cond:Uncon:estimation} with constant $C>0$.
                \item Each $\cP\in \hyD$ is $\sigma$-smooth with respect to $\mu$.\footnote{Distribution $\cP$ is said to be $\sigma$-smooth with respect to $\mu$ if its probability density function $p$ satisfies $p(\cdot)\leq \inparen{\nfrac{1}{\sigma}}\cdot \mu(\cdot)$.}
                \item  $\hyP$ has a finite fat-shattering dimension (\cref{def:fat-shattering-dimension}) $\mathrm{fat}_{\gamma}(\hyP)<\infty$ at scale $\gamma=\Theta(\sigma c^3\eps/C)$;  
                \item $\hyD$ has a finite covering number with respect to TV distance $N_{O(c^3\eps/C)}(\hyD)<\infty$.
            \end{enumerate}
            There is an algorithm that, given $n$ \iid{} samples from the censored distribution $\cC_\cD$ for any $\cD$ realizable with respect to $\inparen{\hyP,\hyD}$ with $2c<\Pr_\cD[T{=}1]<1-2c$ and $\eps,\delta \in (0,1)$,
            outputs an estimate $\hat{\tau}$, such that, with probability $1-\delta$,
            \[
                \abs{\hat{\tau}-\tau_{\cD}}\leq \eps\,.
            \]
            There is a universal constant $\eta\geq \nfrac{1}{256}$, such that, the number of samples $n$ is 
            \[
                n = 
                O\inparen{
                    \frac{C^2}{(c^2 \eps)^4 }
                    \cdot 
                    \inparen{
                        \mathrm{fat}_{\eta\sigma c^3\eps/C}(\hyP)\cdot 
                            \log{\frac{C}{\sigma c^3\eps}}
                        +\log{\frac{N_{\eta c^3\eps/C}(\hyD)}{\delta}}
                    }
                }\,.
            \]  
            \vspace{-5mm}
        \end{restatable}

            \begin{algorithm}[tbh!]
             \begin{curvybox}
                    \DontPrintSemicolon
                    \small
                
                    \KwIn{Classes $(\hyP,\hyD)$, 
                        $\eps \in (0,1)$, 
                        and 
                        i.i.d.\ censored samples $\dataset=\inbrace{c_1,c_2,\dots,c_n}$} %
                
                    \SetKwFunction{FEstimateATE}{\textsf{Estimate~ATE~in~Scenario~III}}
                    \SetKwProg{Fn}{Function}{:}{}
                
                    \Fn{\FEstimateATE{$(\hyP,\hyD), \,\dataset, \,  \,\eps$}}{

                        \For{$t \in \{0,1\}$}{
                            
                            Split censored samples $\dataset_t{=}\inbrace{(X_i,Y_i,T_i = t)}_i \subseteq \dataset$ into $\dataset^{(1)}_t,\dataset^{(2)}_t,\dataset^{(3)}_t$

                            \vspace{2mm}

                            Find an estimate $\wh e_t(\cdot)$ of the propensity score $\Pr[T{=}t\mid X{=}x]$  using $\dataset^{(1)}_t$

                            \vspace{2mm}

                            Create $\wh S_t = \inbrace{x \given \wh e_t(x) \geq c - \eps}$ 
                            
                            \textit{$\#$~ The set $\wh{S}_t$ satisfies $\mathrm{vol}(\wh S_t) \geq c - \eps$}\;

                            \vspace{2mm}
                            Eliminate all distributions $\cP\in \hyD$, which do not satisfy $\cP(\wh{S}_t)\geq c-\sqrt{\eps}$\;

                            \vspace{2mm}

                            Use $\dataset^{(2)}_t$ to find $\sinparen{\wh p_t, \wh\cP_t}\in \hyP\times \hyD$ approximating $p_t \cdot \cD_{X,Y(t)}$

                            \vspace{2mm}

                            Use $\dataset_t^{(3)}$ to find $\cP'_t\in \hyD$ approximating the ratio $\nfrac{\wh p_t(x,y) \wh \cP_t(x,y)}{\wh{e}_t(x)}$ over $\wh{S}_t$, such that, \[\iint_{(x,y)\in \wh{S}_t\times \R} \abs{\cP'_t(x,y)-\frac{\wh p_t(x,y) \wh \cP_t(x,y)}{\wh{e}_t(x)}} \d x \d y\leq O(\eps)\,.
                        \]
                        \vspace{-4mm}
                        }
                     
                        \vspace{2mm}
                
                        Define the estimator $\wh{\tau} \;\gets\; \Ex_{\cP'_1}[y] \;-\; \Ex_{\cP'_0}[y]$

                        \vspace{2mm}
                        
                        \Return $\wh{\tau}$  ~ $\#$ \textit{which is an estimate of $\tau$ that is $\eps$-close with high probability} \;
                        
                    }
                    \vspace{4mm} 
                    
                    \caption{Algorithm to estimate ATE in Scenario III. \label{alg:est:scenario3}}
                    \end{curvybox}
                \end{algorithm}

        \noindent 
        We expect that the $\nfrac{1}{\eps^4}$ dependence on the sample complexity can be improved using boosting, but we did not try to optimize it.
        We refer the reader to \cref{sec:overview} for a sketch of the proof of \cref{lem:est:unconfoundedness} and to \cref{sec:proofof:lem:est:unconfoundedness} for a formal proof.
        {Before showing that \cref{cond:Uncon:estimation} is satisfied by interesting distribution families, we pause to note that apart from the constraints on concept classes $\hyP$ and $\hyD$, we require the additional requirement that $c < \Pr_\cD[T{=}1] < 1-c$.
        First, observe that this is a mild requirement and is significantly weaker than overlap, which requires $c < \Pr_\cD[T{=}1|X={x}] < 1-c$ for each $x$.
        (To see why it is a mild requirement, observe that it allows the propensity scores $e(x)=$ to be 0 or 1 for all covariates as in regression discontinuity designs.)
        Second, our constraints on $\hyP$ and $\hyD$ already ensure that $\Pr[T{=}1]\in (0,1)$, which was sufficient for identification; however, they allow $\Pr[T{=}1]$ to approach 0 or 1, which makes estimation impossible.
        We require this constraint to avoid these extreme cases.}
         
        Next, we show that a rich family of distributions satisfies \cref{cond:Uncon:estimation} (also see \cref{rem:extensionOfExtrpolation}).

        \begin{restatable}{lemma}{lemmaExtrapolation}
            \label{lem:extrapolation}
            Let $K = [0,1]^{d+1}$ and let $M,k\ge 1$ be constants. Define $\hyDpoly(K,M)$ as the set of all distributions with support $K$ of the form $f(x,y) \propto e^{p(x,y)},$
            where $p$ is a degree-$k$ polynomial satisfying %
            \[
                \max_{(x,y)\in K} \abs{p(x,y)} \le M\,.
            \]
            Then, the class $\hyDpoly(K,M)$ satisfies \cref{cond:Uncon:estimation} with constant 
            \[
                C = e^{5M} \cdot \Bigl(O\bigl(\min\{d,k\}\bigr)\Bigr)^k \cdot c^{-(k+1)}\,.
            \]
        \end{restatable}
        In particular, when $M,k=O(1)$ and $c=\Omega(1)$, the constant is $C=O(1)$. This result is a corollary of Lemma~4.5 in \citet{daskalakis2021statistical} and relies on the anti-concentration properties of polynomials \cite{carbery2001distributional}. Moreover, the conclusion can be generalized to the case where $K$ is any convex subset of $\R^d\times\R$. Specifically, if $[a,b]^{d+1}\subseteq K\subseteq [c,d]^{d+1}$, then the constant $C$ will scale linearly with the diameter of $K$ and with a function of the aspect ratio $\frac{d-c}{b-a}$. 
        The proof of \cref{lem:extrapolation} appears in \cref{sec:proofof:lem:extrapolation}.
        \begin{remark}[Extensions of \cref{lem:extrapolation}]
            \label{rem:extensionOfExtrpolation}
            As we mentioned, the key step in proving \cref{lem:extrapolation} is an extrapolation result by \citet{daskalakis2021statistical}. 
            More generally, one can leverage other extrapolation results -- both existing and future ones -- from the truncated statistics literature to show that \cref{cond:Uncon:estimation} is satisfied by distribution families of interest.
            For instance, one can use an extrapolation result by \citet{Kontonis2019EfficientTS} to show that \cref{cond:Uncon:estimation} is satisfied by the family of Gaussians over unbounded domains, and an extrapolation result by \citet{lee2024efficient} to show that it is satisfied by exponential families satisfying mild regularity conditions.
        \end{remark}

    \noindent\textbf{Estimation under Regression Discontinuity Design.}
        Next, we consider the estimation of $\tau$ with regression discontinuity (RD) designs.
        As mentioned before, RD designs are a special case of Scenario~III and, hence, we get the following result as an immediate corollary of \cref{lem:est:unconfoundedness}.
        \begin{corollary}
            \label{lem:est:rdDesign}
            Fix constants $c\in (0,\nfrac{1}{2})$, $C>0$, $\sigma,\eps,\delta\in (0,1)$, and a distribution $\mu$ over $\R^d\times \R$.
            Fix any class $\hyD$ that satisfies the conditions in \cref{lem:est:unconfoundedness} with constants $(C,\sigma,\eps)$.
            There is an algorithm that, given $n$ \iid{} samples from the censored distribution $\cC_\cD$ for any $c$-RD-design $\cD$ and $\eps,\delta \in (0,1)$,
            outputs an estimate $\hat{\tau}$, such that, with probability $1-\delta$,
            \vspace{-4mm}
            
            \[
                \abs{\hat{\tau}-\tau_{\cD}}\leq \eps\,.
            \]
            \vspace{-5mm}
            
            \noindent The number of samples $n$ is  
            \[
                n = 
                O\inparen{
                    \frac{C^2}{(\sigma c \eps)^2 }
                    \cdot 
                    \inparen{
                        \mathrm{fat}_{O(\sigma c\eps/C)}(\hyP)\cdot 
                            \log{\frac{C}{\sigma c\eps}}
                        +\log{\frac{N_{O(\eps c/C)}(\hyD)}{\delta}}
                    }
                }\,.
            \]   
        \end{corollary}

\vspace{-8mm}
               
\section{Conclusion}

This work extends the identification and estimation regimes for treatment effects beyond the standard assumptions of unconfoundedness and overlap, which are often violated in observational studies. Inspired by classical learning theory, we introduce a new condition that is both sufficient and necessary for the identification of ATE, even in scenarios where treatment assignment is deterministic or hidden biases exist. This condition allows us to build a framework that unifies and extends prior identification results by characterizing the distributional assumptions required for identifying ATE without the standard assumptions of unconfoundedness and overlap \cite{tan2006distributional,rosenbaum2002observational,thistlethwaite1960regressionDiscontinuity}. Beyond immediate theoretical contributions, our results establish a deeper connection between learning theory and causal inference, opening new directions for analyzing treatment effects in observational studies with complex treatment mechanisms. 

    \vspace{-1mm}
    
    \subsection*{Acknowledgments}
    \vspace{-2mm}
        This project is in part supported by NSF Award CCF-2342642.
        Alkis Kalavasis was supported by the Institute for Foundations of Data Science at Yale.
        We thank the anonymous COLT reviewers for helpful suggestions on presentation and for suggesting to include a fourth scenario.
         
\printbibliography

\appendix
\addtocontents{toc}{\protect\setcounter{tocdepth}{3}}
\newpage

\addtocontents{toc}{\protect\setcounter{tocdepth}{1}}

\section{Further Discussion of Violation of Unconfoundedness and Overlap}
    \label{sec:examples:violation} 
    In this section, we present different reasons why unconfoundedness and overlap can be violated in practice.
    Following the rest of this paper, we focus on non-longitudinal studies without network effects.
    In longitudinal studies (\ie{}, studies with repeated observations of the same individuals over long periods of time), there are many other reasons why unconfoundedness and overlap can be violated.
    Further, with network effects, unconfoundedness and overlap alone are not sufficient to enable the identification of ATE.
    
    \subsection{Violation of Unconfoundedness}
        \label{sec:examples:violation:Unc}
        First, we present a few scenarios illustrating how unconfoundedness can be violated in observational studies and RCTs.

    \paragraph{Omitted Covariates.}
        One of the main sources of confounding is that certain covariates affecting treatment assignment are omitted from the analysis.
        This can arise due to various reasons.
        As a concrete example, consider an observational study investigating the causal effect of air pollution (\emph{treatment}) on the incidence of asthma (\emph{outcome}). 
        If the study fails to include socioeconomic status (SES) (or an appropriate proxy for it), then unconfoundedness can be violated.
        This is because SES can affect both the likelihood of exposure to air pollution and health outcomes: individuals with higher SES tend to live in urban areas with elevated levels of air pollution, while simultaneously enjoying better access to healthcare services that could mitigate adverse health effects. 
        This dependence can be ``hidden'' if SES is omitted as a covariate, leading to confounding between treatment and outcomes.
        For a more detailed discussion of this scenario, we refer the reader to the comprehensive review by \citet{pope2006healthPollution}.

        As another example, consider observational studies in healthcare that use data drawn from healthcare databases, such as claims data.
        While this data is rich -- incorporating administrative interactions -- it can omit important covariates such as the patient's medical history and disease severity, which affect treatment decisions.
        Here, observational studies based on electronic medical records (EMRs) can offer a more comprehensive set of covariates -- including full treatment and diagnostic histories, past medical conditions, and fine-grained clinical measurements like vital signs \cite{hoffman2011emrs}.
        However, even with this richer dataset, the potential of confounding remains \cite{kallus2021minimax}.

    \paragraph{Excess Covariates.}
        At first, it might seem that including a very rich set of covariates would help ensure unconfoundedness -- by capturing all factors that affect treatment assignment -- however, including certain covariates can introduce confounding.
        One reason for this is that some covariates are themselves dependent on the outcomes $Y(0)$ and $Y(1).$
        For instance, one example by \citet{wooldridge2005violatingIgnorability} is as follows: 
            Consider an observational study evaluating the effects of drug courts (treatment) on recidivism (outcome).\footnote{``Drug Treatment Court is a type of alternative sentencing that allows eligible non-violent offenders who are addicted to drugs or alcohol to complete a treatment program and upon successful completion, get the criminal charges reduced or dismissed;'' see \url{https://www.hhs.gov/opioids/treatment/drug-courts/index.html}}
        Here, one should not include post-sentencing education and employment as covariates, since these quantities are themselves affected by outcomes (recidivism).
        We refer the reader to the work of \citet{wooldridge2005violatingIgnorability} for a concrete mathematical example demonstrating that including certain covariates can introduce confounding.
         
    \paragraph{Non-Compliance in RCTs.}
        In a randomized control trial (RCT), treatment assignment is explicitly randomized and typically depends only on observed covariates, so unconfoundedness is ensured by design under normal conditions. 
        However, for certain types of treatments, such as completing physical exercise and therapy sessions, participants must actively comply, making some degree of non-compliance inevitable. 
        This non-compliance violates unconfoundedness when unobserved covariates -- like the level of stress experienced at work -- affect the probability of complying with the assigned treatment.
        {As a concrete example, consider an RCT conducted by \citet{sommer1991nonComplianceVitaminA} in rural Indonesia -- in northern Sumatra -- during the early 1980s.
        In the trial, villages were randomly assigned to receive Vitamin A supplements or serve as controls. %
        This study displayed non-compliance, and nearly 20\% of the infants in the treatment villages did not receive the supplements. %
        Importantly, the mortality rate among these non-compliant infants was twice that of the control group (1.4\% vs.\ 0.64\%).
        This suggests that infants in treatment villages who did not receive Vitamin A (\ie{}, the non-compliers) had poorer health outcomes, indicating that the non-compliance was likely caused by outcome-related factors --  thereby introducing confounding.}
        For further discussion and empirical evidence on non-compliance in RCTs, see \citet{lee1991itt,rubin1995ittandGoals,hewitt2006noncompliance}. %
        We also refer the reader to \citet{rosenbaum2002observational} (\eg{}, Section 5.4.3), \citet{imbens2015causal} (\eg{}, Chapter 23), and \citet*{ngo2021noncompliance} for a more detailed discussion of non-compliance and its effect on unconfoundedness.
        Additionally, \citet{balke1997bounds,imbens2015causal} and references therein discuss how to obtain non-point estimates of ATE in studies with non-compliance.

\subsection{Violation of Overlap}
    Next, we present several reasons why the overlap condition might be violated in practice.

\paragraph{Regression Discontinuity Designs.}  
Regression discontinuity (RD) designs \cite{thistlethwaite1960regressionDiscontinuity,rubin1977regressionDiscontinuity,sacks1978regressionDiscontinuity} inherently violate the overlap condition by design. In these settings, there is a fixed partition $(S, \R^d \setminus S)$ of the covariates domain $\R^d$ and treatment is assigned to covariates in $S$. Any $x \in \R^d \setminus S$ is assigned to the control group. For instance, consider a university scholarship program that awards financial aid only to applicants whose test scores exceed $c$, \ie{}, the covariate $X$ is one-dimensional, and the treatment assignment is
\[
T = \ind{}\{X \geq c\}.
\]
In this example, no student with $X < c$ receives the treatment, and all students with $X \geq c$ do. Although valid local treatment effects can be estimated near the cutoff, the complete absence of treated individuals on one side of $c$ (or controls on the other side) implies that the overall overlap condition is violated \cite{hahn2001regressionDiscontinuity,imbens2008regressionDiscontinuity}.

\paragraph{Participation-based Studies.}
In studies where individuals must actively show up -- commonly referred to as participation-based or volunteer-based studies -- a key challenge arises in achieving overlap between the treated and control groups. In these settings, the population naturally partitions into those who choose to participate and those who do not, often leading to a self-selected sample. This self-selection (or non-response) can result in significant differences in observed and unobserved covariates between participants and nonparticipants, thereby limiting the common support necessary for valid causal inference.

For instance, consider a study evaluating the effect of a health education workshop on diabetes management. In this study, the intervention requires participants to travel to a centralized location. Individuals with higher mobility, better baseline health, or greater health motivation are more likely to attend the workshop, while those with mobility challenges or lower health literacy might opt out. This leads to a partitioning of the target population into distinct groups: one in which the propensity to participate is near one, and another where it is nearly zero. Standard sampling strategies, such as sending random invitations, oversampling underrepresented groups, or employing stratified sampling methods, are often used to mitigate these issues. However, even these strategies may not fully overcome the challenge, as the willingness to participate is frequently correlated with unobserved factors -- like intrinsic motivation or baseline health status -- that affect the outcome \cite{dillman2014internet,groves2005survey}.

\addtocontents{toc}{\protect\setcounter{tocdepth}{3}}

\section{Proofs of Identification and Estimation Results for ATE} 
    In this section, we present the proofs of results on identification and estimation of the average treatment effect in Scenarios I, II, and III.
    First, we present the proofs of identification in \cref{sec:proofs:identification}, and then the proofs of estimation in \cref{sec:proofs:estimation}.
    \subsection{Proofs of Identification Results for ATE in Scenarios I-III}\label{sec:proofs:identification}
        In this section, we present the proofs of \cref{lem:iden:unconfoundedness-overlap,lem:iden:overlap,lem:iden:unconfoundedness}, which characterize identification of ATE in Scenarios I, II, and III, respectively.
        \subsubsection{Proof of \cref{lem:iden:unconfoundedness-overlap} (Scenario~I)}
            \label{sec:proofof:lem:iden:unconfoundedness-overlap}
            In this section, we prove \cref{lem:iden:unconfoundedness-overlap}, which we restate below. %
            \thmScenarioOne*
            \begin{proof}[Proof of \cref{lem:iden:unconfoundedness-overlap}]
                Toward a contradiction, suppose that $\inparen{\hyPou(c), \hyDall}$ does not satisfy \cref{cond:iden}.
                Hence, there exist a (compatible) pair of tuples $(p, \cP),(q,\cQ)\in \inparen{\hyPou, \hyDall}$ such that 
                \begin{align*}
                    \Ex\nolimits_{\cP}[y] &\neq \Ex\nolimits_{\cQ}[y]\,,
                        \yesnum\label{eq:iden:sanity:negationa}\\
                    \cP_X &= \cQ_X\,,
                        \yesnum\label{eq:iden:sanity:negationb}\\
                    \forall_{x\in \supp(\cP_X)}\,,~~
                        \forall_{y\in\R}\,,~~\quad 
                        p(x,y)&\cdot \cP(x,y) =  q(x,y)\cdot \cQ(x,y)\,.
                    \yesnum\label{eq:iden:sanity:negation2}
                \end{align*}
                Since $p,q\in \hyPou(c)$, $p(x,y)$ and $q(x,y)$ only depend on $x$; for each $x$, let $\overline{p}(x)$ and $\overline{q}(x)$ denote the values of $p(x,y)$ and $q(x,y)$ respectively.
                For each $x\in \supp(\cP_X)$, integrating \cref{eq:iden:sanity:negation2} over $y\in \R$ implies that $\overline{p}(x)\cdot \cP_X(x) = \overline{q}(x)\cdot \cQ_X(x)$.
                But $\cP_X=\cQ_X$, hence $\overline{p}(\cdot)$ and $\overline{q}(\cdot)$ are identical over $\supp(\cP_X)$.
                Further since $0<\overline{p}(\cdot),\overline{q}(\cdot)<1$, \cref{eq:iden:sanity:negation2} implies that $\cP(x,\cdot)=\cQ(x,\cdot)$ for each $x\in \supp(\cP_X)$ contradicting  \cref{eq:iden:sanity:negationa}.
                Thus, due to the contradiction, our initial supposition must be incorrect, and $\inparen{\hyPou(c), \hyDall}$ satisfies \cref{cond:iden}.
            \end{proof}

        \subsubsection{Proof of \cref{lem:iden:overlap} (Scenario~II) }\label{sec:proofof:lem:iden:overlap}
            
            In this section, we prove \cref{lem:iden:overlap}, which is restated below with the corresponding condition.
            \thmScenarioTwo*
            \conditionScenarioTwo*
            \begin{proof}[Proof of \cref{lem:iden:overlap}]
Y                We first prove sufficiency and then necessity.
                
                \paragraph{Sufficiency.}
                    Let $\hyD$ satisfy \cref{cond:Over}.
                    By \cref{infthm:Suff}, it suffices to show that $\inparen{\hyPoverlap(c),\hyD}$ satisfies \cref{cond:iden}.
                    To this end, consider any $\inparen{p,\cP},\inparen{q,\cQ}\in \hyPoverlap(c)\times \hyD$.
                    If either of the following conditions holds, then we are done:
                    $
                        \Ex\nolimits_{\cP}[y]
                        =\Ex\nolimits_{\cQ}[y]
                        \quad\text{or}\quad
                        \cP_X\neq \cQ_X.
                    $
                    Hence, to proceed, suppose that
                    $
                        \Ex\nolimits_{\cP}[y]
                        \neq\Ex\nolimits_{\cQ}[y]
                        \quadand
                        \cP_X = \cQ_X.
                    $
                    Now, since $\hyD$ satisfies \cref{cond:Over}, there exist $x\in \supp(\cP_X)$ and $y\in \R$ such that
                    $
                        \frac{\cP(x,y)}{\cQ(x,y)}\not\in \inparen{
                                \frac{c}{1-c}, \frac{1-c}{c}
                            }
                        .
                    $
                    Fix these $x$ and $y$.
                    Since $p,q\in \hyPoverlap(c)$, it holds that
                    \[
                        \frac{q(x,y)}{p(x,y)}\in \insquare{
                                \frac{c}{1-c}, \frac{1-c}{c}
                            }\,.
                    \]
                    Hence, $\nfrac{\cP(x,y)}{\cQ(x,y)}\neq \nfrac{q(x,y)}{p(x,y)}$ and, therefore, we have found an $x\in \supp(\cP_X)$ and $y$ such that 
                    $
                        p(x,y)\cdot \cP_x(y)\neq q(x,y)\cdot \cQ_x(y),
                    $
                    completing the proof that \cref{cond:iden} holds.
    
                \paragraph{Necessity.}
                    Suppose that $\hyD$ does not satisfy \cref{cond:Over} and, hence, there exist $\cP,\cQ\in \hyD$ satisfying:
                    \begin{align*}
                        \Ex\nolimits_{\cP}[y] &\neq \Ex\nolimits_{\cQ}[y]\,,\yesnum\label{eq:iden:3:a}\\
                        \cP_X &= \cQ_X\,,\yesnum\label{eq:iden:3:b}\\
                        \forall_{x\in \supp(\cP_X)}\,,~~
                        \forall_{y\in \R}\,,~~\quad 
                            \frac{\cP(x,y)}{\cQ(x,y)} &\in \inparen{
                                \frac{c}{1-c}, 
                                \frac{1-c}{c}
                            }\,.
                            \yesnum\label{eq:iden:3:c}
                    \end{align*}
                    Since \cref{eq:iden:3:c} holds, we can find generalized propensity scores $p,q\in \hyPoverlap(c)$ such that\footnote{To see this, note that $\nfrac{q(x,y)}{p(x,y)}$ is an increasing function of $q(x,y)$ and $-p(x,y)$, and maximum and minimum values (namely, $\frac{c}{1-c}$ and $\frac{1-c}{c}$ respectively) are achieved for $\inparen{q(x,y), p(x,y)} = \inparen{1-c, c}$ and $\inparen{q(x,y), p(x,y)} = \inparen{c, 1-c}$ respectively. Now the construction follows due to the intermediate value theorem.} %
                    \begin{align*}
                    \yesnum\label{eq:iden:3:d}
                        \forall_{x\in \supp(\cP_X)}\,,~~
                        \forall_{y\in \R}\,,~~\quad 
                            \frac{\cP(x,y)}{\cQ(x,y)} =
                            \frac{q(x,y)}{p(x,y)}\,.
                    \end{align*}
                    Moreover, for $\wh{p}\coloneq1{-}p$ and $\wh{q}\coloneq1{-}q$, the pairs $(\wh{p}, \cP), (\wh{q},\cQ)$ also belong to $(\hyP, \hyD)$ since $\wh{p}, \wh{q}$ satisfy the $c$-overlap condition.
                    So $(p,\cP)$ and $(q,\cQ)$ are both compatible,\footnote{Note $(\wh{p}, \cP)$ witnesses the compatibility of $(p, \cP)$ (as in the proof of \cref{infthm:Suff}) because $\cP_{X} =\wh\cP_X$ and $\Pr[T{=}0\mid X{=}x] = \nfrac{\int_{y} p(x,y) \cP(x,y) \d y}{\cP_{X}(x)} = 1 - \nfrac{\int_{y} p(x,y) \cP(x,y) \d y}{\cP_{X}(x)} = 1 - \Pr[T{=}1 \mid X{=}x]$. Similarly for $(q,\cQ)$.} 
                    satisfy \cref{eq:iden:3:a,eq:iden:3:b} and 
                    $
                        \forall_{x\in \supp(\cP_X)}\,,
                        \forall_{y\in \R}\,, 
                            p(x,y)\cdot \cP(x,y) =
                            q(x,y) \cdot \cQ(x,y),
                    $
                    because of \cref{eq:iden:3:d}.
                    Thus, $\inparen{\hyPoverlap(c),\hyD}$ does not satisfy \cref{cond:iden} which is necessary for identification of $\tau_{\cD}$.
            \end{proof}
            
        \subsubsection{Proof of \cref{lem:iden:unconfoundedness} (Scenario~III) }
            \label{sec:proofof:lem:iden:unconfoundedness}
            In this section, we prove \cref{lem:iden:unconfoundedness}, which is restated below with the corresponding condition.
            \thmScenarioThree*
            \conditionScenarioThree*
            \begin{proof}[{Proof of \cref{lem:iden:unconfoundedness}}]
                We first prove sufficiency and then necessity.

                \paragraph{Sufficiency.}
                    Toward a contradiction, suppose $\hyD$ satisfies \cref{cond:Uncon} and, yet, $\inparen{\hyPunconf(c), \hyD}$ violate \cref{cond:iden}.
                    Since \cref{cond:iden} is violated, there exist $(p, \cP),(q,\cQ)\in \inparen{\hyPunconf, \hyD}$ such that 
                    \begin{align*}
                        \Ex\nolimits_{\cP}[y] &\neq \Ex\nolimits_{\cQ}[y]\,,
                            \yesnum\label{eq:iden:unconfoundedness:negationa}\\
                        \cP_X &= \cQ_X\,,
                            \yesnum\label{eq:iden:unconfoundedness:negationb}\\
                        \forall_{x\in \supp(\cP_X)}\,,~~
                            \forall_{y\in\R}\,,~~\quad 
                            p(x,y)&\cdot \cP(x,y) =  q(x,y)\cdot \cQ(x,y)\,.
                    \end{align*}
                    Since $\supp(\cP_X)=\R^d$ from the assumption in \cref{lem:iden:unconfoundedness}, \mbox{the last equation above is equivalent to:}
                    \[
                        \forall_{x\in \R^d}\,,~~
                            \forall_{y\in\R}\,,~~\quad 
                            p(x,y)\cdot \cP(x,y) =  q(x,y)\cdot \cQ(x,y)\,.
                        \yesnum\label{eq:iden:unconfoundedness:negationc}
                    \]
                    Since $p,q\in \hyPunconf(c)$, $p(x,y)$ and $q(x,y)$ only depend on $x$; for each $x$, let $\overline{p}(x)$ and $\overline{q}(x)$ denote the values of $p(x,y)$ and $q(x,y)$ respectively.
                    For each $x\in \R^d$, integrating \cref{eq:iden:unconfoundedness:negationc} over $y\in \R$ implies that $\overline{p}(x)\cdot \cP_X(x) = \overline{q}(x)\cdot \cQ_X(x)$, \ie{}, $\overline{p}(\cdot)$ and $\overline{q}(\cdot)$ are identical (we know that $\cP_X=\cQ_X$).
                    Which implies that $p(x,y)$ and $q(x,y)$ are identical.
                    Since $p(x,y)$ and $q(x,y)$ are identical and both lie in $\hyPunconf(c)$, there exists a set $S$ with $\vol(S)\geq c$ such that $p(x,y) > c$ for each $(x,y)\in S\times \R$.
                    So for any $(x,y)\in S\times\R$, $p(x,y)=q(x,y)>0$ and, by \cref{eq:iden:unconfoundedness:negationc}, $\cP(x,y)=\cQ(x,y)$.
                    This along with \cref{eq:iden:unconfoundedness:negationa,eq:iden:unconfoundedness:negationb} is a contradiction to \cref{cond:Uncon} and, hence, our initial supposition must be incorrect, and $\inparen{\hyPou(c), \hyDall}$ satisfies \cref{cond:iden}.

                \paragraph{Necessity.}
                    Next, we show that if $\inparen{\hyPunconf(c),\hyD}$ violate \cref{cond:Uncon}, then they also violate \cref{cond:iden}.
                    To see this, suppose $\inparen{\hyPunconf(c), \hyD}$ do not satisfy \cref{cond:Uncon}.
                    Then, there exist $\cP,\cQ\in \hyD$ such that 
                    \begin{align*}
                        \Ex\nolimits_{\cP}[y] &\neq \Ex\nolimits_{\cQ}[y]\,,
                            \yesnum\label{eq:iden:unconfoundedness:negationd}\\
                        \cP_X &= \cQ_X\,, 
                            \yesnum\label{eq:iden:unconfoundedness:negatione} 
                    \end{align*}
                    and there exists a set $S\subseteq\R^d$ with $\vol(S)\geq c$ such that 
                    \[
                        \forall_{(x,y)\in S\times \R}\,,~~\quad 
                            \cP(x,y) = \cQ(x,y)\,.
                            \yesnum\label{eq:iden:unconfoundedness:negationf}
                    \]
                    Define the generalized propensity score $p(x,y)$ as follows: for each $(x,y)\in \R^d\times \R$
                    \[
                        p(x,y) = \begin{cases}
                            \nfrac{1}{2} & \text{if } x\in S\,,\\
                            0 & \text{otherwise}\,.
                        \end{cases}
                        \yesnum\label{eq:iden:unconfoundedness:defp}
                    \]
                    Observe that $p(\cdot)$ satisfies unconfoundedness (as it is only a function of the covariates) and also satisfies $c$-weak-overlap (\cref{def:weakoverlap}) since $\vol(S)\geq c$ and $p(x,y)> c$ for each $(x,y)\in S\times \R$ (as $c<\nfrac{1}{2}$).
                    Hence, $p(\cdot)\in \hyPunconf(c)$.
                    Also, observe that for $\wh{p}\coloneq1{-}p$ we have $\wh{p}(x,y){=}1{-}p(x,y){=}\sfrac{1}{2}$ for $x\in S$ and, so, $\wh{p} \in \hyPunconf(c)$ as well.
                    We claim that the tuples $(p,\cP)$ and $(p,\cQ)$ witness that \cref{cond:iden} is violated:
                    First, notice that $(p,\cP),(p,\cQ)$ are compatible as witnessed by $(\wh{p},\cP),(\wh{p},\cQ)$ respectively.\footnote{Note that $(\wh{p}, \cP)$ witnesses the compatibility of $(p, \cP)$ (as in the proof of \cref{infthm:Suff}) since $\cP_{X} =\wh \cP_X$ and $\Pr[T{=}0\mid X{=}x] = \nfrac{\int_{y} p(x,y) \cP(x,y) \d y}{\int_y \cP(x,y)\d y} = 1 - \nfrac{\int_{y} p(x,y) \cP(x,y) \d y}{\int_y \cP(x,y)\d y} = 1 - \Pr[T{=}1 \mid X{=}x]$. Similarly for $(p',\cQ)$.}
                    Then,  since $\Ex_{\cP}[y]\neq \Ex_{\cQ}[y]$ and $\cP_X=\cQ_X$ (\cref{eq:iden:unconfoundedness:negationd,eq:iden:unconfoundedness:negatione}), it suffices to show that 
                    \[
                        \forall_{x\in \supp(\cP_X)}\,,~~
                            \forall_{y\in\R}\,,~~\quad 
                                p(x,y)\cP(x,y)=p(x,y)\cQ(x,y)\,.
                    \] 
                    To see this consider any $(x,y)\in \R^d\times \R$.
                    If $x\in S$, then the relation holds, since $\cP(x,y)=\cQ(x,y)$ (\cref{eq:iden:unconfoundedness:negationf}) and, otherwise if $x\not\in S$, the relation still holds as $p(x,y)=0$ (\cref{eq:iden:unconfoundedness:defp}).
                    
            \end{proof}

    \subsection{Proof of Estimation Results for ATE in Scenarios I-III}
        \label{sec:proofs:estimation}
        
        In this section, we present the proofs of \cref{lem:est:unconfoundedness-overlap,lem:est:overlap,lem:est:unconfoundedness}, which provide sufficient conditions for estimation of ATE in Scenarios I, II, and III, respectively.
        \subsubsection{Proof of \cref{lem:est:unconfoundedness-overlap} (Scenario~I)}
            \label{sec:proofof:lem:est:unconfoundedness-overlap}
            In this section, we prove \cref{lem:est:unconfoundedness-overlap}, which we restate below.
            \thmScenarioOneEstimation*
            \noindent {Since \cref{lem:est:unconfoundedness-overlap} is well known (\eg{}, \cite{wager2020notes}), we only provide a sketch of the proof here.}
            \begin{proof}[{Proof {sketch} of \cref{lem:est:unconfoundedness-overlap}}]
                First, we construct the estimator $\wh{\tau}$. 
                Let $e(x)=\Pr[T{=}1\mid X{=}x]$ be the propensity score function. 
                By \cref{thm:prop-scores-estimation}, we get an estimation for the propensity score function $\wh{e}(\cdot)$ such that {$\wh{e}(\cdot)$ satisfies $c$-overlap and}
                $
                    \Ex_{x\sim\cD_X}\abs{e(x)-\wh{e}(x)} 
                        \leq \frac{c^2\eps}{4B}$
                with probability $1-\delta$,
                using $O\inparen{\frac{B^2}{(c^2\eps)^2} \inparen{\text{fat}_{\sfrac{c^2\eps}{B}}(\hyP)\log(\nfrac{B}{c^2\eps})+\log(\nfrac{1}{\delta})}}$ samples from $\cC_{\cD}$.
                
                Then we define $\wh{\tau}$ as
                \[
                    \wh{\tau} = \frac{1}{m} \sum_{i=1}^m \frac{y_it_i}{\wh{e}(x_i)}
                                - \frac{1}{m} \sum_{i=1}^m \frac{y_i(1-t_i)}{1-\wh{e}(x_i)}
                    \,,
                \]
                for $m$ is the number of {(fresh)} samples $(x_i, y_i, t_i)$ from $\cC_{\cD}$. Now, we show that $\wh{\tau}$ is $\eps$-close to $\tau$ in two steps.
                First, suppose we knew the propensity score function $e$.
                Then we define the estimator
                $
                    \bar{\tau} = \frac{1}{m} \sum_{i=1}^m \frac{y_it_i}{e(x_i)}
                                - \frac{1}{m} \sum_{i=1}^m \frac{y_i(1-t_i)}{1-e(x_i)}.
                $
                The result follows due to the following standard observations
                \[
                    \sabs{
                        \Ex[\wh{\tau}]
                        -
                        \Ex[\bar{\tau}]
                    }
                    \leq \frac{\eps}{2}\,,\quad 
                    \Ex[\bar{\tau}] = \tau\,,\quadand
                    \Pr\insquare{
                        \sabs{
                            \wh{\tau}
                            - 
                            \Ex[\bar{\tau}] 
                        }
                        \leq \frac{\eps}{2}
                    }
                    \geq 1-\delta\,.
                \]
                The first result follows by simple calculations since both $e(\cdot)$ and $\wh{e}(\cdot)$ satisfy $c$-overlap.
                The second observation is a consequence of the linearity of expectation and unconfoundedness.
                The third observation follows due to, \eg{}, Hoeffding's bound and the fact that the outcome variables are bounded in absolute value by $B$.
            \end{proof}
 
        \subsubsection{Proof of \cref{lem:est:overlap} (Scenario~II)}
            \label{sec:proofof:lem:est:overlap}
            
            In this section, we prove \cref{lem:est:overlap}, which is restated below with the corresponding condition.
            \thmScenarioTwoEstimation*
            \conditionScenarioTwoEstimation*
            \medskip 
            
            \begin{proof}[{Proof of \cref{lem:est:overlap}}]
                First, we construct $\hat{\tau}$ and then show that its $\eps$-close to $\tau$ under \cref{cond:Over:estimation}.

                \paragraph{Construction of $\wh{\tau}$.}
                    The algorithm to construct $\wh{\tau}$ is simple (see Algorithm~\ref{alg:est:scenario2}) and relies on estimating certain nuisance parameters.
                    We explain the construction of the nuisance parameter estimators in \cref{sec:nuisanceParameters} and use the estimators as black boxes here.
                    Notice that, since $c$-overlap holds, it also holds that $\Pr[T{=}1]\in(c,1-c)$, as required by \cref{thm:oracleConstruction}.\footnote{To see this, consider that $\Pr[T{=}1] = \int_{(x,y)} p_1(x,y) \cD_{X,Y(1)}\d x \d y$ and $p_1(x,y)\in(c,1-c)$ for all $(x,y)$.}
                    In particular, to construct $\wh{\tau}$ we query the $L_1$-estimation oracle (\cref{def:densityOracle}) with accuracy $\sfrac{M(\nfrac{\eps}{2})}{2}$ and confidence $\delta$.
                    This oracle has the property that, with probability $1-\delta$, the tuples $\inparen{p,\cP}$ and $\inparen{q,\cQ}$ returned by the oracle satisfy $\cP_X=\cQ_X$ and are close to $p_1\cD_{X,Y(1)}$ and $p_0\cD_{X,Y(0)}$ in the following sense:\footnote{Where, recall that, where we define $
            \norm{p\cP - q\cQ}_1\coloneqq
                \iint \abs{p(x,y)\cP(x,y)-q(x,y)\cQ(x,y)} \d x \d y\,.
        $}
                    \[
                        \norm{p_1{\cD_{X,Y(1)}} - p\cP}_1 < \frac{M(\nfrac{\eps}{2})}{2}
                        \qquadand
                        \norm{p_0{\cD_{X,Y(0)}} - q\cQ}_1 < \frac{M(\nfrac{\eps}{2})}{2}\,,
                        \yesnum\label{def:estOne:guarantee}
                    \]
                    We define the estimator $\wh{\tau}$ as follows 
                    \[
                        \hat{\tau} = \Ex\nolimits_{\cP}[y] - \Ex\nolimits_{\cQ}[y]
                        \,.
                        \yesnum\label{def:estOne:estimato}
                    \] 
                    In the above construction, samples from $\cC_\cD$ are only used in the query to the $L_1$-approximation oracle, and the sample complexity claimed in the result follows from the sample complexity of the $L_1$-approximation oracle (see \cref{thm:oracleConstruction}).

                \paragraph{Accuracy of $\wh{\tau}$.}
                Condition on the event $\evE$ that the above guarantee holds.
                We show that 
                \[
                    \abs{\Ex\nolimits_{\cD_{X,Y(1)}}[y]- \Ex\nolimits_{\cP}[y]}\leq \frac{\eps}{2}
                    \qquadand
                    \abs{\Ex\nolimits_{\cD_{X,Y(0)}}[y]- \Ex\nolimits_{\cQ}[y]}\leq \frac{\eps}{2}
                    \,,
                    \yesnum\label{def:estOne:guarantee:b}
                \]
                which implies the desired result due to the triangle inequality and the definition of $\wh{\tau}$ (\cref{def:estOne:estimato}).
                Toward a contradiction, suppose Inequality~\eqref{def:estOne:guarantee:b} is violated. 
                First, suppose $\sabs{\Ex\nolimits_{\cD_{X,Y(1)}}[y]- \Ex\nolimits_{\cP}[y]} > \nfrac{\eps}{2}$ and the other case will follow by substituting $Y(0)$, $\cP$, and $p_1$ by $Y(1)$, $\cQ$, and $p_0$ respectively in the subsequent proof.
                Consider the set $S$ in \cref{cond:Over:estimation} for the tuple $\sinparen{\cD_{X,Y(1)}, \cP}$ and partition it into the following two parts:\footnote{To see why this is a partition, observe that $S$ satisfies \cref{cond:Over:estimation:req} in \cref{cond:Over:estimation}.}
                \[  
                    S_L\coloneqq 
                        \inbrace{
                            (x,y)\in S \given 
                            \frac{\cD_{X,Y(1)}(x,y)}{\cP(x,y)} < \frac{c}{2(1-c)}
                    }
                    \quadand
                    S_R\coloneqq  \inbrace{
                            (x,y)\in S \given 
                            \frac{\cD_{X,Y(1)}(x,y)}{\cP(x,y)} > \frac{2(1-c)}{c}
                    }
                    \,.
                \]
                These parts satisfy the following properties:
                \begin{itemize}
                    \item[(P1)] For each $(x,y)\in S_L$ and the generalized propensity score $p(\cdot)$ returned by the density oracle
                    \[
                        p(x,y)\cP(x,y) > 2(1-c) \cdot \cD_{X,Y(1)}(x,y)\,,
                    \]
                    where we used the definition of $S_L$ and that $p(\cdot)\in \hyPoverlap(c)$ and, hence, $p(x,y)>c$.
                    \item[(P2)]
                        For each $(x,y)\in S_R$, 
                        \[
                            p(x,y)\cP(x,y) < \frac{c}{2} \cdot \cD_{X,Y(1)}(x,y)\,,
                        \]
                        where we used the definition of $S_R$ and that $p(\cdot)\in \hyPoverlap(c)$ and, hence, $p(x,y)<1-c$.
                \end{itemize}
                In the remainder of the proof, we lower bound $\snorm{p_1\cD_{X,Y(1)}-p\cP}_1$ to obtain a contradiction to \cref{def:estOne:guarantee}.
                The definition of the $L_1$-norm and the disjointness of $S_L$ and $S_R$ implies 
                \[
                    \norm{p_1\cD_{X,Y(1)}-p\cP}_1
                    \geq 
                    \snorm{\mathds{1}_{S_L}\cdot p_1\cD_{X,Y(1)}-\mathds{1}_{S_L}\cdot p\cP}_1
                    + 
                    \snorm{\mathds{1}_{S_R}\cdot p_1\cD_{X,Y(1)}-\mathds{1}_{S_R}\cdot p\cP}_1\,.
                    \yesnum\label{eq:estOne:lb}
                \]
                Where for each set $T\in \inbrace{S_L,S_R}$, $\mathds{1}_{T}$ denotes the indicator function $\mathds{1}\inbrace{(x,y)\in T}$.
                Toward lower bounding the first term, observe that for any $(x,y) \in S_L$
                \begin{align*}
                    p_1(x,y)\cD_{X,Y(1)}(x,y) - p(x,y)\cP(x,y)~~
                    &\Stackrel{(\mathrm{P1})}{<}~~ 
                    p_1(x,y)\cD_{X,Y(1)}(x,y) - 2(1-c)\cD_{X,Y(1)}(x,y)\\
                    &\leq~~ -(1-c)\cdot \cD_{X,Y(1)}(x,y)\,.
                    \tag{since $p_1\in \hyPoverlap(c)$}
                \end{align*}
                Therefore, 
                \begin{align*}
                    \snorm{\mathds{1}_{S_L}\cdot p_1\cD_{X,Y(1)}-\mathds{1}_{S_L}\cdot p\cP}_1
                    > (1-c)\cdot \cD_{X,Y(1)}(S_L)\,.
                    \yesnum\label{eq:estOne:lb:1}
                \end{align*}
                A similar approach lower bounds the second term:
                    for any $(x,y) \in S_R$
                \begin{align*}
                    p_1(x,y)\cD_{X,Y(1)}(x,y) - p(x,y)\cP(x,y)~~
                    &\Stackrel{(\mathrm{P2})}{>}~~ 
                    p_1(x,y)\cD_{X,Y(1)}(x,y) - \frac{c}{2}\cdot \cD_{X,Y(1)}(x,y)\\
                    &\geq~~ \frac{c}{2}\cdot \cD_{X,Y(1)}(x,y)\,.
                    \tag{since $p_1\in \hyPoverlap(c)$}
                \end{align*}
                Consequently, we obtain the following lower bound on the second term in \cref{eq:estOne:lb}
                \begin{align*}
                    \snorm{\mathds{1}_{S_R}\cdot p_1\cD_{X,Y(1)}-\mathds{1}_{S_R}\cdot p\cP}_1
                    > \frac{c}{2} \cdot \cD_{X,Y(1)}({S_R})\,.
                    \yesnum\label{eq:estOne:lb:2}
                \end{align*}
                Substituting \cref{eq:estOne:lb:1,eq:estOne:lb:2} into \cref{eq:estOne:lb} and using $c<\nfrac{1}{2}$, implies that 
                \[
                    \norm{p_1\cD_{X,Y(1)}-p\cP}_1
                    > 
                    \frac{c}{2}\inparen{
                        \cD_{X,Y(1)}(S_L)
                        + \cD_{X,Y(1)}(S_R)
                    }\,.
                \]
                Since $S=S_L\cup S_R$, $S_L$ and $S_R$ are disjoint, and $\cD_{X,Y(1)}(S)\geq \sfrac{M(\nfrac{\eps}{2})}{c}$ due to \cref{cond:Over:estimation},
                \[
                    \norm{p_1\cD_{X,Y(1)}-p\cP}_1
                    > 
                    \frac{c}{2}\cdot \frac{M(\nfrac{\eps}{2})}{c}
                    = \frac{M(\nfrac{\eps}{2})}{2}
                    \,,
                \]
                which is a contradiction to \cref{def:estOne:guarantee}.
                Finally, in the other case, where $\sabs{\Ex\nolimits_{\cD_{X,Y(0)}}[y]- \Ex\nolimits_{\cQ}[y]}\leq \frac{\eps}{2}$, substituting $Y(1)$, $p_1(\cdot)$, $p$, and $\cP$ in the above argument by $Y(0)$, $p_0(\cdot)$, $q$, and $\cQ$ implies that $\snorm{p_0\cD_{X,Y(0)}-q\cQ}_1 >  \sfrac{M(\nfrac{\eps}{2})}{2}$ also contradicting \cref{def:estOne:guarantee}.

            \end{proof}

        \subsubsection{Proof of \cref{lem:est:unconfoundedness} (Scenario~III)}
            \label{sec:proofof:lem:est:unconfoundedness}
            
            In this section, we prove \cref{lem:est:unconfoundedness}, which is restated below with the corresponding condition.
            \thmScenarioThreeEstimation*
            \conditionScenarioThreeEstimation*
            
            \paragraph{Overview of Estimation Algorithm (Algorithm~\ref{alg:est:scenario3}).}
            In this scenario, unconfoundedness holds and we assume a weak form of overlap: there are $S_0, S_1 \subseteq \mathbb{R}^d$ with $\vol(S_0),\vol(S_1) \geq c$ such that
            \[
            \forall\, x \in S_0:\quad e_0(x) \coloneqq \Pr[T{=}0\mid X{=}x] \geq c,\quad \text{and} \quad
            \forall\, x \in S_1:\quad e_1(x) \coloneqq \Pr[T{=}1\mid X{=}x] \geq c.
            \]
            If we had oracle membership access to $S_0, S_1$ and query access to functions $e_0(\cdot),e_1(\cdot)$, a slight modification of the Scenario II estimator (Algorithm~\ref{alg:est:scenario2}) suffices: one would find a pair $(p,\mathcal{P})$ so that the product $p\mathcal{P}$ approximates $p_1\mathcal{D}_{X,Y(1)}$ on $S_1$ and output $\mathbb{E}_{\mathcal{P}}[y]$ as an estimator for $\mathbb{E}_{\mathcal{D}}[Y(1)]$ (with an analogous procedure for $\mathbb{E}_{\mathcal{D}}[Y(0)]$). 
            Under \cref{cond:Uncon:estimation}, one can prove the correctness of this approach. 
            However, we lack direct membership and query access to $S_0, S_1$ and the propensity functions; hence, these must be estimated from samples while controlling for estimation error. 
            This is what Algorithm~\ref{alg:est:scenario3} does.

            \paragraph{Correctness of Algorithm~\ref{alg:est:scenario3}.}  
            The algorithm proceeds in three phases. For brevity, we detail the argument for estimating $\mathbb{E}[Y(1)]$; the analysis for $\mathbb{E}[Y(0)]$ is analogous.
            First, since $\Pr_{\mathcal{D}}[T{=}1] > 2c$, the set 
            \[
                S_1 \coloneqq \{ x \in \mathbb{R}^d \mid e(x) \geq c \}
            \]
            has $\mathcal{D}_X$-mass at least\footnote{Indeed, 
            $\Pr[T{=}1] \leq c\int_{x\not\in S_1} \, d\mathcal{D}_X(x) + \int_{x\in S_1} \, d\mathcal{D}_X(x) \leq c + \mathcal{D}_X(S_1)$, so that $\mathcal{D}_X(S_1) \geq c$.}
            \[
                    \cD_X(S_1)\geq c\,.
                    \yesnum\label{eq:est:scenario3:masslb}
            \]
            The first step of Algorithm~\ref{alg:est:scenario3} is to estimate the propensity score $e(\cdot)$. Because the hypothesis class $\hyP$ has finite fat-shattering dimension (see \cref{sec:nuisanceParameters}),
            we can obtain
            a propensity score estimate $\widehat{e}(\cdot)$ that satisfies
            \[
\mathbb{E}_{x\sim\mathcal{D}_X}\Bigl[ \bigl|\widehat{e}(x)-e(x)\bigr|\Bigr] \leq \eps.
            \]
            Since $|\widehat{e}(x)-e(x)| \in [0,1]$, Markov's inequality yields that for any $\gamma>0$
            \[
            \Pr_{x\sim \cD_X}\insquare{
                        \abs{\wh{e}(x)-e(x)}
                        \geq \gamma
                    } \leq \nfrac{\eps}{\gamma}\,.
            \]
            Define the \emph{bad set} 
            \[
                B \coloneqq \{ x \in \mathbb{R}^d \mid |\widehat{e}(x)-e(x)| \geq \sqrt{\eps} \}\,.
            \]
                The previous inequality implies that 
                \[
                    \cD_X(B) \leq \sqrt{\eps}\,.\yesnum\label{eq:est:scenario3:badSet}
                \]
                The next step in Algorithm~\ref{alg:est:scenario3} is to construct the following set 
                \[
                    \wh{S}_1 \coloneqq \inbrace{x\in \R^d\mid \wh{e}(x) \geq c-\sqrt{\eps}}\,.
                \]
                Since for any point $x\not\in B$, $\abs{\wh{e}(x)-e(x)} \leq \sqrt{\eps}$, and for each $x\in S_1$, $e(x)\geq c$, it follows that 
                \[
                    \wh{S}_1 \supseteq S_1 \setminus B\,,
                \]
                and, hence, \cref{eq:est:scenario3:masslb,eq:est:scenario3:badSet} imply that,
                \[
                    \cD_X(\wh{S}_1)\geq c-\sqrt{\eps}\,.
                    \yesnum\label{eq:est:scenario3:lbEstSet}
                \] 
                Further, for any $x\in \wh{S}_1\setminus B$, $e(x)\geq \wh{e}(x)-\sqrt{\eps}\geq c-2\sqrt{\eps}$ by the definition of $\wh{S}_1$ and $B$.
                Therefore, 
                \begin{align*}
                    \Pr_{x\sim \cD_X}\insquare{
                        e(x)\geq {c-2\sqrt{\eps}}
                        \mid x\in \wh{S}_1
                    }
                    &{\geq} \frac{\cD_X\sinparen{ \wh{S}_1\setminus B}}{\cD_X(\wh{S}_1)}\,.
                \end{align*}
                Since $\cD_X(B) < \sqrt{\eps}$ and $\cD_X(\wh{S}_1)\geq c-\sqrt{\eps}$, it follows that 
                \[
                    \Pr_{x\sim \cD_X}\insquare{
                        e(x)\geq c-2\sqrt{\eps}
                        \mid x\in \wh{S}_1
                    }
                    \geq \frac{\cD_X\sinparen{\wh{S}_1}-\cD_X(B)}{\cD_X(\wh{S}_1)}
                    \geq \frac{c-2\sqrt{\eps}}{c-\sqrt{\eps}}
                    = 1 - \frac{\sqrt{\eps}}{c-\sqrt{\eps}}\,.
                    \yesnum\label{eq:est:scenario3:whp}
                \]
                Since $\cD_X$ satisfies the constraint in \cref{eq:est:scenario3:lbEstSet} and $\wh{S}_1$ is known, we eliminate all distributions $\cP\in \hyD$, which do not satisfy $\cP(\wh{S}_1)\geq c-\sqrt{\eps}$. 
                With abuse of notation, we use $\hyD$ to denote the resulting concept class in the remainder of the proof.

                The final pair of steps in Algorithm~\ref{alg:est:scenario3} are as follows:
                \begin{enumerate}
                    \item Estimate a pair $\sinparen{\wh{p}, \wh{\cP}}\in \hyP\times \hyD$ such that the product $\wh{p}(x,y)\wh{\cP}(x,y)$ is close to the product ${p}(x,y){\cP}(x,y)$ in the following sense 
                    \[
                        \iint 
                            \abs{
                                \wh{p}(x,y)\wh{\cP}(x,y)
                                -
                                {p}(x,y){\cP}(x,y)
                            }\d x \d y 
                            < \eps\,.
                            \yesnum\label{eq:est:scenario3:productOracle}
                    \]
                    \item Estimate a distribution $\cP'\in \hyD$ such that 
                    \[
                        \iint_{x\in \wh{S}_1} 
                            \abs{
                                \cP'(x,y)
                                -
                                \frac{
                                    \wh{p}(x,y)\wh{\cP}(x,y)
                                }{  
                                    \wh{e}(x)
                                }
                            }\d x \d y < O\inparen{\frac{\sqrt{\eps}}{c}}\,.
                            \yesnum\label{eq:est:scenario3:accuracyTruncated}
                    \]
                \end{enumerate}
                The claim is that $\mathbb{E}_{(x,y)\sim {\cP'}}[y]$ is $O(\sqrt{\epsilon})$-close to $\mathbb{E}_{\mathcal{D}}[Y(1)]$. 
                However, before proving this, we must verify that the preceding steps can be implemented with finite samples.
                First, note that the estimation in the first step is feasible because, by the requirements on $\hyP$ and $\hyD$ in 
                \cref{lem:est:unconfoundedness}, one can construct an $\epsilon$-cover of $\hyP\times \hyD$ with respect to the specified distance (see \cref{sec:nuisanceParameters} for details).
                Regarding the second step, two checks are needed: (i) that there exists a distribution $\cP'$ satisfying \cref{eq:est:scenario3:accuracyTruncated}, and (ii) that such a $\cP'$ can be found from samples. For (ii), it suffices to construct an $O(\nfrac{\epsilon}{c})$-cover of $\hyP$, which is possible since $\hyP$ has finite fat-shattering dimension and we have a lower bound on the mass assigned to $\wh{S}_1$ by any distribution $\cP\in \hyD$ (see \cref{sec:nuisanceParameters} for the construction). It remains to verify (i).

                Towards this, it suffices to show that the function $\nfrac{
                                    \wh{p}(x,y)\cdot \wh{\cP}(x,y)
                                }{  
                                    \wh{e}(x)
                                }$ is $O\inparen{\nfrac{\sqrt{\eps}}{c}}$-close to some density function in $\hyD$ over the set $\wh{S}_1\times \R$.
                In fact, we show closeness to $\cP$.
                In other words, we want to upper bound
                \[
                    \iint_{x\in \wh{S}_1}
                    \abs{
                        \frac{
                            \wh{p}(x,y)\cdot \wh{\cP}(x,y)
                        }{  
                            \wh{e}(x)
                        }
                        - \cP(x,y)
                    }\d x \d y\,.
                \]
                The triangle inequality implies that 
                \begin{align*}
                    &\iint_{x\in \wh{S}_1}\abs{
                        \frac{
                            \wh{p}(x,y)\cdot \wh{\cP}(x,y)
                        }{  
                            \wh{e}(x)
                        }
                        - \cP(x,y)
                    }\d x\d y\\
                    &\quad\leq 
                    \iint_{x\in \wh{S}_1}
                    \frac{1}{\wh{e}(x)}
                    \cdot \abs{
                        \wh{p}(x,y)\cdot \wh{\cP}(x,y)
                             - 
                             p(x,y)\cdot \cP(x,y)
                    } \d x\d y  + 
                    \iint_{x\in \wh{S}_1}\abs{
                        \frac{
                            {p}(x,y)\cdot {\cP}(x,y)
                        }{  
                            \wh{e}(x)
                        }
                        -
                        \cP(x,y)
                    }\d x\d y\,.
                    \yesnum\label{eq:est:scenario3:bound0}
                \end{align*}
                Since for each $x\in \wh{S}_1$, $\wh{e}(x)\geq c-\sqrt{\eps}\geq \nfrac{c}{2}$, \cref{eq:est:scenario3:productOracle} implies that the first term is at most $\nfrac{2\eps}{c}$.
                Hence, it remains to upper bound the second term.
                Another application of the triangle inequality implies the following 
                \begin{align*}
                    &\iint_{x\in \wh{S}_1}\abs{
                        \frac{
                            {p}(x,y)\cdot {\cP}(x,y)
                        }{  
                            \wh{e}(x)
                        }
                        -
                        \cP(x,y)
                    }\d x\d y
                    \leq \\
                    &\quad \iint_{x\in \wh{S}_1\setminus B}\abs{
                        \frac{
                            {p}(x,y)\cdot {\cP}(x,y)
                        }{  
                            \wh{e}(x)
                        }
                        -
                        \cP(x,y)
                    }\d x\d y
                    + \iint_{x\in \wh{S}_1\cap B}\abs{
                        \frac{
                            {p}(x,y)\cdot {\cP}(x,y)
                        }{  
                            \wh{e}(x)
                        }
                        -
                        \cP(x,y)
                    }\d x\d y
                    \,.
                    \yesnum\label{eq:est:scenario3:bound1}
                \end{align*}
                Since $\abs{e(x)-\wh{e}(x)}\leq \sqrt{\eps}$ and $\wh{e}(x)\geq c-\sqrt{\eps}$ for each $x\in \wh{S}_1\setminus B$, for each $x\in \wh{S}_1\setminus B$, $\frac{p(x,y)}{\wh{e}(x)}=\frac{e(x)}{\wh{e}(x)}=1\pm \frac{\sqrt{\eps}}{c-\sqrt{\eps}}$ (where in the first equality we used unconfoundedness) the first term is at most 
                \begin{align*}
                    \iint_{x\in \wh{S}_1\setminus B}\abs{
                        \frac{
                            {p}(x,y)\cdot {\cP}(x,y)
                        }{  
                            \wh{e}(x)
                        }
                        -
                        \cP(x,y)
                    }\d x\d y
                    \leq 
                    \sqrt{\eps}
                    \cdot \iint_{x\in \wh{S}_1\setminus B}\abs{ \cP(x,y) }\d x\d y
                    \leq \sqrt{\eps}\,.
                    \yesnum\label{eq:est:scenario3:bound2}
                \end{align*}
                Regarding the second term, for each $x\in \wh{S}_1$, $\wh{e}(x)\geq c-\sqrt{\eps}$ and $p(x,y)\le 1$ (for any $y\in \R$), and hence, triangle inequality and \cref{eq:est:scenario3:badSet} imply
                \[
                    \iint_{x\in \wh{S}_1\cap B}\abs{
                        \frac{
                            {p}(x,y)\cdot {\cP}(x,y)
                        }{  
                            \wh{e}(x)
                        }
                        -
                        \cP(x,y)
                    }\d x\d y
                    \leq \frac{1}{c-\sqrt{\eps}}\cP\inparen{\wh{S}_1\cap B}
                    + \cP\inparen{\wh{S}_1\cap B}
                    \leq \frac{2 \sqrt{\eps}}{c-\sqrt{\eps}}\,.
                    \yesnum\label{eq:est:scenario3:bound3}
                \]
                Combining \cref{eq:est:scenario3:bound0,eq:est:scenario3:bound1,eq:est:scenario3:bound2,eq:est:scenario3:bound3} implies the desired bound in \cref{eq:est:scenario3:accuracyTruncated}.
                This completes the proof that $\cP'$ satisfies \cref{eq:est:scenario3:accuracyTruncated}, then 
                \[
                    \abs{\Ex_{(x,y)\sim \cP'}[y] - \Ex_\cD[Y(1)]}
                    \leq O\inparen{\frac{\sqrt{\eps}\cdot C}{c}}\,,
                \]
                this follows by an application of \cref{cond:Uncon:estimation} since $\cP'\in \hyD$ and $\cD$ is realizable with respect to $\hyD$.
                \cref{lem:est:unconfoundedness} follows by changing $\eps$ to $\eps^2\cdot \nfrac{c}{C}$.

\addtocontents{toc}{\protect\setcounter{tocdepth}{1}}
\section{{Proofs Omitted from Scenario III}}
        In this section, we prove \cref{lem:extrapolation,lem:iden:unconfoundedness:examples}, which give natural examples of distribution classes $\hyD$ that satisfy our conditions.

        \subsection{Proof of \cref{lem:iden:unconfoundedness:examples} (Classes $\hyD$ Identifiable in Scenario~III)} 
            \label{sec:proofof:lem:iden:unconfoundedness:examples}
            \label{sec:proofof:lem:unconfoundedness:identifiableFamilies}
            In this section, we prove \cref{lem:unconfoundedness:identifiableFamilies}, which we restate below.
            \lenUnconfoundednessIdentifiableFamilies*
            \smallskip
            \begin{proof}[Proof of \cref{lem:iden:unconfoundedness:examples}]
                We proceed in two parts: one for each distribution family.
                
                \paragraphit{Proof for Polynomial Log-Densities.}
                    Consider any pair $\cP,\cQ$ in the polynomial log-density family such that $\Ex_{(x,y)\sim\cP}[y]\neq\Ex_{(x,y)\sim\cQ}[y]$.
                    It is immediate that \cref{cond:Uncon} holds if $\cP_{X}\neq\cQ_{X}$.
                    So, we need to show that, if $\cP_{X}=\cQ_{X}$, the following is true
                    \[
                        \nexists S\subseteq\R^d
                        \quadtext{with}
                            \vol(S)\geq c
                        \quad \text{such that} \quad
                        \forall_{(x,y)\in S\times\R}
                            \,,~\cP(x,y)=\cQ(x,y) 
                        \,.
                    \]
                    For the sake of contradiction, assume there exists such a set $S\subseteq\R^d$.
                    Then, since $\cP(x,y) = \cP_{Y\mid X}(y\mid x) \cP_X(x)$ and similarly for $\cQ$, it follows that, for every $(x,y)\in S\times\R$
                    \[
                        \cP_{Y\mid X}(y\mid x) \cP_X(x)
                             = \cQ_{Y\mid X}(y\mid x) \cQ_X(x)
                        \,.
                    \]
                    Since $\cP_X=\cQ_X$ and $\supp(\cP_x)=\supp(\cQ_X)=\R^d$, we have
                    \[
                        \cP_{Y\mid X}(y\mid x) = \cQ_{Y\mid X}(y\mid x)
                        \,.
                    \]
                    Let $f_{\cP}$ be the polynomial for which $\cP(x,y)\propto e^{f_{\cP}(x,y)}$ and, similarly, for polynomial $f_{\cQ}$ and $\cQ$.
                    Then it must be
                    \[
                        e^{f_{\cP}(x,y)}
                            = c(x) \cdot e^{f_{\cQ}(x,y)}
                        \,,
                    \]
                    where $c(x)$ is the ratio of {the partition functions} $Z_{\cP}(x) = \int_y e^{f_{\cP}(x,y)}$ over $Z_{\cQ}(x) = \int_y e^{f_{\cQ}(x,y)}$.
                    Equivalently, 
                    \[
                        f_{\cP}(x,y) - f_{\cQ}(x,y)
                            - \log(c(x)) = 0
                        \,,
                        \yesnum\label{eq:polynomials}
                    \]
                    for all $(x,y)\in S\times\R$.
                    Fix any value $x\in\R^d$. Then the LHS in \cref{eq:polynomials} is a polynomial with respect to $y$ and can either be the zero polynomial or be zero only on a finite number of points, equal to the degree of the polynomial.
                    The second case cannot be true since we want $\cP(x,y)=\cQ(x,y)$ for all $(x,y)\in S\times\R$.
                    Thus, the polynomial must be identically zero with respect to $y$, for all $x\in S$, \ie{}, its coefficients must be identically zero.
                    However, the coefficients of $y$, are polynomials over $x$, so, if they are zero on an infinite number of points, they must be identically zero as well.
                    Thus, $f_{\cP}(x,y)-f_{\cQ}(x,y) = p(x) = \log(c(x))$, for all $x\in\R^d$, where $p(x)$ is a polynomial of $x$.
                    Now, for all $(x,y)\in\R^d\times\R$ 
                    \[
                        \cP(x,y) = \cP_{Y\mid X}(y\mid x) \cP_X(x)
                                 = \frac{1}{Z_{\cP}(x)} \cdot e^{f_{\cP}(x,y)} \cP_X(x)
                        \,.
                    \]
                    Note that $Z_{\cP}(x) = \int_y e^{f_{\cP}(x,y)} \d y = e^{p(x)} \int_y e^{f_{\cQ}(x,y)} \d y $.
                    So, since $\cP_X=\cQ_X$, we have
                    \[
                        \cP(x,y) 
                        = \frac{e^{f_{\cP}(x,y)} }{Z_{\cP}(x)} \cdot \cP_X(x)
                        = \frac{e^{p(x)} \cdot e^{f_{\cQ}(x,y)}}{e^{p(x)}Z_{\cQ}(x)} \cdot \cP_X(x)
                        = \frac{1}{Z_{\cQ}(x)} \cdot e^{f_{\cQ}(x,y)} \cdot \cQ_X(x)
                        = \cQ(x,y)
                        \,,
                    \]
                    for all $(x,y)\in\R^d\times\R$, and so $\cP,\cQ$ are the same distribution, and thus have the same mean value over $y$, which is a contradiction.

                \bigskip
                
                \paragraphit{Proof of Polynomial Expectations.}
                    Consider any pair $\cP,\cQ$ in the polynomial expectations family such that $\Ex_{(x,y)\sim\cP}[y]\neq\Ex_{(x,y)\sim\cQ}[y]$.
                    It is immediate that \cref{cond:Uncon} holds if $\cP_{X}\neq\cQ_{X}$.
                    So, we need to show that, if $\cP_{X}=\cQ_{X}$, the following is true
                    \[
                        \nexists S\subseteq\R^d
                        \quadtext{with}
                            \vol(S)\geq c
                        \quad \text{such that} \quad
                        \forall_{(x,y)\in S\times\R}
                            \,,~\cP(x,y)=\cQ(x,y) 
                        \,.
                    \]
                    For the sake of contradiction, assume there exists such a set $S\subseteq\R^d$.
                    Let the polynomial $f_{\cP}$ be such that $\Ex_{(x,y)}[y\mid X{=}x] = f_{\cP}(x)$ and similarly for polynomial $f_{\cQ}$ and $\cQ$.
                    Then, since $\cP(y\mid x) = \sfrac{\cP(x,y)}{\cP_X(x)}$ and similarly for $\cQ$, it is true for every $(x,y)\in S\times\R$
                    \[
                        \Ex_{(x,y)\sim\cP}[y\mid X{=}x] 
                            = \int_y y \cP_{Y\mid X}(y\mid x) \d y
                            = \Ex_{(x,y)\sim\cQ}[y\mid X{=}x]
                        \,.
                    \]
                    So, for all $x\in S$ we have $f_{\cP}(x)=f_{\cQ}(x)$ and since $S$ is infinite, it must be $f_{\cP}(x)=f_{\cQ}(x)$ for all $x\in\R^d$.
                    This is because two polynomials can be either identically equal or agree on a finite number of points, as many as their degree.
                    But then we have
                    \[
                        \Ex_{(x,y)\sim\cP}[y] 
                            = \iint_{(x,y)} y \cP(x,y) \d y\d x
                            = \int_x f_{\cP}(x) \cP_X(x) \d x
                        \,.
                    \]
                    Since $\cP_X=\cQ_X$ and $f_{\cP}=f_{\cQ}$, we get $\Ex_{(x,y)\sim\cP}[y]=\Ex_{(x,y)\sim\cQ}[y]$, which is a contradiction.

            \end{proof}

        \subsection{Proof of \cref{lem:extrapolation} (Classes $\hyD$ Identifiable in Scenario~III)} 
            \label{sec:proofof:lem:extrapolation}
            In this section, we prove \cref{lem:extrapolation}, which we restate below.
            \lemmaExtrapolation*
            \smallskip
            
            \begin{proof}[Proof of \cref{lem:extrapolation}]
                Consider any pair $\cP, \cQ\in\hyDpoly(K,M)$. 
                Let $S\subseteq\R^d$ be such that $\vol(S)>c$ such that $\tv{\cP(S\times\R)}{\cQ(S\times\R)}\leq\eps$ for some $\eps>0$.
                Since $\cP$ and $\cQ$ are supported on $K=[0,1]^{d+1}$, it follows that 
                \[
                    \tv{\cP(S\times[0,1])}{\cQ(S\times[0,1])}
                    =\tv{\cP(S\times\R)}{\cQ(S\times\R)}
                    \leq\eps \,.
                \]
                Consider the following bound:
                \[
                    \abs{
                        \Ex\nolimits_{(x,y)\sim\cP}[y] - \Ex\nolimits_{(x,y)\sim\cQ[y]}
                    } \leq \iint \abs{y} \abs{\cP(x,y)-\cQ(x,y)} \d y\d x
                    \,.
                \]
                Since $(x,y)\in[0,1]^{d+1}$, $\abs{y}\leq 1$ and we can write
                \[
                    \abs{
                        \Ex\nolimits_{(x,y)\sim\cP}[y] - \Ex\nolimits_{(x,y)\sim\cQ}[y]
                    } \leq 
                    \iint \abs{\cP(x,y)-\cQ(x,y)} \d y\d x
                    =
                    2\tv{\cP}{\cQ}
                    \,.
                \]
                Now, it suffices to upper bound $\tv{\cP}{\cQ}$ by the total variation distance between the \textit{truncated} distributions: $\tv{\cP(S\times [0,1])}{\cQ(S\times [0,1])}$ (which is at most $\eps$).
                For this, we use the following  result by \citet{daskalakis2021statistical}.
                \begin{lemma}[Lemma 4.5, \citet{daskalakis2021statistical}]
                    \label{lem:extrapolation:das}
                    Consider any two distributions $\cP,\cQ\in\hyDpoly(K,M)$ such that the logarithms of their probability density functions are proportional to polynomials of degree at most $k$.
                    There exists absolute constant $C>0$ such that for every $T\subseteq[0,1]^{d+1}$ with $\vol(T)>0$ it holds %
                    \[
                        e^{-2M}\vol(S) \leq
                        \frac{\tv{\cP}{\cQ}}{\tv{\cP(T)}{\cQ(T)}} 
                            \leq 8e^{5M} \frac{(2C\min\{d,2k\})^k}{\vol(T)^{k+1}}
                        \,. 
                    \]
                \end{lemma}
                Substituting the bound from \cref{lem:extrapolation:das} for $T=S\times [0,1]$ implies that 
                \[
                    \abs{
                        \Ex\nolimits_{(x,y)\sim\cP}[y] - \Ex\nolimits_{(x,y)\sim\cQ}[y]
                    } \leq 
                    2\tv{\cP}{\cQ}
                    \leq 
                    16e^{5M} \frac{(2C\min\{d,2k\})^k}{\vol(T)^{k+1}}\cdot \eps
                    \,.
                \]
                Since $\vol(T)=\vol(S) > c$, the desired result follows.
            \end{proof}

\addtocontents{toc}{\protect\setcounter{tocdepth}{1}}

\section{Need for Distributional Assumptions}
    \label{sec:unconfoundednessOverlap}
    This section presents some reasons why unconfoundedness and overlap {cannot} be weakened without restricting $\hyD$.  
    {We believe that these results are well-known but since we could not find an explicit reference, we include the results and proofs for completeness.}

    \subsection{Need for Distributional Assumptions to Relax Overlap} %
    Let us assume that we make no distributional assumptions, \ie{}, $\hyD = \hyDall.$
    We show that if $\hyP$ does not satisfy overlap even at two points then the pair $(\hyP, \hyDall)$ fails to satisfy \cref{cond:iden}.
        \begin{proposition}
            [Impossibility of Point Identification without Distributional Assumptions]
            \label{lem:impossibility:overlap}
            Fix any class $\hyP$ that violates overlap at two points in the following sense: there exists a generalized propensity score $p\in \hyP$, a covariate $x\in \R^d$, and distinct values $y_1,y_2\in \R$ such that $p(x,y_1)=p(x,y_2)=0$.
            Then, the pair $\inparen{\hyP,\hyDall}$ does not satisfy \cref{cond:iden}.
        \end{proposition}
        Hence, \cref{infthm:Suff} implies that, for any $\hyP$ that violates overlap in the above sense, there are observational studies $\cD$ realizable with respect to $\inparen{\hyP,\hyDall}$ where $\tau_\cD$ cannot be identified. 
        
        \smallskip
        
        \begin{proof}[Proof of \cref{lem:impossibility:overlap}]
            Fix any concept class $\hyP$ satisfying the condition described.
            Due to this condition, there exists $p\in \cP$, $x\in \R^d$, and distinct $y_1,y_2\in\R$ with $p(x,y_1)=p(x,y_2)=0$.
            Consider any two distributions $\cP,\cQ\in \hyDall$ that satisfy the following conditions:
            \begin{enumerate}[itemsep=0pt]
                \item They have the same marginal on $X$, \ie{}, $\cP_X=\cQ_X$;  
                \item The densities satisfy: $\cP(x,y_1) < \cQ(x,y_1)$ and $\cP(x,y_2) > \cQ(x,y_2)$.
                \item For each $(x',y')\not\in S$, $\cP(x',y')=\cQ(x',y')$ where $S\coloneqq \inbrace{(x,y_1),(x,y_2)}.$
            \end{enumerate} 
            We claim that the tuples $(p,\cP)$ and $(p,\cQ)$ witness that $\inparen{\hyP,\hyDall}$ violate \cref{cond:iden}.
            To see this, fix any $(x',y')$ and from the following cases observe that regardless of the choice of $(x',y')$, $p(x,y)\cP(x,y)=p(x,y)\cQ(x,y)$.

            \paragraph{Case A ($(x',y') \in S$):}
                In this case, $p(x,y)=0$ and, hence, $p(x,y)\cP(x,y)=p(x,y)\cQ(x,y)$.

            \paragraph{Case B ($(x',y') \not\in S$):}
                Since we have that $\cP(x',y')=\cQ(x',y')$ for $(x',y')\not\in S$, it holds that $p(x,y)\cP(x,y)=p(x,y)\cQ(x,y)$.
        \end{proof}

    \subsection{Unconfoundedness and Overlap are Maximal for Distribution-Free Identification}
        Next, we show that unconfoundedness and overlap are \emph{maximal} when $\hyD = \hyDall$: we show that if one extends the class $\hyP$ to be a strict superset of $\hyPou$, then $(\hyP, \hyDall)$ cannot satisfy \cref{cond:iden}.
        
        \begin{proposition}[Impossiblity of Identification without Distributional Assumptions]\label{lem:impossibility:unconfoundedness}
            For any class $\hyP \supsetneq \hyPou(0)$ that satisfies overlap (\ie{}, for each $p\in \hyP$ and $(x,y)\in \R^d\times \R$, $p(x,y)\in (0,1)$), the tuple $\inparen{\hyP, \hyDall}$ does not satisfy \cref{cond:iden}.
        \end{proposition}
        Hence, \cref{infthm:Suff} implies that, for any $\hyP$ satisfying the condition in \cref{lem:impossibility:unconfoundedness}, there is $\cD$ realizable with respect to $\inparen{\hyP, \hyDall}$ where $\tau_\cD$ cannot be identified.
        \smallskip
        \begin{proof}[Proof of \cref{lem:impossibility:unconfoundedness}]
            Fix any concept class $\hyP\supsetneq \hyPou(c)$.
            By definition, $\hyP$ contains $\bar{p}(\cdot)$ with the following property:
                for some $x^\star\in \R^d$ and $y_1\neq y_2$, 
                \[
                    \bar{p}(x^\star,y_1)\neq \bar{p}(x^\star,y_2) \quadtext{with}
                    \bar{p}(x^\star,y_1),\bar{p}(x^\star,y_2)\in (0,1)\,.
                \]
            The second requirement holds since $\hyP$ satisfies overlap.
            Our goal is to show that the pair $\inparen{\hyP,\hyDall}$ does not satisfy \cref{cond:iden}.
            Recall that to show this it suffices to find distinct tuples $(p,\cP)$ and $(q,\cQ)$ such that $\cP_X=\cQ_X$ and for each $x\in \supp(\cP_X)$ and $y\in \R$
            \[
                p(x,y)\cP(x,y)=q(x,y)\cQ(x,y)\,.
                \yesnum\label{eq:impossibility:requirement}
            \]
            Fix $p(\cdot)=\bar{p}$. %
            Next, for each $x$, we iteratively construct the function $q(\cdot)\in \hyPou \supsetneq \hyP$ and distributions $\cP(x,y), \cQ(x,y)$ to satisfy \cref{eq:impossibility:requirement} and $\cP_X=\cQ_X$.
            For each $x\in \R^d$, we consider the following cases.

            \paragraph{Case A ($\forall_{y_1,y_2\in \R}$,~~~ $\bar{p}(x,y_1)=\bar{p}(x,y_2)$):}
                In this case, for each $y\in \R$, we set $\cP(x,y)=\cQ(x,y)=0$ and set $q(x,y)=\alpha$ for an arbitrary constant $\alpha\in (0,1)$ independent of $y$ (which ensures that $q$ can be an element of $\hyPou$).
                Therefore, $\cP_X(x)=\cQ_X(x)=0$ and \cref{eq:impossibility:requirement} is satisfied.

            \paragraph{Case B ($\exists_{y_1,y_2\in \R}$,~~~ $\bar{p}(x,y_1)\neq \bar{p}(x,y_2)$):}
                We set $q(x,y)=p(x,y_2)\in (0,1)$ for each $y\in \R$.
                (Since $q(x,y)$ is independent of $y$, $q$ can be an element of $\hyPou$.)
                We also set 
                \[
                    \forall_{y\neq y_2},~~\cP(y|x)=\cQ(y\mid x)=0\,,
                    \quad 
                    \cP(y_2|x)=\cQ(y_2\mid x)=1\,,
                    \quadand 
                    \cP_X(x)=\cQ(x)>0\,.
                \]
                Now, by construction $\cP_X(x)=\cQ_X(x)$ and \cref{eq:impossibility:requirement} holds.

            \medskip 
            
            \noindent Observe that in both cases, the function $q(\cdot)$ satisfies the requirements of $\hyPou$ and, hence, $q\in \hyPou\supsetneq \hyP$.
            Further, since in $\cP_X(x)=\cQ(x)$ and \cref{eq:impossibility:requirement} holds in both cases, we have proved that the \cref{cond:iden} is violated for the pair of tuples $\inparen{p,\cP}$ and $\inparen{q,\cQ}$.
        \end{proof}
 
\section{Identifiability of the Heterogeneous Treatment Effect}
    \label{sec:iden:extension}
    In this section, we study identification of the heterogeneous treatment effect: given an observational study $\cD$, the heterogeneous treatment effect for covariate $x\in \R^d$ is defined as 
    \[
        \tau_\cD(x) \coloneqq \Ex\nolimits_\cD\insquare{Y(1)-Y(0)\mid X{=}x}\,.
    \]
    By identification of the heterogeneous treatment effect, we mean identification of the function $\tau_\cD(\cdot)$.
    {We show that the} following variant of \cref{cond:iden}, characterizes the identification of heterogeneous treatment effects.
    \begin{restatable}[{Identifiability} Condition for HTE]{condition}{conditionIdenHTE}
        \label{cond:iden:HTE}
        The concept classes $\inparen{\hyP,\hyD}$ satisfy the Identifiability Condition if for any compatible, distinct tuples $(p,\cP), (q,\cQ) \in {\hyP} \times {\hyD}$, at least one of the following holds:
    
        \begin{enumerate}[itemsep=-1pt,leftmargin=17.5pt]
            \item \textbf{(Equivalence Outcome Distributions)} $\cP=\cQ$
            \item \textbf{(Distinction of Covariate Marginals)} $\cP_X \neq \cQ_X$ 
            \item \textbf{(Distinction under Censoring)} $\exists (x,y)\in \supp(\cP_X)\times \R$, such that, $p(x,y) \cP(x, y) \neq q(x,y) \cQ(x,y)$
        \end{enumerate}
    \end{restatable}
    The above condition is sufficient to identify the heterogeneous treatment effect. 
    The reason is similar to why \cref{cond:iden} is sufficient to identify ATE: consider two observational studies $\cD_1$ and $\cD_2$ which correspond to the pairs $(p, \cP)$ and $(q,\cQ)$ respectively, where $\cP$ and $\cQ$ are ``guesses'' for the distributions of, say, $(X,Y(1)).$ 
    Assume that the true observational study $\cD$ is either $\cD_1$ or $\cD_2$.
    Then, one can identify the correct observational study $\cD_1$ or $\cD_2$ with samples from the censored distribution $\cC_\cD$ and, as a consequence, one can identify the correct HTE from among $\tau_{\cD_1}(\cdot)$ and $\tau_{\cD_2}(\cdot)$.
    As for necessity, for any observational study $\cD$ realizable with respect to $\hyP$ and $\hyD$, \cref{cond:iden:HTE} is necessary for identifying HTE.
    The proofs of sufficiency and necessity are nearly identical to the proof of \cref{infthm:Suff} and are omitted.
    We summarize the results for HTE's identifiability below.
    \begin{theorem}
        For any concept classes $(\hyP, \hyD)$, the following are true:
        \begin{itemize}[itemsep=-1pt,leftmargin=17.5pt]
            \item[$\triangleright$] \textbf{(Sufficiency)}\quad If $(\hyP, \hyD)$ satisfy \cref{cond:iden:HTE}, then the heterogeneous treatment effect $\tau_\cD(\cdot)$ is identifiable from the censored distribution $\cC_\cD$ for any observational study $\cD$ realizable with respect to $(\hyP, \hyD)$.
            \item[$\triangleright$] \textbf{(Necessity)}\quad If the heterogeneous treatment effect $\tau_\cD(\cdot)$ is identifiable from the censored distribution $\cC_\cD$ for any observational study $\cD$ realizable with respect to $(\hyP, \hyD)$, then $(\hyP, \hyD)$ satisfy \cref{cond:iden:HTE}.
        \end{itemize}
    \end{theorem}
    
    \vspace{-5mm}
    
\section{Estimation of Nuisance Parameters from Finite Samples}
\label{appendix:probconcepts}
    \label{sec:nuisanceParameters}
    As it is standard in Causal Inference \cite{foster2023orthognalSL}, our estimators for treatment effects use certain \textit{nuisance parameters}, such as the generalized propensity scores and the outcome distributions, and then use these nuisance parameters to deduce the treatment effects of interest. 
    In this section, we prove that estimators of these nuisance parameters can be implemented under standard assumptions.
    In this section, we implement the following two nuisance parameter oracles.
    \begin{restatable}    
        [Propensity Score Estimation Oracle]{definition}{defPropensityScoreEstimationOracle}
        \label{def:propOracle}
        The propensity score estimation oracle for class $\hyE\subseteq\inbrace{e\mid e\colon \R^d\to [0,1]}$ is a primitive that, given accuracy parameter $\eps>0$, a confidence parameter $\delta>0$, and $N_P(\eps,\delta)$ independent samples from the censored distribution $\cC_\cD$ for some $\cD$ realizable with respect to $\hyE$, outputs an estimate of the propensity score $\wh{e}\colon \R^d \to [0,1]$, such that, with probability $1-\delta$,
                \[
                    \Ex\nolimits_{x\sim \cD_X}{
                        \abs{
                            \Pr\nolimits_\cD\insquare{T{=}1\mid X{=}x}
                            -
                            \wh{e}(x)
                        }
                    }
                    \leq \eps\,.
                \] 
    \end{restatable}
    \vspace{-5mm}
    \begin{restatable}    
        [$L_1$-Approximation Oracle]{definition}{defDensityEstimationOracle}
        \label{def:densityOracle}
        The $L_1$-approximation oracle for class $\hyP\times\hyD$ is a primitive that, given accuracy parameter $\eps>0$, a confidence parameter $\delta>0$, and $N_D(\eps,\delta)$ independent samples from the censored distribution $\cC_\cD$, 
            outputs generalized propensity scores $p,q \colon \R^d\times \R\to[0,1]$ and 
            distributions $\cP,\cQ$  
        such that, with probability $1-\delta$,
        \begin{align*}
            \norm{p_1\cD_{X,Y(0)} - p\cP}_1\leq \eps\,,\quad
            \norm{p_0\cD_{X,Y(1)} - q\cQ}_1\leq \eps\,,
        \end{align*}
        where we define the $L_1$-norm between $\alpha(x,y)$ and $\beta(x,y)$ as $\|\alpha - \beta\|_1 \coloneqq \iint \bigl|\alpha(x,y) - \beta(x,y)\bigr|\d x\d y$.
    \end{restatable}
    A few remarks are in order.
    First, as a sanity check, one can verify that all the quantities being estimated by the above oracles are identifiable from the censored distribution $\cC_\cD$.
    Second, while in the definition of the oracles, we measure the error in the $L_1$-norm one can change to the $L_2$-norm without affecting the results.
        The above strategy, based on estimating nuisance parameters, may not always be optimal. 
        For instance, for specific concept classes $\hyP$ and $\hyD$, one may be able to learn $\tau$ without estimating the nuisance parameters, resulting in significantly better sample complexity.
        We focus on the above strategy because it is simple and already widely used \cite{foster2023orthognalSL}, but obtaining better sample complexities for specific examples is an important direction for future work.
    \subsection{Implementing the Propensity Score Oracle}
    In this section, we construct the propensity score estimation oracle (\cref{def:propOracle}). 
    The task of estimating propensity scores turns out to be equivalent to the problem of learning probabilistic concepts (henceforth, $p$-concepts) introduced by \citet{kearns1994pconcept}, and we use the results on learning $p$-concepts by \citet*{63453,kearns1994pconcept,alon1997scale} to implement the propensity score oracle and bound its sample complexity.
    The following condition characterizes when $p$-concepts are learnable and hence will also characterize when propensity scores can be estimated.
    \begin{definition}[Fat-shattering dimension]\label{def:fat-shattering-dimension}
        Let $\hyE \subseteq \{e \colon \R^d \to [0,1]\}$ be a hypothesis class and let $S = \{s_1, \ldots, s_m\}$ be a set of points in $\R^d$.  
        We say $S$ is $\gamma$-shattered by $\hyE$ if there exists a threshold vector $t = \{t_1, \ldots, t_m\}\in\R^m$ such that for any binary vector $b = \{b_1, \ldots, b_m\} \in \{\pm1\}^m$, there exists a function $e_b \in \hyE$ satisfying  
        \[
        b_i(e_b(s_i) - t_i) \geq \gamma \quad \text{for all } i \in [m].
        \]  
        The fat-shattering dimension of $\hyE$ at scale $\gamma$, denoted $\fatShatDim(\hyE)$, is the maximum cardinality of a set $S$ that is $\gamma$-shattered by $\hyE$.
        \end{definition}
    If the fat-shattering dimension of $\hyE$ is finite, then we get the following result.
    \begin{theorem}[Propensity score estimation]\label{thm:prop-scores-estimation} 
        Let $\hyE$ be a concept class of propensity scores with fat-shattering dimension $\fatShatDim(\hyE)<\infty$ at all scales $\gamma>0$.
        Then, there exists a propensity score estimation oracle for $\hyE$ with sample complexity (for any $\eps,\delta\in (0,1)$)
        \[
            N_P(\eps,\delta)=O\inparen{ \frac{1}{\eps^2} \cdot \inparen{ \mathrm{fat}_{\sfrac{\eps}{256}}(\hyE) \log\inparen{\nfrac{1}{\eps}} + \log\inparen{\nfrac{1}{\delta}}}}\,.
        \]  
        \end{theorem}
    \begin{proof}[Proof of \cref{thm:prop-scores-estimation}]
        First, we introduce the probabilistic concept learning framework and argue that propensity score estimation is a probabilistic concept learning problem.
        Then the result follows by the main result of \citet{kearns1994pconcept}. Let us define the problem of learning probabilistic concepts.
        Consider a concept class $\hyH\subseteq\{h\colon\R^d\to[0,1]\}$ and a function $h\in\hyH$ which we call the $p$-concept.
        Then we get a sample $X\in\R^d$ from a distribution $\cF$, \ie{}, $X\sim\cF$ and assign $X$ label $1$ with probability $h(X)$, otherwise, we give label $0$.
        Our goal is to use the samples $X$ along with their $\{0,1\}$ labels to estimate $h$.
        That is, we want an algorithm to find a concept $\hat{h}\colon\R^d\to[0,1]$ such that $\Ex_{x\sim\cF}\abs{\hat{h}(x)-h(x)}\leq\eps$.
        Notice that this is the exact problem of estimating the propensity score $e(x)=\Pr[T{=}1\mid X{=}x]$ from the samples $X$ and the labels $T\in\{0,1\}$.
        In terms of sample complexity for learning probabilistic concepts, the work of \cite{kearns1994pconcept} gave an upper bound based on the Pollard dimension (a.k.a.\ pseudo-dimension) and a lower bound based on the (scaled) fat-shattering dimension, \ie{}, a scaled version of the pseudo-dimension. Next, \citet{alon1997scale} provided a tight dimension for $p$-concept learnability based on the fat-shattering dimension, establishing an equivalence between this dimension and uniform convergence. \footnote{{The sample complexity in \cref{thm:concept-estimation} follows from Theorem 3.6 of \cite{alon1997scale} (see also discussion below Theorem 3.6) along with the fact that $\eps$-uniform convergence implies learnability in the $p$-concept model with accuracy $3\eps$ \cite[Lemma 4.3]{alon1997scale}. The constant $256$ is not optimized and can probably be improved.}}
        \begin{theorem} 
        [Theorem 4.2 in \cite{alon1997scale}]
        \label{thm:concept-estimation}
            Fix any $\eps,\delta\in (0,1)$.
            Consider the function class $\hyH = \{ h \colon \R^{\ell}\to[0,1]\}$ with finite fat-shattering dimension $\fatShatDim(\hyE)<\infty$ for all scales $\gamma>0$.
            Let $\cF$ be a probability distribution over $\R^{\ell}$.
            Then, there exists an algorithm that, for 
            \[
                m=O\inparen{\frac{1}{\eps^2}\cdot \inparen{\mathrm{fat}_{\sfrac{\eps}{256}}(\hyH) \log(\nfrac{1}{\eps})+\log(\nfrac{1}{\delta})}}\,,
            \]
            given a sample set $S=\inbrace{(X_i, Y_i)}_{i=1}^m$ of \iid{} samples $X_i\sim\cF$ and $Y_i\sim\Be(h(X_i))$, returns a function $\hat{h}\in\hyH$ such that with probability $1-\delta$ it holds
            $
                \Ex\nolimits_{x\sim\cF}\sabs{h(x)-\hat{h}(x)} < \eps
                \,.
            $\qedhere
        \end{theorem} 
    \end{proof}

    \subsection{Implementing the $L_1$-Approximation Oracle}
    In this section, we construct the $L_1$-approximation oracle (\cref{def:densityOracle}).
    We require some standard assumptions on the classes $\hyP$ and $\hyD$ to implement the $L_1$-approximation oracle.
    Concretely, we require (1) a bound on $\hyD$'s covering number with respect to the TV distance (\cref{def:covers}), (2) a bound on the smoothness of the distributions in $\hyD$ with respect to some measure $\mu$ (\cref{def:smooth-distribution}), and (3) a bound on the fat-shattering dimension of $\hyP$ (\cref{def:fat-shattering-dimension}).
    These assumptions enable us to ``cover'' the class $\hyP\times\hyD$ in $L_1$-norm.
    \mbox{We begin by formally defining a cover.}
    \begin{definition}
    [Covers and Covering Numbers]
    \label{def:covers}
        Consider the concept class $\hyH \subseteq \{ h \colon \R^{\ell}\to[0,1]\}$ with a metric $d(\cdot,\cdot)$.
        Then the function class $\hyH_{\eps}$ is a $\eps$-cover of $\hyH$, if, for every function $h\in\hyH$, there is a function $\bar{h}\in\hyH_{\eps}$ such that $d(h, \bar{h})\leq\eps$.
        The size of the smallest cover $\hyH_{\eps}$ for $\hyH$ is called the covering number of $\hyH$ and is denoted by $N(\hyH,d,\eps)$.
    \end{definition}
    Having a cover of the class $\hyP\times \hyD$ is useful because, roughly speaking, given a cover, standard results in statistical estimation enable us to identify the element of the cover closest to the true concept with finite samples.
    \begin{theorem}
        [\citet{yatracos1985rates}]
        \label{thm:distrLearningFiniteSet}
            There is a deterministic algorithm that, 
            given candidate distributions $f_1,f_2,\ldots,f_M$, 
            a parameter $\zeta>0$, and 
            $\lceil\log(3M^2/\delta)/2\zeta^2\rceil$ samples from an unknown distribution $g$, 
            outputs index $j\in[M]$ such that
            $
                \norm{f_j - g}_1 \leq 
                    3\min_{i\in[M]}\norm{f_i-g}_1 + 4\zeta
                \,,
            $
            with probability at least $1-\nfrac{\delta}{3}$.
        \end{theorem}
    Note that the above theorem holds for covers over \textit{distributions}.
    However, elements of $\hyP\times \hyD$ may not be distributions.
    Nevertheless, the above theorem is still sufficient for us because elements of $\hyP\times \hyD$ that interest us are distributions up to a {normalizing} factor of $\Pr[T{=}1]$ or $\Pr[T{=}0]$ (the choice depends on the specific element).
    That is, the true distribution of samples $(X,Y(t),t)$ is an element in the class $\hyP\times\hyD$, normalized by $\Pr[T{=}t]$, for $t\in\{0,1\}$.
    So the elements that we are interested in, should also satisfy this condition, and thus, define a probability distribution class.

    In the remainder of this section, we present the assumptions on $\hyP$ and $\hyD$ and then use these assumptions to bound the size of the resulting cover.

    \paragraph{Assumption 1 (Covering Number of $\hyD$).}
    We directly impose such an assumption over $\hyD$.
    For a hypothesis class, it is well known that the notion of fat-shattering defined in \cref{def:fat-shattering-dimension} implies the existence of a cover over the class in the following sense.
    \begin{lemma}
    [\citet{rudelson2006combinatorics}]%
    \label{lem:hypothesis-covering-bound}
        Fix $\eps>0$, $R>0$ and let $\mu$ be a probability density function over $\R^{\ell}$.
        Consider a concept class $\hyP\subseteq \{p\colon \R^{\ell}\to[0,1]\}$ with finite fat-shattering dimension such that $\Ex_{x\sim\mu}[|p(x)|^4]\leq R$ for all $p\in\hyP$.
        Then it holds
        \[
            \log(N(\hyP, L_1(\mu), \eps)) \leq 4C \fatShatDim(\hyP) \log(\frac{R}{c\eps})
            \,,
        \]
        where $\gamma=c\eps$, $C, c$ are universal constants, and the metric $L_1(\mu)$ is $\Ex_{x\sim\mu}\abs{p(x)-q(x)}$ for any $p,q\in\hyP$.
    \end{lemma}
    Observe that this cover is with respect to the expected $L_1$-norm given a probability density function $\mu$.
    Thus, in our case, such a cover is not directly useful.
    Ideally, we would like this cover to be with respect to the distribution in $\hyD$ which is the underlying distribution for our problem $\cC_{\cD}$.
    However, we cannot know this distribution and cannot estimate from samples as argued before.
    
    \paragraph{Assumption 2 (Smoothness of $\hyD$).} This is where our next assumption comes in: it will make the distributions in the class ``comparable'' to another measure $\mu$, that we have access to.
    \begin{definition}
    [Smooth Distribution]
    \label{def:smooth-distribution}
        Consider any probability density function $\mu$ over $\R^{\ell}$.
        We say that a probability density function $p$ over $\R^{\ell}$ is $\sigma$-smooth with respect to $\mu$ if $p(x) \leq \inparen{\nfrac{1}{\sigma}} \mu(x)$ for all $x\in\R^{\ell}$.
    \end{definition}

    \paragraph{Assumption 3 (Fat-shattering dimension of $\hyP$).}
        Our final assumption is a bound on the fat-shattering dimension of $\hyP$. 
        We have already discussed the fat-shattering dimension in the previous section (see \cref{def:fat-shattering-dimension}). 
        It is useful for us because it turns out that a bound on the fat-shattering dimension also implies a bound on the covering number in $L_1$ norm.
        
    \paragraph{$L_1$-approximation oracle.}
    We are now ready to construct the cover of $\hyP\times\hyD$, which immediately gives us the $L_1$-approximation oracle.
    \begin{theorem}[Sample Complexity for $L_1$-approximation]
    \label{thm:oracleConstruction}
        Fix any $\eps\in(0,1)$, $\sigma\in(0,1], \eta\in(0,\nfrac{1}{2}]$, with $\eta>\eps$, and a distribution $\mu$ over $\R^d\times\R$.
        Let the concept classes $\hyP$ and $\hyD$ satisfy:
        \begin{enumerate}[itemsep=0pt]
            \item Each $\cP\in \hyD$ is $\sigma$-smooth with respect to the distribution $\mu$.
            \item 
            $\hyP$ has a finite fat-shattering dimension $\mathrm{fat}_{(\eta\sigma\eps)/16}(\hyP)<\infty$ at scale $\nfrac{\eta\sigma\eps}{16}$; %
            \item $\hyD$ has a finite covering number with respect to TV distance $N=N(\hyD,d_{\mathsf{TV}},\nfrac{\eta\eps}{32})<\infty$.
        \end{enumerate}
        Consider any  $\cD$ is realizable by $(\hyP,\hyD)$ and satisfying $\Pr[T{=}1]\in(\eta, 1-\eta)$.
        Then, there exists an algorithm that implements an $L_1$-approximation oracle of accuracy $\eps$ and confidence parameter $\delta$ for $\hyP\times\hyD$ using $N_D(\eps,\delta)$ samples from $\cC_{\cD}$ where
        \[
            N_D(\eps,\delta) = O\inparen{\frac{1}{\eps^2}\inparen{\mathrm{fat}_{(\eta\sigma\eps)/16}(\hyP)\cdot \log(\nfrac{1}{\eta\sigma\eps})+\log(\nfrac{N}{\delta})}}\,.
        \] 
    \end{theorem}
    \begin{proof}[Proof of \cref{thm:oracleConstruction}]
        First, we construct a $\nfrac{\eta\eps}{8}$-cover in $L_1$ norm for $\hyP\times\hyD$ and then show that we find a good estimation of the true product $p\cdot\cP$ from samples using this cover. 
    
        \paragraphit{Cover of $\hyP\times\hyD$.}
        We show that the product space $\hyP\times\hyD$ has a $(\nfrac{\eta\eps}{8})$-cover in $L_1$ norm.
        Note that, by \cref{lem:hypothesis-covering-bound}, $\hyP$ accepts a cover in $L_1(\mu)$-norm of size $N'$ such that $\log(N')\leq 4C \mathrm{fat}_{\sfrac{\eta\sigma\eps}{16}} \log(\nfrac{1}{c\eta\sigma\eps})$ for universal constants $c, C >0$, 
        since it has finite fat-shattering dimension $d_{\mathrm{fat}}$ and range $[0,1]$, \ie{}, $\Ex_{(x,y)\sim\mu}[|p(x)|^4]\leq 1$, for all $p\in\cP$.
        Let $\hyP_{\nfrac{\eta\sigma\eps}{16}}$ be the cover of $\hyP$ with respect to $L_1(\mu)$-norm and $\hyD_{\nfrac{\eta\eps}{32}}$ be the total variation distance cover of $\hyD$ (\cref{def:covers}). 
        Then, we show that the product $\hyP_{\nfrac{\eta\sigma\eps}{16}}\times\hyD_{\nfrac{\eta\eps}{32}}$ is an $\nfrac{\eta\eps}{8}$-cover of $\hyP\times\hyD$, 
        \ie{}, for any $p\cP\in\hyP\times\hyD$, there exists a $\bar{p}\bar{\cP}\in\hyP_{\nfrac{\eta\sigma\eps}{16}}\times\hyD_{\nfrac{\eta\eps}{32}}$
        such that $\norm{p\cP - \bar{p}\bar{\cP}}_1 \leq\nfrac{\eta\eps}{8}$.
        Let $\bar{p}\in\bar{\hyP}_{\nfrac{\eta\sigma\eps}{16}}$ be such that $\Ex_{(x,y)\sim\mu}[|p(x)-\bar{p}(x)|]\leq\nfrac{\eta\sigma\eps}{16}$ and,
        $\bar{\cP}\in\hyD_{\nfrac{\eta\eps}{32}}$ such that $\tv{\cP}{\bar{\cP}}\leq\nfrac{\eta\eps}{32}$ (we know these exist by definition of the cover).
        Then we can bound the desired quantity using triangle inequality as follows
        \[
            \norm{p\cP - \bar{p}\bar{\cP}}_1 
                = \norm{p\cP - \bar{p}\cP + \bar{p}\cP - \bar{p}\bar{\cP}}_1
                \leq \norm{p\cP - \bar{p}\cP}_1 + \norm{\bar{p}\cP - \bar{p}\bar{\cP}}_1
            \,.
        \]
        {Where} $\norm{p\cP-\bar{p}\cP}_1 = \iint  |p(x,y)-\bar{p}(x,y)|\cP(x,y)\d y\d x = \Ex_{(x,y)\sim\cP}[|p(x,y)-\bar{p}(x,y)|]$.
        Also, $\cP$ is $\sigma$-smooth by assumption. So the previous expression implies 
        \[
            \norm{p\cP-\bar{p}\cP}_1
                \leq \frac{1}{\sigma} \Ex_{(x,y)\sim\mu}[|p(x,y)-\bar{p}(x,y)|]
            \,,
        \]
        which is at most $\eta\eps/16$ by cover's construction.
        Finally, $\snorm{\bar{p}\cP-\bar{p}\bar{\cP}}_1 = \bar{p}\snorm{\cP-\bar{\cP}}_1\leq\snorm{\cP-\bar{\cP}}_1$, since $\bar{p}\in[0,1]$.
        Also, $\tv{\cF}{\cQ} = (\nfrac{1}{2})\snorm{\cF-\cQ}_1$, for any \mbox{two distributions $\cF, \cQ$, and so as required:}
        \[
            \norm{p\cP - \bar{p}\bar{\cP}}_1 
                \leq \frac{\eta\eps}{16} + 2\cdot \frac{\eta\eps}{32} = \frac{\eta\eps}{8}
            \,.
        \]
        
        \paragraphit{Estimation of True $p\cP$.}
        We want to use \cref{thm:distrLearningFiniteSet} to get an estimate for $p\cP$.
        However, the samples $(X,Y(T),T)$ we get are not distributed according to $p\cP$, but rather $\nfrac{p\cP}{Z(p\cP)}$, where $Z(p\cP) = \iint p(x,y)\cP(x,y)\d y\d x$, and the candidate concepts we have are not probability distributions.
        However, we can turn them into probability distributions by normalizing them.
        Notice that, for any $p\cP\in\hyP\times\hyD$ and its closest element in the cover $\hat{p}\hat{\cP}\in\hyP_{\nfrac{\sigma\eps}{16}}\times\hyD_{\nfrac{\eps}{32}}$ it holds
        \[
            \abs{Z(p\cP)-Z(\hat{p}\hat{\cP})} =
            \abs{\iint p(x,y)\cP(x,y) \d y\d x - \iint \hat{p}(x,y)\hat{\cP}(x,y) \d y\d x}
                \leq \norm{p\cP - \hat{p}\hat{\cP}}_1
                \leq \frac{\eta\eps}{8}
            \,.
        \]
        So the normalization factors will be close.
        The only issue that remains to be taken care of is the possibility of dividing with a very small number, close to zero.
        We know that $\cC_{\cD}$ is such that the $\Pr[T{=}1]\in(\eta,1-\eta)$, so we can ignore any elements in the cover whose normalization constant is smaller than $\eta-\eps$.
        Then, for every $p\cP\in\hyP\times\hyD$ and its closest element in the cover $\hat{p}\hat{\cP}\in\hyP_{\nfrac{\sigma\eps}{16}}\times\hyD_{\nfrac{\eps}{32}}$, it holds
        \[
            \norm{ \frac{p\cP}{Z(p\cP)} - \frac{\hat{p}\hat{\cP}}{Z(\hat{p}\hat{\cP})} }_1 
                \leq \frac{1}{Z(p\cP)} \norm{p\cP-\hat{p}\hat{\cP}}_1 
                + \frac{1}{Z(p\cP)} \abs{Z(\hat{p}\cP)-Z(p\hat{\cP})}
                \leq \frac{\eps}{8}
            \,.
        \]        
        Now we can use \cref{thm:distrLearningFiniteSet} that, given samples from a distribution $g$ determines the best approximation for it among a finite set of candidate distributions.
        In our case, we know that $g$ belongs to the class.
        Moreover, let the distributions $f_1,\ldots,f_M$ be the distributions on the $(\nfrac{\eps}{8})$-cover of $\hyP\times\hyD$ normalized, and so, $M \leq N \inparen{\frac{1}{c\sigma\eps}}^{4CD}$.
        For $\zeta = \nfrac{\eps}{8}$, we can use the above algorithm to implement the $L_1$ approximation oracle of accuracy $\eps$ and success probability $1-\delta$ using $O((D\log(\nfrac{1}{\sigma\eps})+\log(\nfrac{N}{\delta}))/\eps^2)$ samples.
    \end{proof}
    
\end{document}